\documentclass[11pt]{article}

\usepackage{amssymb, amsfonts, amsxtra, amsmath}
\usepackage{mathrsfs}
\usepackage{amsthm}
\usepackage{latexsym}
\usepackage[all]{xy}
\usepackage{graphics}
\usepackage{makeidx}
\usepackage{rotating}
\usepackage{bbold}
\usepackage{syntonly}
\usepackage{nomencl}
\usepackage{pdfpages}

\usepackage[normalem]{ulem}

\usepackage{hyperref}

\usepackage{parskip}
\makeatletter % need this to avoid the conflict between amsthm and parskip
\def\thm@space@setup{%
 \thm@preskip=\parskip \thm@postskip=0pt
}
\def\th@remark{%
  \thm@headfont{\itshape}%
  \normalfont % body font
  \thm@preskip\parskip \thm@postskip=0pt
}
\makeatother

\newtheorem{Theorem}{Theorem}[section]
\newtheorem{Def}[Theorem]{Definition}
\newtheorem{Lem}[Theorem]{Lemma}
\newtheorem{Prop}[Theorem]{Proposition}
\newtheorem{Cor}[Theorem]{Corollary}

\newtheorem{Rem}[Theorem]{Remark}

\DeclareMathOperator{\Pol}{\mathrm{Pol}}

\DeclareMathOperator{\Func}{\mathrm{Fun}}
\DeclareMathOperator{\Fun}{\mathscr{F}}

\DeclareMathOperator{\Tr}{\mathrm{Tr}}
\DeclareMathOperator{\bidual}{\mathrm{bd}}

\newcommand{\C}{\mathbb{C}}
\newcommand{\R}{\mathbb{R}}
\newcommand{\N}{\mathbb{N}}
\newcommand{\Z}{\mathbb{Z}}

\newcommand{\Hsp}{\mathcal{H}}

\newcommand{\rd}{\mathrm{d}}

\newcommand{\evv}{\mathrm{ev}}

\newcommand{\Ww}{\mathcal{W}}
\newcommand{\Vv}{\mathcal{V}}

\newcommand{\mfsu}{\mathfrak{su}}
\newcommand{\mfg}{\mathfrak{g}}

\newcommand{\mfk}{\mathfrak{k}}
\newcommand{\mfh}{\mathfrak{h}}
\newcommand{\mfb}{\mathfrak{b}}

\newcommand{\mft}{\mathfrak{t}}
\newcommand{\mfa}{\mathfrak{a}}

\newcommand{\mfn}{\mathfrak{n}}

\newcommand{\mbK}{\mathbf{K}}
\newcommand{\mbL}{\mathbf{L}}
\newcommand{\mbE}{\mathbf{E}}
\newcommand{\mbF}{\mathbf{F}}

\newcommand{\mbX}{\mathbf{X}}

\newcommand{\msR}{\mathscr{R}}

\newcommand{\acsigma}{\sigma\!\!\!^{_{\textrm{ }\!\!\!\;_\mathbb{\shortmid}}}\,}

\DeclareMathOperator{\compact}{\mathrm{c}}
\DeclareMathOperator{\Res}{\mathrm{Res}}

\newcommand{\blhd}{\blacktriangleleft}

\newcommand{\can}{\mathrm{can}}

\newcommand{\opp}{\mathrm{op}}
\newcommand{\cop}{\mathrm{cop}}
\newcommand{\End}{\mathrm{End}}

\newcommand{\Ad}{\mathrm{Ad}}

\newcommand{\id}{\mathrm{id}}

\newcommand{\loc}{\mathrm{ext}}
\newcommand{\wt}{\mathrm{wt}}
\newcommand{\rwt}{\mathrm{rwt}}
\newcommand{\lwt}{\mathrm{lwt}}
\newcommand{\ext}{\mathrm{ext}}

\newcommand{\hol}{\mathrm{h}}

\newcommand{\vNtimes}{\bar{\otimes}}

\title{$I$-factorial quantum torsors and Heisenberg algebras of quantized universal enveloping type}

\author{Kenny De Commer\thanks{Vakgroep wiskunde, Vrije Universiteit Brussel (VUB), B-1050 Brussels, Belgium, email: {\tt kenny.de.commer@vub.ac.be}} \thanks{Partially supported by the FWO grant G.0251.15N and the grant H2020-MSCA-RISE-2015-691246-QUANTUM DYNAMICS}}
\date{}

\begin{document}
\maketitle
\date

\begin{abstract}\noindent We introduce a notion of \emph{$I$-factorial quantum torsor}, which consists of an integrable ergodic action of a locally compact quantum group on a type $I$-factor such that also the crossed product is a type $I$-factor. We show that any such $I$-factorial quantum torsor is at the same time a $I$-factorial quantum torsor for the dual locally compact quantum group, in such a way that the construction is involutive. As a motivating example, we show that quantized compact semisimple Lie groups, when amplified via a crossed product construction with the function algebra on the associated weight lattice, admit $I$-factorial quantum torsors, and give an explicit realization of the dual quantum torsor in terms of a deformed Heisenberg algebra for the Borel part of a quantized universal enveloping algebra. 
\end{abstract}

\section*{Introduction}

The theory of \emph{locally compact quantum groups} \cite{KV00,KV03} provides a vast generalization of the classical theory of locally compact groups. Being steeped in the theory of von Neumann algebras and Tomita-Takesaki theory, it is the proper setting in which to study quantum symmetries such as arising for example from subfactor theory \cite{EnN96,Eno98,Vae01}. One of the main attributes of the theory is the existence of a \emph{generalized Pontryagin duality theory}, allowing for a uniform treatment of many classical group theoretical results and constructions. 

In this article, we will be concerned with the theory of \emph{Galois objects}. These structures first arose in the setting of Hopf algebras as the proper generalization of the notion of \emph{torsor}, see \cite{Sch04} for an overview. A theory of Galois objects in the analytic framework of locally compact quantum groups was developed in \cite{DeC11}. They can be defined as von Neumann algebras with an ergodic and integrable action of a locally compact quantum group such that the crossed product von Neumann algebra is a type $I$-factor. In this paper we will study Galois objects which are \emph{themselves} type $I$-factors. These Galois objects will be called \emph{$I$-factorial Galois objects} or \emph{I-factorial quantum torsors}. Our main result shows that the natural adjoint action of the dual quantum group is again a $I$-factorial Galois object, and that moreover this construction is involutive. 

This situation is not as uncommon as it may seem on first sight. In the setting of finite-dimensional Hopf algebras, such structures were studied in \cite[Section 5]{AEGN02}. In the analytic setting, there is a canonical class of examples whereby the \emph{Heisenberg double} of a locally compact quantum group \cite{VV03} becomes a $I$-factorial quantum torsor for the Cartesian product of the quantum group with its dual. However, our main focus in this article will be on a class of $I$-factorial quantum torsors which are of Heisenberg type in the \emph{complex analytic} setting, but not in the \emph{real analytic} setting: they arise as a (deformed) Heisenberg double of a complex Hopf algebra, and come equipped with a $*$-structure swapping the holomorphic and the anti-holomorphic parts. This purely algebraic data contains however insufficient spectral information to integrate this structure to the locally compact quantum setting, and a more hands-on approach needs to be taken for this, using the general theory developed in the first part. 

The complex Hopf algebra that we are concerned with is the quantized universal enveloping algebra $U_q(\mfb)$ of a Borel subalgebra $\mfb$ of a semisimple complex Lie algebra $\mfg$. Indeed, this Hopf algebra has a natural skew self-pairing, giving rise to a \emph{Heisenberg double} and a \emph{Drinfeld double}, the latter being an amplification of the \emph{quantized universal enveloping algebra} $U_q(\mfg)$.  The key observation, allowing to connect the algebraic with the analytic framework, will be a variation on the fact that the Heisenberg algebra of the nilpotent part of the quantized Borel algebra can be realized as the function algebra on the (big open) \emph{quantum Schubert cell} associated to $\mfg$, see \cite[Section 10]{Jos95} and the more recent works \cite{Yak10,KOY13,DCN15}.

The precise contents of this paper are as follows. 

In the \emph{first four sections}, we deal with the general analytic theory. In the \emph{first section}, we recall the notion of a locally compact quantum group \cite{KV03}.  In the \emph{second section}, we introduce the notion of \emph{Galois object} for a locally compact quantum group \cite{DeC11}, and recall some of the main results. In the  \emph{third section} we introduce \emph{$I$-factorial Galois objects}, and prove our main duality theorem, Theorem \ref{TheoDualMain}. In the short fourth section, we briefly consider the example of the Heisenberg double in the context of quantum group von Neumann algebras.

In the \emph{next four sections}, we deal with the specific algebraic theory of quantized universal enveloping algebras and their duals. This part can be read independently of the first four sections. In the \emph{fifth section}, we recall the main definitions and results concerning quantized enveloping algebras and their duals. In the \emph{sixth section}, we amplify the quantized universal enveloping algebras by their Cartan parts, and realize their duals as \emph{multiplier Hopf $*$-algebras} \cite{VDae94}. In the \emph{seventh section}, we treat the theory of Galois objects in the algebraic setting of Hopf algebras \cite{Sch04} and multiplier Hopf $*$-algebras \cite{DeC09}, and show how the above amplified quantized universal enveloping algebras as well as their duals admit natural Galois objects. We also show how these two Galois objects can be brought into correspondence with each other. In the \emph{eighth section}, we recall some of the Levendorskii-Soibelman representation theory of duals of quantized universal enveloping algebras, as well as of the deformed Heisenberg $*$-algebra of the quantized universal enveloping Borel subalgebras. 

Finally, in the \emph{ninth section} we tie the above parts together, and show by means of our general theory, developed in the first part, how the above algebraic constructions integrate to the analytic setting. The \emph{tenth section} provides an outlook to further research. 

\emph{General notations}

We denote by $\otimes$ both the tensor product of vector spaces and Hilbert spaces -- it should be clear from the context which tensor product is used. We denote by $\vNtimes$ the spatial tensor product of von Neumann algebras. We further denote by \[\Sigma: V\otimes W\rightarrow W\otimes V,\quad v\otimes w \mapsto w\otimes v\] the flip map for vector spaces as well as Hilbert spaces, while for algebras $A$,$B$ we will rather use the notation \[\varsigma: A\otimes B\rightarrow B\otimes A,\quad a\otimes b\mapsto b\otimes a,\] and similarly for tensor products of von Neumann algebras. 

For $\Hsp$ a Hilbert space we let the scalar product be linear in the second factor, and denote 
\[
\omega_{\xi,\eta}(x) = \langle \xi,x\eta\rangle,\qquad x\in B(\Hsp),\xi,\eta\in \Hsp.
\] 
We also make frequent use of the leg numbering notation as is common in quantum group theory. For example, if $X\in B(\Hsp^{\otimes 2})$, then 
\[
X_{13} = (\id\otimes \varsigma)(X\otimes 1) = (1\otimes \Sigma)(X\otimes 1)(1\otimes \Sigma) \in B(\Hsp^{\otimes 3}).
\] 

When $M$ is a von Neumann algebra, we denote by $M^+$ its positive cone, by $M_*$ its pre-dual, by $M_*^+$ the positive cone of its predual, and by $\mathcal{S}_*(M)$ its space of normal states. For $\omega \in M_*$ we denote 
\[
\overline{\omega}(x) = \overline{\omega(x^*)}.
\] 
We denote by $L^2(M) = (L^2(M),\pi_M,J_M,\mathfrak{P})$ a \emph{standard form} of $M$ \cite[Chapter IX, Definition 1.13]{Tak2}, where $\pi_M$ is the standard representation of $M$ on $L^2(M)$, where $J_M = J_M^* = J_M^{-1}$ is the \emph{modular conjugation} and where $\mathfrak{P}$ is the positive cone. In practice we will suppress the notation $\pi_M$ and view $M\subseteq B(L^2(M))$. 

We also use standard notation for weight theory: if $\varphi$ is a normal, semi-finite, faithful (nsf) weight on a von Neumann algebra $M$, we denote 
\[
\mathscr{N}_{\varphi} = \{x\in M\mid \varphi(x^*x)<\infty\},\quad \mathscr{M}_{\varphi} = \mathscr{N}_{\varphi}^*\mathscr{N}_{\varphi},\quad \mathscr{M}_{\varphi}^+ = \mathscr{M}_{\varphi}\cap M^+.
\] 
We denote by
\[
\Lambda_{\varphi}: \mathscr{N}_{\varphi} \rightarrow L^2(M)
\]
the canonical \emph{GNS-map}.

For $A$ an algebra and $V,W$ subspaces of $A$, we denote \[VW = \{\sum_i v_i w_i \mid v_i\in V,w_i\in W\}.\] The same applies when $W$ is a submodule of a module for $A$.

When working with non-unital algebras, we will always assume that they  imbed into their multiplier algebra $M(A)$ and that the identity map $\iota: A \rightarrow A$ is non-degenerate, see \cite[Appendix]{VDae94} for a brief discussion of this terminology. 

For $(H,\Delta)$ a Hopf algebra, we will use the sumless Sweedler notation \[\Delta(x) = x_{(1)}\otimes x_{(2)},\qquad x\in H.\]
For $\theta_1,\theta_2$ two linear maps from $H$ into vector spaces $V,W$, we denote \[\theta_1* \theta_2 = (\theta_1\otimes \theta_2)\circ \Delta: H \rightarrow V\otimes W\] for their convolution product.

\section{Quantum group von Neumann algebras} 

\begin{Def}\cite[Definition 1.1]{KV03} A \emph{quantum group von Neumann algebra} consists of a von Neumann algebra $M$ together with a unital normal $*$-homomorphism \[\Delta: M \rightarrow M\vNtimes M\] satisfying the coassociativity condition \[(\Delta\otimes \id)\Delta = (\id\otimes \Delta)\Delta\] and for which there exist normal, semi-finite, faithful (nsf) weights \[\varphi,\psi: M^+ \rightarrow [0,+\infty]\] such that for all $x\in M^+$ and all $\omega\in \mathcal{S}_*(M)$ \[\varphi((\omega\otimes \id)\Delta(x)) = \varphi(x),\qquad \psi((\id\otimes \omega)\Delta(x)) = \psi(x).\] These properties are called respectively \emph{left-invariance} and \emph{right-invariance}.
\end{Def} 

The nsf weights $\varphi,\psi$ are unique up to a scalar multiple \cite[Theorem 2.5]{VD14}. If we fix $\varphi$, there is a unique right invariant nsf weight $\psi = \varphi^R$ such that the Connes cocycle derivative of $\psi$ with respect to $\varphi$ satisfies \[(D\psi : D\varphi)_t = \nu^{it^2/2}\delta^{it}\] for some scalar $\nu>0$ and some operator $\delta>0$ affiliated with $M$ which is \emph{group-like},  \[\Delta(\delta^{it}) = \delta^{it}\otimes \delta^{it},\] see for example the proof of \cite[Theorem 2.11]{VD14}. We will in the following always assume $\psi = \varphi^R$. Then $\nu$ and $\delta$ are independent of the choice of $\varphi$, and are called respectively  the \emph{scaling constant} and \emph{modular element}. We will use as a shorthand notation
\[J  = J_M,\qquad \Lambda = \Lambda_{\varphi},\qquad \Lambda^R = \Gamma =  \Lambda_{\psi}.\] The GNS-maps $\Lambda$ and $\Lambda^R$ are then related by \[\Lambda^R(x) = \nu^{-i/8}\Lambda(x\delta^{1/2}),\] with $x$ an element in $M$ such that $x\delta^{1/2}$ closes to a bounded operator in $\mathscr{N}_{\varphi}$.

Associated to $(M,\Delta)$ we have the \emph{unitary left and right regular corepresentations} \[W \in M\vNtimes B(L^2(M)),\qquad V\in B(L^2(M))\vNtimes M,\] uniquely determined by the fact that for all $\omega \in B(L^2(M))_*$ one has 
\[(\omega \otimes \id)(W^*) \Lambda(x) = \Lambda((\omega\otimes \id)\Delta(x)),\qquad x\in \mathscr{N}_{\varphi},\]
\[(\id \otimes \omega)(V) \Gamma(x) = \Gamma((\id\otimes \omega)\Delta(x)),\qquad x\in \mathscr{N}_{\psi}.\] 
They are multiplicative unitaries in the sense that $W$ and $V$ are unitaries satisfying the pentagon equations \begin{equation}\label{EqPentWV} W_{12}W_{13}W_{23} = W_{23}W_{12},\qquad V_{12}V_{13}V_{23} = V_{23}V_{12},\end{equation} see \cite[Theorem 1.2]{KV03}. They are independent of the precise choice of normalisation for $\varphi$ or $\psi$ and implement $\Delta$, \begin{equation}\label{EqImplMU} \Delta(x) = W^*(1\otimes x)W = V(x\otimes 1)V^*.\end{equation} Together with the pentagon equations this implies \begin{equation}\label{EqMultDel} (\Delta\otimes \id)(W) = W_{13}W_{23},\qquad (\id\otimes \Delta)(V) = V_{12}V_{13}.\end{equation}

The quantum group von Neumann algebra $(M,\Delta)$ can be modified by changing the product or coproduct, cf. \cite[Section 4]{KV03}. First of all, we can flip the coproduct \[\Delta^{\opp} = \varsigma \circ \Delta,\qquad M^{\cop}= (M,\Delta^{\opp}).\] We endow this with the respective left and right invariant nsf weights $\psi = \varphi^R$ and $\varphi = \psi^{R}$, so that the associated multiplicative unitaries become  \begin{equation}\label{EqMultUnCop}W^{\cop} = \Sigma V^*\Sigma,\qquad V^{\cop} = \Sigma W^*\Sigma,\end{equation} with $\Sigma$ the flip map on $L^2(M)\otimes L^2(M)$. 

On the other hand, we can also flip the product, which by Tomita theory can be realized by taking the commutant, \[M^{\prime}  =  JMJ.\] We can endow $M^{\prime}$ with the coproduct \[\Delta^{\prime}(x) = (J\otimes J)\Delta(JxJ)(J\otimes J),\] so that invariant weights are given by \[\varphi^{\prime}(x) = \varphi(Jx^*J),\qquad \psi'(x) =\varphi^{\prime R}(x) = \psi(Jx^*J).\] Choosing as the GNS-maps \[\Lambda^{\prime}(x) = J\Lambda(JxJ),\qquad \Gamma'(x) = \Lambda^{\prime R}(x) = J\Gamma(JxJ),\] we have the associated multiplicative unitaries \begin{equation}\label{EqMultUnOp} W^{\prime} = (J\otimes J)W(J\otimes J),\qquad V^{\prime} = (J\otimes J)V(J\otimes J).\end{equation}

 Quantum group von Neumann algebras also admit a duality theory. 

\begin{Theorem}\cite[Definition 1.5]{KV03} The $\sigma$-weak closure of \[\{(\omega\otimes \id)(W)\mid \omega \in M_*\}\subseteq B(L^2(M))\] is a von Neumann algebra, and defines a quantum group von Neumann algebra $\hat{M} = M^{\wedge}$ by means of the comultiplication \[\hat{\Delta}(x) = \Sigma W(x\otimes 1)W^*\Sigma.\]
\end{Theorem} 
One calls $\hat{M}= (\hat{M},\hat{\Delta})$ the \emph{(Pontryagin) dual} of $M= (M,\Delta)$. 

By \cite[Proposition 2.15]{KV03} the regular corepresentations of $\hat{M}$ are given by \begin{equation}\label{EqFormDual}\hat{W} = \Sigma W^*\Sigma,\qquad \hat{V} = (J\otimes J)W(J\otimes J).\end{equation}

 With $\mathscr{I}$ the set of $\omega \in M_*$ for which there exists a (necessarily unique) vector $\xi_{\omega}\in L^2(M)$ such that \[\omega(y^*) =\langle \Lambda(y),\xi_{\omega}\rangle,\qquad \forall y \in \mathscr{N}_{\varphi},\] one can show that 
there exists a unique left invariant nsf weight \[\hat{\varphi}: \hat{M}^+\rightarrow [0,+\infty]\] such that
 the $\sigma$-strong-norm closure $\hat{\Lambda}$ of \[\hat{M}\supseteq \{(\omega\otimes \id)(W)\mid \omega \in \mathscr{I}\} \rightarrow L^2(M),\quad (\omega\otimes \id)(W) \mapsto \xi_{\omega}\] is the GNS-map of $\hat{\varphi}$, see e.g.~ \cite[Theorem 3.13]{VD14}. In the following we take $\hat{\varphi}$ as a canonical left invariant weight on $\hat{M}$, and identify $L^2(\hat{M}) = L^2(M)$.

\begin{Theorem}\cite[Theorem 3.18]{VD14} The following \emph{Pontryagin biduality} holds:  \[M^{\wedge\wedge} = M.\] 
\end{Theorem}

We can write here an actual equality since by construction the above von Neumann algebras are both concretely implemented on the same Hilbert space $L^2(M)$. 

Let $\hat{J}$ be the modular conjugation of $\hat{M}$ with respect to $L^2(M)$. Then conjugation with $\hat{J}$ leaves $M$ fixed and we obtain a $*$-anti-isomorphism \[R: M \rightarrow M,\quad x\mapsto \hat{J}x^*\hat{J},\] called the \emph{unitary antipode} of $(M,\Delta)$. It satisfies \[\Delta\circ R = (R\otimes R)\circ \Delta^{\opp},\quad \varphi^R = \varphi \circ R.\]
 
We denote the scaling constant and modular element of $\hat{M}$ by $\hat{\nu}$ and $\hat{\delta}$. In fact, one has $\hat{\nu} = \nu^{-1}$. The modular conjugations $J$ and $\hat{J}$ commute up to a scalar, \[J\hat{J} = \nu^{-i/4} \hat{J}J.\]

We will mostly need the quantum group von Neumann algebra $\hat{M}' = M^{\wedge\prime}$, which by \eqref{EqFormDual}, its dual and \eqref{EqMultUnOp} has respective regular corepresentations \begin{equation}\label{EqFormDualOp} \hat{W}' = V,\qquad \hat{V}' = (J\hat{J}\otimes J\hat{J})W(\hat{J}J\otimes \hat{J}J).\end{equation}
It follows in particular that \[M^{\wedge\,\prime} = \{(\id\otimes \omega)(V) \mid \omega \in M_*\}'',\] and by \eqref{EqImplMU}  we obtain that \begin{equation}\label{EqObsImpl}\hat{V}' (x\otimes 1)\hat{V}^{\prime*} = \hat{\Delta}'(x),\qquad \forall x\in \hat{M}'.\end{equation} Also the operation $M \mapsto M^{\wedge\prime}$ is involutive, but only after twisting with $J\hat{J}$. More precisely, \[M^{\wedge\prime\wedge\prime} = M^{\prime \cop}\] and we have an isomorphism of quantum groups von Neumann algebras \[M^{\prime \cop}\rightarrow M,\quad x \mapsto J\hat{J} x \hat{J}J.\] We also have the identity $\hat{M}' = M^{\cop \wedge}$, but one has to be careful however with the respective GNS-maps as they are shifted by a gauge factor: for $x\in \hat{M}' = M^{\cop,\wedge}$ square-integrable, we have \[\Lambda^{R\wedge}(x) = \nu^{-i/8}\hat{\Lambda}^{\prime}(x).\]

\section{Coactions and Galois objects}

Let us fix a quantum group von Neumann algebra $(M,\Delta)$. 

\begin{Def}
We call \emph{right coaction} of $(M,\Delta)$ on a von Neumann algebra $N$ any unital normal $*$-homomorphism $\alpha: N\rightarrow N\vNtimes M$ such that \[(\id\otimes \Delta)\alpha= (\alpha\otimes \id)\alpha.\] 
\end{Def} 

Similarly, one defines left coactions $\gamma:N\rightarrow M\vNtimes N$. Any right coaction $(N,\alpha)$ of $(M,\Delta)$ then determines a left coaction $(N,\alpha^{\opp})$ of $(M,\Delta^{\opp})$ by \[\alpha^{\opp}: N \rightarrow M\vNtimes N,\quad x \mapsto \varsigma \alpha(x),\] allowing to transfer statements concerning right coactions to corresponding ones for left coactions. 

We will in the following be mainly concerned with right coactions. 

\begin{Def} The \emph{subalgebra of coinvariants} for a right coaction $(N,\alpha)$ is the von Neumann subalgebra  \[N^{\alpha} = \{x\in N\mid \alpha(x) = x\otimes 1\}\subseteq N.\] We call $\alpha$ \emph{ergodic} if $N^{\alpha}= \C$.
\end{Def}

\begin{Def} The  \emph{crossed product} von Neumann algebra for a right coaction $(N,\alpha)$ is the von Neumann algebra \[N\rtimes M = N\rtimes_{\alpha} M = \{(1\otimes x)\alpha(y)\mid x\in \hat{M}',y\in N\}'' \subseteq N\vNtimes B(L^2(M)).\] 
\end{Def}

One has on $N\rtimes M$ a \emph{dual right coaction} $\hat{\alpha}$ of $\hat{M}'$, determined by \begin{equation}\label{EqDualCoaction}\hat{\alpha}(z) = \hat{V}'_{23}(z\otimes 1)\hat{V}^{\prime*}_{23}.\end{equation} This definition entails for $x\in N$ and $y\in \hat{M}'$ 
\[\hat{\alpha}(\alpha(x)) = 1\otimes \alpha(x),\qquad \hat{\alpha}(1\otimes y) = 1\otimes \hat{\Delta}'(y).\] 
 
Let us recall the following \emph{biduality result}, cf. \cite[Theorem 2.6]{Vae01} for the left handed version.

\begin{Theorem}\label{TheoBiduCoac}  Considering \[(N\rtimes_{\alpha} M)\rtimes_{\hat{\alpha}} \hat{M}' \subseteq N\vNtimes B(L^2(M)\otimes L^2(M)),\] one has an isomorphism of von Neumann algebras \begin{equation}\label{EqDefChiRtimes}\pi_{\bidual}: N \vNtimes B(L^2(M)) \cong 
(N\rtimes M)\rtimes \hat{M}',\qquad x \mapsto V^*_{23}(\alpha\otimes \id)(x)V_{23}.\end{equation} In particular, for $x\in N,y\in \hat{M}'$ and $z \in M^{\wedge\prime\wedge\prime}$ we have \[\alpha(x) \mapsto \alpha(x)\otimes 1,\quad 1\otimes y \mapsto 1\otimes \hat{\Delta}'(y),\quad 1\otimes z \mapsto 1\otimes 1 \otimes z.\]

Moreover, this isomorphism satisfies the equivariance condition 
\begin{equation}\label{EqEqBidual}
(\pi_{\bidual}^{-1}\otimes \Ad(J\hat{J}))(\alpha^{\wedge\wedge}(x)) = \Sigma_{23}W_{23}(\alpha\otimes \id)(\pi_{\bidual}^{-1}(x))W_{23}^*\Sigma_{23}
\end{equation}
for $x\in (N\rtimes M) \rtimes \hat{M'}$.
\end{Theorem}

\begin{Def} A right coaction $(N,\alpha)$ is called \emph{integrable} if the set \[\{x \in N^+\mid \exists y\in N^+,\forall \omega \in M_*^+, \varphi((\omega\otimes \id)\alpha(x)) = \omega(y)\}\] has $\sigma$-weakly dense linear span in $N$.
\end{Def}

If $x$ lies in the above set, the element $y$ is uniquely determined and lies in $N^{\alpha}$. We then write \[y = (\id\otimes \varphi)\alpha(x)\in N^{\alpha}.\] Choosing a fixed nsf weight $\mu$ on $N^{\alpha}$, we can in the case of an integrable coaction define on $N$ the nsf weight \begin{equation}\label{EqVarphiN}\mu_N(x) = \mu((\id\otimes \varphi)\alpha(x)),\qquad x\in N^+\end{equation} with associated GNS-map $\Lambda_N$. For example, any dual coaction $\hat{\alpha}$ is integrable, and for $\mu_N$ an nsf weight on $N$ we then write \[\mu_{N\rtimes M}(x) =\mu_N(\alpha^{-1}((\id\otimes \hat{\varphi}')(\hat{\alpha}(x))),\qquad x\in N\rtimes M.\] We can in this case make $L^2(N)\otimes L^2(M)$ into a GNS-space for $N\rtimes M$ by the unique GNS-map $\Lambda_{N\rtimes M}$ such that \[\Lambda_{N\rtimes M}((1\otimes y)\alpha(x))  =  \Lambda_{N}(x)\otimes \hat{\Lambda}'(y),\quad  x\in \mathscr{N}_{\mu_{N}},y\in\mathscr{N}_{\hat{\varphi}'},\] see \cite[Definition 3.4, Proposition 3.10]{Vae01}. This realisation of $L^2(N)\otimes L^2(M)$ as the standard Hilbert space for $N\rtimes M$ is independent of the choice of $\mu$ by \cite[Proposition 4.1]{Vae01}, as well as from the choice of invariant measure on $\hat{M}'$ for that matter. With $J_{N\rtimes M}$ the associated modular conjugation, the canonical unitary \[U=J_{N\rtimes M} (J_N\otimes \hat{J})\] is called the \emph{standard implementing unitary corepresentation} for $\alpha$,  as it satisfies by \cite[Proposition 3.7.1., Proposition 3.12, Theorem 4.4]{Vae01} the properties \[U \in B(L^2(N))\vNtimes M,\quad U(x\otimes 1)U^* = \alpha(x),\qquad (\id\otimes \Delta)(U) = U_{12}U_{13}.\] 

In the case of an \emph{integrable} coaction $(N,\alpha)$, we have the following expression for the standard implementing unitary corepresentation.

\begin{Lem}\cite[Proposition 4.3]{Vae01} Let $\alpha$ be an integrable right coaction, let $\mu$ be an nsf weight on $N^{\alpha}$ and let $\mu_N$ be the associated nsf weight on $N$ with GNS-map $\Lambda_N$. Then the standard implementing unitary corepresentation $U \in B(L^2(N))\vNtimes M$ satisfies \[(\id\otimes \omega_{\eta,\xi})(U)\Lambda_N(x) = \Lambda_N((\id\otimes \omega_{\eta,\delta^{-1/2}\xi})\alpha(x))\] for all $\eta \in L^2(M)$, $\xi\in \mathscr{D}(\delta^{-1/2})$ and $x\in \mathscr{N}_{\varphi_N}$. 
\end{Lem} 

We can then also present $N\rtimes M$ directly on $L^2(N)$, although not necessarily in a faithful way.

\begin{Theorem}\label{TheoIntImp}\cite[Theorem 5.3]{Vae01} If $(N,\alpha)$ is an integrable right coaction of $(M,\Delta)$, there exists a unique normal unital $*$-homomorphism 
\begin{equation}\label{EqRepNcrossM}
\pi_{\rtimes}: N\rtimes M \rightarrow B(L^2(N))
\end{equation}
such that 
\[
\alpha(x) \mapsto x,\qquad 1\otimes (\id\otimes \omega)(V) \mapsto (\id\otimes \omega)(U)
\] for $x\in N$ and $\omega\in M_*$.  Moreover, the range of this map equals $J_N (N^{\alpha})' J_N$. 
\end{Theorem}

In particular, we will write \begin{equation}\label{EqRephatM} \hat{\pi}': \hat{M}'\rightarrow B(L^2(N)),\qquad 1\otimes x \mapsto \hat{\pi}'(x) = \pi_{\rtimes}(1\otimes x),\qquad x\in \hat{M}'.\end{equation}

\begin{Def}\cite[Definition 3.9]{DeC11}\label{DefGalObvN} Let $(N,\alpha)$ be an integrable coaction. We call $(N,\alpha)$ a \emph{(right) Galois coaction} if $\pi_{\rtimes}$ is faithful. 

We call $(N,\alpha)$ a \emph{(right) Galois object} if moreover $\alpha$ is ergodic. 
\end{Def}

In case $(N,\alpha)$ is an integrable and ergodic coaction, there exists a unique nsf-weight $\varphi_N$ on $N$ such that \[\varphi_N(x)1 = (\id\otimes \varphi)\alpha(x),\qquad x\in N^+.\] 
We will call it, by analogy, the \emph{left invariant nsf weight} on $N$. There exists then a unique unitary \[\Ww: L^2(N)\otimes L^2(M)\rightarrow L^2(N)\otimes L^2(N),\] called the \emph{Galois map}, such that for all $\omega \in B(L^{2}(N))_*$ \[(\omega\otimes \id)(\mathcal{W}^*) \Lambda_{N}(x) = \Lambda_{N}((\omega\otimes \id)\alpha(x)),\quad x\in \mathscr{N}_{\varphi_{N}}.\] It can be shown that $(N,\alpha)$ is a Galois object if and only if $\Ww$ is unitary. Note that in \cite{DeC11} the notation $\widetilde{G} = \Sigma \mathcal{W}^*\Sigma$ was used.

For $(N,\alpha)$ a Galois object, it was shown in \cite[Section 4]{DeC11} that \begin{equation}\label{EqWInN}\Ww \in N\vNtimes B(L^2(M),L^2(N)),\end{equation} that $\Ww$ satisfies the \emph{hybrid pentagon equation} \begin{equation}\label{EqHybPent} \Ww_{12}\Ww_{13}W_{23} = \Ww_{23}\Ww_{12},\end{equation} and that $\Ww$ implements the coaction $\alpha$ in the sense that \begin{equation}\label{EqImplGalAl}\alpha(x) = \Ww^*(1\otimes x)\Ww,\qquad x\in N.\end{equation} In particular, we find \begin{equation}\label{ActAlphN} (\alpha\otimes \id)\Ww= \Ww_{13}W_{23}.\end{equation} We also recall from \cite[Lemma 4.2, (iii),(iv)]{DeC11} that, with $\hat{\pi}'$ as in \eqref{EqRephatM},  \begin{equation}\label{EqGalImpCross2} \Ww(1\otimes x) = (1\otimes \hat{\pi}'(x))\Ww,\qquad x\in \hat{M}',\end{equation}
\begin{equation}\label{EqGalImpCross3} (\hat{\pi}'(JxJ)\otimes 1)\Ww = \Ww(\hat{\pi}'\otimes \id)((J\otimes \hat{J})\hat{\Delta}'(x)(J\otimes \hat{J})),\quad x\in \hat{M}',
\end{equation} and, by \cite[Lemma 4.9.(iii)]{DeC11}
\begin{equation}\label{EqGalImpCross4} J_N\hat{\pi}'(x)J_N = \hat{\pi}'(JxJ),\qquad x\in \hat{M}'.
\end{equation}

Note that in case $(N,\alpha)$ is a Galois object, we obtain an isomorphism \begin{equation}\label{EqIsoCrossI} N\rtimes M \underset{\pi_{\rtimes}}{\cong} B(L^2(N))\end{equation} by the last part of Theorem \ref{TheoIntImp}. By \eqref{EqImplGalAl} and \eqref{EqGalImpCross2} the inverse of $\pi_{\rtimes}$ is implemented by the Galois unitary, \begin{equation}\label{EqGalCrossIso} \pi_{\rtimes}^{-1}: B(L^2(N))\cong N\rtimes M,\quad x\mapsto \Ww^*(1\otimes x)\Ww.\end{equation}
Let us write \begin{equation}\label{EqWeightB}\varphi_{B(L^2(N))} = \varphi_{N\rtimes M}\circ \pi_{\rtimes}^{-1}.\end{equation} From \cite[Proposition 3.7]{DeC11} it follows that $\varphi_{B(L^2(N))}$ satisfies \[\varphi_{B(L^2(N))}(\Lambda_N(x)\Lambda_N(y)^*) = \varphi_N(xy^*),\qquad x,y\in \mathscr{N}_{\varphi}\cap \mathscr{N}_{\varphi}^*,\] where we identify $L^2(N) = B(\C,L^2(N))$. Hence \begin{equation}\label{EqFormWeightB}\varphi_{B(L^2(N))}(x) = \Tr(\nabla_N^{1/2}x\nabla_N^{1/2}),\quad x\in B(L^2(N))_+,\end{equation} where $\nabla_N$ is the modular operator of $\varphi_N$. 

Consider now $L^2(N) \otimes L^2(N)$ as the standard space for $B(L^2(N))$ by  
 \begin{equation}\label{EqStandB} \pi_{B(L^2(N))}(x) = x\otimes 1,\quad J_{B(L^2(N))}(\xi\otimes \eta) = J_N\eta\otimes J_N\xi,\end{equation} and with $\mathfrak{P}_{B(L^2(N))}$ corresponding to the positive Hilbert-Schmidt operators in $B(L^2(N))$ under the imbedding \[L^2(N) \otimes L^2(N) \hookrightarrow B(L^2(N)),\quad \xi\otimes \eta\mapsto \xi(J_N\eta)^*.\]
Then it is elementary to verify that  \[\Lambda_{B(L^2(N))}(\xi\Lambda_N(y^*)^*) = \xi\otimes \Lambda_N(y),\quad \xi\in L^2(N),y\in \mathscr{N}_{\varphi_N}\cap \mathscr{N}_{\varphi_N}^*.\] 
 It now follows from \cite[Proposition 3.7]{DeC11} that \begin{equation}\label{EqLinkFacCross}\Ww^* \Sigma \Lambda_{B(L^2(N))}(x) = \Lambda_{N\rtimes M}(\pi_{\rtimes}^{-1}(x)),\qquad x\in \mathscr{N}_{\varphi_{B(L^2(N))}}.\end{equation}

If $(N,\alpha)$ is a right coaction, an nsf weight $\psi_{N}$ on $N$ is called \emph{right invariant} if for all $\omega \in \mathcal{S}_*(M)$ and all $x\in N^+$ one has \[\psi_{N}((\id \otimes \omega)\alpha(x)) = \psi_{N}(x).\] 

\begin{Theorem}\cite[Theorem 4.19]{DeC11} If $(N,\alpha)$ is a Galois object, there exists a right invariant nsf weight $\psi_{N}$ for $\alpha$, unique up to multiplication by a positive scalar. 
\end{Theorem}

In general, there is no canonical normalisation available for $\psi_{N}$. Nevertheless, once $\psi_N$ has been chosen there exists by its construction in \cite[Theorem 4.19]{DeC11} and \cite[Lemma 4.18 and Theorem 4.23]{DeC11} a unique $\delta_N>0$ affiliated with $N$ such that \begin{equation}\label{EqConDerPhiPsi}(D\psi_{N} : D\varphi_{N})_t = \nu^{it^2/2} \delta_{N}^{it}\end{equation} with $\delta_{N} \eta N$ an invertible positive operator and $\nu$ the scaling constant of $M$. We further have  by \cite[Proposition 4.16]{DeC11} that  
\begin{equation}\label{EqDelGrouplike} \alpha(\delta_{N}^{it}) = \delta_{N}^{it}\otimes \delta^{it},\end{equation} with $\delta$ the modular element of $M$, and this uniquely characterizes $\delta_N$ up to a positive scalar by ergodicity of $\alpha$. 

Evidently, one can develop also a theory of \emph{left} Galois objects $(N,\gamma)$, where $\gamma$ is an integrable ergodic left coaction with the Galois map $\Vv$ determined by \begin{equation}\label{EqVv}(\omega \otimes \id)(\Vv)\Gamma_N(x) = \Gamma_N((\omega \otimes \id)\gamma(x)),\qquad \omega \in B(L^2(M))_*\end{equation} a unitary, where $\Gamma_N$ is the GNS-map with respect to the weight $\psi_N$ determined by $\psi_N(x) = \psi((\id\otimes \omega)\gamma(x))$ for all $x\in N^+$ and $\omega\in \mathcal{S}_*(N)$. All of the above results then have their left analogue. 

The main source of examples of Galois objects comes from the theory of \emph{projective corepresentations}.   

\begin{Def}\cite[Definition 7.1 and Theorem 7.2]{DeC11} Let $(M,\Delta)$ be a quantum group von Neumann algebra. A \emph{projective corepresentation} of $(M,\Delta)$ is a coaction $\alpha: N \rightarrow N\vNtimes M$ with $N$ a type $I$-factor. 
\end{Def}

For $\alpha: N\rightarrow N\vNtimes M$ a projective corepresentation we can form 
\[N_{\alpha} = \alpha(N)' \cap (N\rtimes M),\] and we have a canonical isomorphism of von Neumann algebras \[N\rtimes M \cong N\vNtimes N_{\alpha},\] as $N$ is a type $I$ factor. It is easy to see that the dual right coaction $\hat{\alpha}$ of $\hat{M}'$ restricts to a coaction \begin{equation}\label{EqAlCirc}\alpha^{\circ}: N_{\alpha} \rightarrow N_{\alpha}\vNtimes \hat{M}'.\end{equation}

\begin{Theorem}\label{TheoRightGaloisDual} The couple $(N_{\alpha},\alpha^{\circ})$ is a right Galois object for $(\hat{M}',\hat{\Delta}')$.
\end{Theorem} 

\begin{proof} This is contained in the proof of Theorem 7.2 of \cite{DeC11}.
\end{proof}

In particular, let $\alpha:N\rightarrow N\vNtimes M$ be a right Galois object for $(M,\Delta)$. As $N\rtimes M$ is a type $I$-factor, $\hat{\alpha}$ is a projective corepresentation of $\hat{M}'$. We hence obtain a right Galois object for $(M,\Delta)$ defined by \[((N\rtimes_{\alpha}M)_{\hat{\alpha}},\hat{\alpha}^{\circ}).\]

\begin{Theorem}\label{TheoBidualGal} Let $(N,\alpha)$ be a right Galois object for $(M,\Delta)$. Then we have an isomorphism of von Neumann algebras \[\pi: N \cong (N\rtimes_{\alpha}M)_{\hat{\alpha}}, \quad x\mapsto \pi_{\bidual}(\Ww^*(x\otimes 1)\Ww)\] which is equivariant in the sense that \begin{equation}\label{EqEqProj} (\pi \otimes \Ad(J\hat{J}))(\alpha(x)) = \hat{\alpha}^{\circ}(\pi(x)),\qquad x\in N.\end{equation}
\end{Theorem}
 
Recall that the isomorphism $\pi_{\bidual}$ was defined in \eqref{EqDefChiRtimes}.

\begin{proof} This is indirectly contained in \cite{DeC11}, but let us give a direct proof. 

As $\pi_{\bidual}^{-1}$ puts $N\rtimes M$ is in its ordinary position on $L^2(N)\otimes L^2(M)$, we find \[\pi_{\bidual}^{-1}( (N\rtimes M)_{\hat{\alpha}})  = (N\vNtimes B(L^2(M))\cap (N\rtimes M)'.\]

Let now $\Ww$ be the Galois unitary for $(N,\alpha)$. It then follows from \eqref{EqImplGalAl}, \eqref{EqWInN}, \eqref{EqGalImpCross2} and the final part of Theorem \ref{TheoIntImp} that \begin{equation}\label{EqConjXX} N\otimes 1 = \Ww((N\vNtimes B(L^2(M))\cap (N\rtimes M)')\Ww^*.\end{equation} This proves that $\pi$ as in the statement of the theorem is a well-defined isomorphism. 

It remains to show \eqref{EqEqProj}, which by \eqref{EqEqBidual} reduces to \[W_{23}(\alpha\otimes \id)(\Ww^*(x\otimes 1)\Ww)W_{23}^*  = \Ww_{13}^*\alpha(x)_{12}\Ww_{13},\qquad x\in N.\] 
But this follows from the fact that $\Ww$ implements $\alpha$, together with the hybrid pentagon equation \eqref{EqHybPent}.
\end{proof} 

Let us single out an important operation which was hidden in the definition of the map $\pi$ in the above theorem.

\begin{Lem} Let $(N,\alpha)$ be a Galois object with Galois unitary $\Ww$. Then \[\Ad_{\alpha}(x) = \Ww^*(x\otimes 1)\Ww\] defines a right coaction of $(\hat{M},\hat{\Delta})$ on $N$.
\end{Lem}

\begin{proof}  By \eqref{EqWInN} and \eqref{EqGalImpCross2}, we have \[\Ww^*(x\otimes 1)\Ww \in N\vNtimes \hat{M},\qquad x\in N.\] Since \[\hat{\Delta}(y)= \Ad(\Sigma W^*\Sigma)(1\otimes y),\qquad y\in \hat{M},\] the coaction property of $\Ad_{\alpha}$ follows straightforwardly from \eqref{EqHybPent}.
\end{proof}

\begin{Def}\label{DefAdj} We call the coaction $\Ad_{\alpha}$ the \emph{adjoint coaction} associated to the Galois object.
\end{Def} 

In the Hopf algebra setting, this is known as the \emph{Miyashita-Ulbrich action}, see \cite{Sch04}. 
  
\section{$I$-factorial Galois objects}

We now come to the main topic of this paper: Galois objects which define \emph{at the same time} a projective corepresentation, that is, Galois object structures on $B(\Hsp)$ for some Hilbert space $\Hsp$.

\begin{Def} Let $(M,\Delta)$ be a quantum group von Neumann algebra.  We call \emph{$I$-factorial (right) Galois object} any (right) Galois object $(N,\alpha)$ with $N$ a type $I$-factor.
\end{Def}

Our main theorem is the following \emph{duality statement}. Recall the notion of adjoint coaction from Definition \ref{DefAdj}.

\begin{Theorem}\label{TheoDualMain} If $(N,\alpha)$ is a $I$-factorial Galois object for $(M,\Delta)$, then $(N,\Ad_{\alpha})$ is a $I$-factorial Galois object for $(\hat{M},\hat{\Delta})$. Moreover, \[\Ad_{\Ad_{\alpha}} = \alpha.\]
\end{Theorem} 

The theorem will be proven in two steps, see Theorem \ref{TheoAdjGal} and Theorem \ref{TheoMain}.

\begin{Theorem}\label{TheoAdjGal} If $(N,\alpha)$ is a $I$-factorial Galois object with respect to $(M,\Delta)$, then $(N,\Ad_{\alpha})$ is a $I$-factorial Galois object with respect to $(\hat{M},\hat{\Delta})$.
\end{Theorem}

\begin{proof} As $N$ is a type $I$-factor, we can interpret $\alpha$ as a right projective corepresentation of $(M,\Delta)$. Using the notation from Theorem \ref{TheoRightGaloisDual} and the paragraph following it, we can hence consider the right Galois object $(N_{\alpha},\alpha^{\circ})$ for $\hat{M}'$. As by definition \[N_{\alpha} = \alpha(N)' \cap N\rtimes M,\] with $N$ a type $I$-factor by assumption and $N\rtimes M$ a type $I$-factor by \eqref{EqIsoCrossI}, it follows that $N_{\alpha}$ is again a type $I$-factor, and $(N_{\alpha},\alpha^{\circ})$ is a right $I$-factorial Galois object for $\hat{M}'$ by Theorem \ref{TheoRightGaloisDual}.

Now the isomorphism $\pi_{\rtimes}: N\rtimes M  \cong B(L^2(N))$ from Theorem \ref{TheoIntImp} restricts to an isomorphism \[N_{\alpha} \underset{\pi_{\rtimes}}{\cong} N'.\]  Since \[N \rightarrow N',\quad x\mapsto J_Nx^*J_N\] is a $^*$-anti-isomorphism, and \[\hat{M}\rightarrow \hat{M}',\quad x\mapsto \hat{J}x^*\hat{J}\] a comultiplication-preserving $*$-anti-isomorphism, it follows that we can transport the above right coaction $\alpha^{\circ}$ of $\hat{M}'$ on $N_{\alpha}$ to a right coaction of $\hat{M}$ on $N$,   \begin{equation}\label{EqDiffCoact}x \mapsto (J_N\otimes \hat{J})(\pi_{\rtimes}\otimes \id)(\alpha^{\circ}(\pi_{\rtimes}^{-1}(J_NxJ_N)))(J_N\otimes \hat{J}),\qquad x\in N,\end{equation} making $N$ a $I$-factorial right Galois object for $\hat{M}$.

Let us prove that the coaction in \eqref{EqDiffCoact} is equal to $\Ad_{\alpha}$. This is equivalent with proving that, for $x\in N'$,  \begin{multline*}\alpha^{\circ}(\pi_{\rtimes}^{-1}(x)) \\= (\pi_{\rtimes}^{-1}\otimes \id) ((J_N\otimes \hat{J})\Ww^*(J_N\otimes J_N))(x\otimes 1)((J_N\otimes \hat{J})\Ww^*(J_N\otimes J_N))^*.\end{multline*}

Recalling now that $\alpha^{\circ}$ is the restiction of the dual coaction $\hat{\alpha}$, we see by the formula for $\pi_{\rtimes}^{-1}$ in \eqref{EqGalCrossIso} and the defining formula \eqref{EqDualCoaction} for the dual coaction $\hat{\alpha}$ that it is sufficient to show that \begin{equation}\label{EqPentAdCo} \Ww_{12}\hat{V}^{\prime*}_{23} \Ww_{12}^*((J_N\otimes \hat{J})\Ww^*(J_N\otimes J_N))_{23}  \in B(L^2(N))\vNtimes N\vNtimes B(L^2(M)).\end{equation}

In fact, we will show something stronger, namely that the above operator lies in $B(L^2(N))\vNtimes \C\vNtimes B(L^2(M))$. 

To see this, we first observe that $\hat{\pi}'(\hat{M}')$ and $N$ generate $B(L^2(N))$ as a von Neumann algebra by \eqref{EqIsoCrossI}. It is hence enough to prove that the above operator commutes with $1\otimes x\otimes 1$ for $x\in N \cup \hat{\pi}'(\hat{M}')$. 

For $x\in N$ we compute, using that $\Ww\in N\vNtimes B(L^2(M),L^2(N))$ and that $\Ww$ implements $\alpha$, \begin{eqnarray*}&& \hspace{-1.5cm} \Ww_{12}\hat{V}_{23}^{\prime*} \Ww_{12}^*((J_N\otimes \hat{J})\Ww^*(J_N\otimes J_N))_{23} (1\otimes x \otimes 1) \\ &&= \Ww_{12}\hat{V}_{23}^{\prime*} \Ww_{12}^*(1\otimes x \otimes 1)((J_N\otimes \hat{J})\Ww^*(J_N\otimes J_N))_{23}  \\ &&= \Ww_{12}\hat{V}_{23}^{\prime*}\alpha(x)_{12} \Ww_{12}^*((J_N\otimes \hat{J})\Ww^*(J_N\otimes J_N))_{23} \\ &&= \Ww_{12}\alpha(x)_{12} \hat{V}_{23}^{\prime*}\Ww_{12}^*((J_N\otimes \hat{J})\Ww^*(J_N\otimes J_N))_{23}  \\ &&=  (1\otimes x\otimes 1)\Ww_{12} \hat{V}_{23}^{\prime*}\Ww_{12}^*((J_N\otimes \hat{J})\Ww^*(J_N\otimes J_N))_{23}.\end{eqnarray*} For $x\in \hat{M}'$, we find using \eqref{EqGalImpCross4}, \eqref{EqGalImpCross3}, \eqref{EqGalImpCross2} and \eqref{EqObsImpl} that   \begin{eqnarray*}&& \hspace{-1cm} \Ww_{12}\hat{V}_{23}^{\prime*} \Ww_{12}^*((J_N\otimes \hat{J})\Ww^*(J_N\otimes J_N))_{23} (1\otimes \hat{\pi}'(x) \otimes 1) \\ &&=  \Ww_{12}\hat{V}_{23}^{\prime*} \Ww_{12}^*((J_N\otimes \hat{J})\Ww^*(\hat{\pi}'(JxJ) \otimes 1)(J_N\otimes J_N))_{23} \\ &&= \Ww_{12}\hat{V}_{23}^{\prime*} \Ww_{12}^*((\hat{\pi}'\otimes \id)(\hat{\Delta}'(x)))_{23} ((J_N\otimes \hat{J})\Ww^*(J_N\otimes J_N))_{23} \\
&& = \Ww_{12}\hat{V}_{23}^{\prime*} \hat{\Delta}'(x)_{23} \Ww_{12}^*((J_N\otimes \hat{J})\Ww^*(J_N\otimes J_N))_{23} \\ &&= \Ww_{12}(1\otimes x \otimes 1)\hat{V}_{23}^{\prime*}  \Ww_{12}^*((J_N\otimes \hat{J})\Ww^*(J_N\otimes J_N))_{23} \\ && =  (1\otimes \hat{\pi}'(x) \otimes 1)\Ww_{12}\hat{V}_{23}^{\prime*}  \Ww_{12}^*((J_N\otimes \hat{J})\Ww^*(J_N\otimes J_N))_{23} 
\end{eqnarray*}
\end{proof}

The second part of Theorem \ref{TheoDualMain} is a bit more involved. If we are only interested in proving $(N,\Ad_{\Ad_{\alpha}})\cong (N,\alpha)$ equivariantly, the proof is not that hard, and can be derived relatively straightforwardly from Theorem \ref{TheoBidualGal}. However, to have an actual equality requires computing the Galois unitary of $\Ad_{\alpha}$.  We need some preparations.

Fix a Hilbert space $\Hsp$ such that $N  = B(\Hsp)$. Then we can identify the standard form of $N$ as \[L^2(N) \cong \Hsp\otimes \overline{\Hsp},\] with $N$ acting in the canonical way on the first component, with \[J_N: \Hsp \otimes \overline{\Hsp} \rightarrow \Hsp\otimes \overline{\Hsp},\quad \xi\otimes \overline{\eta} \mapsto \eta\otimes \overline{\xi}\] and with the self-dual cone $\mathfrak{P}_N$ consisting of the positive Hilbert-Schmidt operators under the canonical embedding $\Hsp \otimes \overline{\Hsp} \hookrightarrow B(\Hsp)$. 

We then have canonically that $N' = 1\otimes B(\overline{\Hsp}) \cong B(\overline{\Hsp})$, and we can identify the GNS-space of $N'$ with $L^2(N) = \Hsp \otimes \overline{\Hsp}$ in such a way that $N'$ acts by its natural action on the second leg and such that $J_{N'} = J_N$ and $\mathfrak{P}_{N'} = \mathfrak{P}_N$. 

For the standard form of $B(L^2(N))$ we have two natural choices. On the one hand, one has the \emph{first standard implementation} which was given by \eqref{EqStandB}.
On the other hand, since we also have \[B(L^2(N)) = B(\Hsp)\vNtimes B(\overline{\Hsp}) = N\vNtimes N',\] we can use the tensor product standard construction on \[L^2(N) \otimes L^2(N') = L^2(N)\otimes L^2(N) = (\Hsp\otimes \overline{\Hsp})^{\otimes 2}\] such that \[\pi_{N\vNtimes N'}(x)  = x_{14},\quad  J_{N\vNtimes N'} = J_N\otimes J_N\] and $\mathfrak{P}_{N\vNtimes N'}$ the closed positive linear span of elements $\xi\otimes \eta$ with $\xi,\eta\in \mathfrak{P}_{N}$. We will call this the \emph{second standard implementation}. A careful inspection shows that the two standard forms are related by the involutive unitary \begin{multline*} \Xi = \Sigma_{24} \in B((\Hsp\otimes \overline{\Hsp})^{\otimes 2}) = B(L^2(N)\otimes L^2(N)),\\ \Ad(\Xi)\circ \pi_{B(L^2(N))} = \pi_{N\vNtimes N'}.\end{multline*}

Let us use in the following the notation \[\hat{\Vv} = (J_N\otimes J_N)\Ww(J_N\otimes J),\] by analogy with \eqref{EqFormDual}. Recall also again the isomorphism $\pi_{\rtimes}: N\rtimes M \cong B(L^2(N))$ from \eqref{EqRepNcrossM}, which restricts to the representation $\hat{\pi}'$ of $\hat{M}'$ on $L^2(N)$.
 
\begin{Theorem}\label{TheoCompCalcGalAd} Let $\Ww_{\Ad}$ be the Galois unitary for $(N,\Ad_{\alpha})$. Then $\Ww_{\Ad}$ satisfies \begin{equation}\label{EqExprWad} \Ww_{\Ad} = \Xi\hat{\Vv}((\hat{\pi}'\otimes \id)((J\hat{J}\otimes J\hat{J})\hat{W}(\hat{J}J\otimes 1))\end{equation} as a map from $L^2(N)\otimes L^2(M)$ to  $L^2(N)\otimes L^2(N)$. 
\end{Theorem}
\begin{proof} Let $\varphi_N$ be the left invariant nsf weight on $N$ for $\alpha$, and let $\varphi_{\Ad}$ be the left invariant nsf weight for $\Ad_{\alpha}$. Let $\Lambda_N,\Lambda_{\Ad}$ be their respective GNS-maps. Recall the coaction $\alpha^{\circ}$ from \eqref{EqAlCirc}, which is the restriction of the dual coaction $\hat{\alpha}$ to $N_{\alpha} = \alpha(N)'\cap N\rtimes M$, and let $\widetilde{\alpha}^{\circ}$ be the coaction \[\widetilde{\alpha}^{\circ} = (\pi_{\rtimes}\otimes \id)\circ \alpha^{\circ}\circ \pi_{\rtimes}^{-1}\] of $\hat{M}'$ on $N'$, so that by the proof of Theorem \ref{TheoAdjGal} \begin{equation}\label{DefAltild} \Ad_{\alpha}(x) =(J_N\otimes \hat{J})\widetilde{\alpha}^{\circ}(J_NxJ_N)(J_N\otimes \hat{J}),\qquad x\in N.\end{equation} Then $(N',\widetilde{\alpha}^{\circ})$ is a right Galois object with left invariant weight \[\varphi_{\Ad}'(x) = \varphi_{\Ad}(J_Nx^*J_N),\] for which the canonical GNS-map is hence given by \[\Lambda_{\Ad}'(x) = J_N\Lambda_{\Ad}(J_NxJ_N),\qquad x\in \mathscr{N}_{\varphi_{\widetilde{\alpha}^{\circ}}}.\] By construction we see that under the isomorphism $B(L^2(N)) \cong N\vNtimes N'$, we have \[\varphi_{B(L^2(N))} \cong \varphi_N\otimes \varphi_{\Ad}',\] where $\varphi_{B(L^2(N))}$ is the nsf weight introduced in \eqref{EqWeightB}. Hence the GNS-map for the  weight $\varphi_{B(L^2(N))}$ is given with respect to the second standard implementation as \[\Lambda_{N\vNtimes N'}(xy) = \Lambda_N(x) \otimes \Lambda_{\Ad}'(y),\qquad x\in N,y\in N'.\] On the other hand, we will continue to denote \[\Lambda_{B(L^2(N))} = \Xi \circ \Lambda_{N\vNtimes N'}\] for the GNS-implementation with respect to the first standard implementation.

Let now $\Ww_{\widetilde{\alpha}^{\circ}}$ be the Galois unitary for $\widetilde{\alpha}^{\circ}$. By \eqref{DefAltild}, we immediately get \begin{equation}\label{EqWadaltild} \Ww_{\Ad} = (J_N\otimes J_N)\Ww_{\widetilde{\alpha}^{\circ}} (J_N\otimes \hat{J}).\end{equation} 

By definition, $\Ww_{\widetilde{\alpha}^{\circ}}$ satisfies \[\Ww_{\widetilde{\alpha}^{\circ}}^*: L^2(N)\otimes L^2(N)\rightarrow L^2(N)\otimes L^2(M),\]\[\Lambda_{\Ad}'(x)\otimes \Lambda_{\Ad}'(y)  \mapsto (\Lambda_{\Ad}'\otimes \hat{\Lambda}')(\widetilde{\alpha}^{\circ}(y)(x\otimes 1)),\qquad x,y\in \mathscr{N}_{\varphi_{\widetilde{\alpha}^{\circ}}}.\] 

Consider the following maps $Z_1,Z_2,Z_3$, which are clearly well-defined and isometric: 
\[Z_1: L^2(N)\otimes L^2(N) \otimes L^2(N)\rightarrow L^2(N) \otimes L^2(N) \otimes L^2(M),\]\[\Lambda_{B(L^2(N))}(x) \otimes \Lambda_{\Ad}'(y) \mapsto (\Lambda_{B(L^2(N))}\otimes \hat{\Lambda}')((\pi_{\rtimes}\otimes \id)(\hat{\alpha}(\pi_{\rtimes}^{-1}(x)))(y\otimes 1)),\]

\[Z_2: L^2(N) \otimes L^2(M)\otimes L^2(N) \rightarrow L^2(N)\otimes L^2(N)\otimes L^2(M),\]\[\Lambda_N(x) \otimes \hat{\Lambda}'(y)\otimes \Lambda_{\Ad}'(z) \mapsto (\Lambda_{B(L^2(N))}\otimes \hat{\Lambda}')((\hat{\pi}'\otimes \id)(\hat{\Delta}'(y))_{23}(xz\otimes 1)),\]

\[Z_3: L^2(N) \otimes L^2(N) \otimes L^2(M) \rightarrow L^2(M)\otimes L^2(N)\otimes L^2(N),\]\[\Lambda_{B(L^2(N))}(x)\otimes \hat{\Lambda}'(y) \mapsto ( \Lambda_{N\rtimes M}\otimes \hat{\Lambda}')(\hat{\Delta}'(y)_{23}(\pi_{\rtimes}^{-1}(x)\otimes 1)).\]

Then it is easily verified that one has the following commuting squares, where we write $Z_0 =  1\otimes \Ww_{\widetilde{\alpha}^{\circ}}^*$,

\begin{equation}\label{EqnComm1} \xymatrix{
\ar[d]_{\Xi_{12}\Sigma_{23}}L^2(N)\otimes L^2(N)\otimes L^2(N) \ar[r]^{Z_0}  &
L^2(N)\otimes L^2(N)\otimes L^2(M) \ar[d]^{\Xi_{12}} \\
L^2(N)\otimes L^2(N)\otimes L^2(N) \ar[r]_{Z_1} &
L^2(N)\otimes L^2(N)\otimes L^2(M)}
\end{equation}
\begin{equation}\label{EqnComm2} \xymatrix{
L^2(N)\otimes L^2(N)\otimes L^2(N)\ar[r]^{Z_1}  &
L^2(N)\otimes L^2(N)\otimes L^2(M)  \\
L^2(N)\otimes L^2(M)\otimes L^2(N) \ar[u]^{\Sigma_{12}\Ww_{12}}\ar[r]_{Z_2} & L^2(N) \otimes L^2(N)\otimes L^2(M)\ar[u]_{\id}}
\end{equation}

\begin{equation}\label{EqnComm3} \xymatrix{
\ar[d]_{\Xi_{12}\Sigma_{23}}L^2(N)\otimes L^2(M)\otimes L^2(N)\ar[r]^{Z_2}  &
L^2(N)\otimes L^2(N)\otimes L^2(M) \ar[d]^{\Ww^*_{12}\Sigma_{12}} \\
L^2(N)\otimes L^2(N)\otimes L^2(M) \ar[r]_{Z_3} & L^2(N) \otimes L^2(M)\otimes L^2(M) }
\end{equation}

\begin{equation}\label{EqnComm4} \xymatrix{
L^2(N)\otimes L^2(N)\otimes L^2(M)\ar[r]^{Z_3}  &
L^2(N)\otimes L^2(M)\otimes L^2(M)  \\\ar[u]^{\Sigma_{12}\Ww_{12}}
L^2(N)\otimes L^2(M)\otimes L^2(M) \ar[r]_{Z_4}& L^2(N) \otimes L^2(M)\otimes L^2(M) \ar[u]_{\id}}
\end{equation}

where $Z_4 = 1 \otimes ((\hat{J}\otimes \hat{J})\hat{W}^*(\hat{J}\otimes \hat{J}))$. 

Combining the commuting squares \eqref{EqnComm1}, \eqref{EqnComm2}, \eqref{EqnComm3} and \eqref{EqnComm4} together, clearing away $\Sigma_{23}$ and using that $\Xi_{12}$ and $\Ww_{13}^*$ commute since the first leg of the latter lies in $N$, we find 
\begin{equation}\label{EqExpralti} 1\otimes \Ww_{\widetilde{\alpha}^{\circ}}^* = (\Sigma \Xi)_{12}^* \Ww_{12} ((\hat{J}\otimes \hat{J})\hat{W}^*(\hat{J}\otimes \hat{J}))_{23}\Ww_{12}^*\Ww_{23}^*(\Sigma \Xi)_{12}(\Sigma \Xi)_{13}\end{equation} Recall now \eqref{EqWadaltild}. Then multiplying \eqref{EqExpralti} to the left with $J_N\otimes J_N\otimes \hat{J}$ and to the right with $J_N\otimes J_N\otimes J_N$, the left hand side turns into $1\otimes \Ww_{\Ad}^*$. On the other hand,  using that \[\Xi = (J_N\otimes J_N)(\Sigma \Xi)^*(J_N\otimes J_N)\] by a small computation, and using that the second leg of $\hat{\Vv}$ intertwines the standard representation of $\hat{M}'$ with $\hat{\pi}'$ by \eqref{EqGalImpCross2}, we find from \eqref{EqGalImpCross4} that \[1\otimes
\Ww_{\Ad}^* = \Xi_{12}((\hat{\pi}'\otimes \id)((J\hat{J}\otimes 1))\hat{W}^*(\hat{J}J\otimes \hat{J}J))\hat{\Vv}^*)_{23}\Xi_{12}\Xi_{13}.\] Now using that $\Xi_{12}\Xi_{13} = \Xi_{23}\Xi_{12}$, we obtain by taking adjoints and moving $\Xi_{12}$ to the other side that  \[ \Xi_{12}(1\otimes \Ww_{\Ad})\Xi_{12} = 1 \otimes (\Xi\hat{\Vv}((\hat{\pi}'\otimes \id)((J\hat{J}\otimes J\hat{J})\hat{W}(\hat{J}J\otimes 1))).\] But since the first leg of $\Ww_{\Ad}$ lies in $N$, and $1\otimes N$ commutes with $\Xi$, we obtain the expression for $\Ww_{\Ad}$ in \eqref{TheoCompCalcGalAd}.
\end{proof}

\begin{Theorem}\label{TheoMain} If $(N,\alpha)$ is a $I$-factorial Galois object, then $\Ad_{\Ad_{\alpha}} = \alpha$. 
\end{Theorem} 
\begin{proof}  We are to prove that \[\Ww_{\Ad}\alpha(x) = (x\otimes 1)\Ww_{\Ad},\qquad x\in N.\]
Since both legs of $\Xi$ commute with $N$, and since the first leg of $\hat{\Vv}$ commutes with $N$, it is by Theorem \ref{TheoCompCalcGalAd} sufficient to prove that for $x\in N$ one has \begin{multline*}(\hat{\pi}'\otimes \id)((J\hat{J}\otimes 1)\hat{W}(\hat{J}J\otimes 1))\alpha(x) \\= (x\otimes 1)(\hat{\pi}'\otimes \id)((J\hat{J}\otimes 1)\hat{W}(\hat{J}J\otimes 1)).\end{multline*}
As any $x\in N$ can be approximated $\sigma$-weakly by elements of the form $(\id\otimes \omega)(\Ww^*)$ with $\omega \in B(L^2(M),L^2(N))_*$, and since $(\alpha\otimes \id)(\Ww) = \Ww_{13}W_{23}$, we thus have to prove the identity \begin{multline*} ((\hat{\pi}'\otimes \id)((J\hat{J}\otimes 1)\hat{W}(\hat{J}J\otimes 1)))_{12}W_{23}^*\Ww_{13}^* \\ = \Ww_{13}^*((\hat{\pi}'\otimes \id)((J\hat{J}\otimes 1)\hat{W}(\hat{J}J\otimes 1)))_{12}.\end{multline*} However, this is equivalent with \begin{multline*} (\hat{\Delta}'\otimes \id)((J\hat{J}\otimes J\hat{J})\hat{W}(\hat{J}J\otimes \hat{J}J)) \\ = (J\hat{J}\otimes J\hat{J} \otimes J\hat{J})\hat{W}_{23}\hat{W}_{12} (\hat{J}J \otimes \hat{J}J \otimes \hat{J}J),\end{multline*} which follows by \eqref{EqGalImpCross3} and the fact that $\hat{V}' = (J\hat{J}\otimes J\hat{J})W(\hat{J}J\otimes \hat{J}J)$ is the right regular corepresentation for $\hat{M}'$. 

\end{proof} 

The following theorem relates the invariant weights on a $I$-factorial Galois object $(N,\alpha)$ and its dual. 

\begin{Theorem}\label{TheoWeight} Let $(N,\alpha)$ be a $I$-factorial right Galois object for $(M,\Delta)$, say $N = B(\Hsp)$ for a Hilbert space $\Hsp$. Let $h$ be the unique positive (unbounded, invertible) operator on $\Hsp$ such that the left invariant weight $\varphi_N$ for $\alpha$ is given by \[\varphi_N(x) = \Tr(h^{1/2}xh^{1/2}),\quad \forall x\in N^+.\] Then the left invariant weight $\varphi_{\Ad}$ for $(N,\Ad_{\alpha})$ is given by \[\varphi_{\Ad}(x) = \Tr(h^{-1/2}xh^{-1/2}),\qquad x\in N^+.\]  
\end{Theorem}

\begin{proof} Let as before \[\varphi_{N\rtimes M}: (N\rtimes_{\alpha} M)^+ \rightarrow [0,+\infty],\quad x\mapsto \varphi_N \circ \alpha^{-1}\circ (\id\otimes \id\otimes \hat{\varphi}')\hat{\alpha}(x)\] be the dual weight of $\varphi_N$, and denote as in \eqref
{EqWeightB} and \eqref{EqFormWeightB} \[\varphi_{B(L^2(N))} = \varphi_{N\rtimes M}\circ \pi_{\rtimes}^{-1} = \Tr(\nabla_N^{1/2}\,\cdot\,\nabla_N^{1/2}).\] With $h$ as in the statement of the theorem, we have however \[\nabla_N^{it} = h^{it}\otimes \overline{h^{it}}  = h^{it}\otimes(\overline{h})^{-it},\] where we write $\overline{x}\overline{\xi} = \overline{x\xi}$ for $x\in B(\Hsp)$ and $\overline{\xi}\in \overline{\Hsp}$ the conjugate of $\xi\in \Hsp$. Hence on $B(\Hsp)\vNtimes B(\overline{\Hsp}) \cong N \vNtimes N' \cong B(L^2(N))$, the above weight can be expressed as \[\varphi_{B(L^2(N))} = \Tr(h^{1/2}\,\cdot\,h^{1/2}) \otimes \Tr((\overline{h})^{-1/2}\,\cdot\,(\overline{h})^{-1/2}).\]

Now by the proof of Theorem \ref{TheoAdjGal} we have that \[\Ad_{\alpha}(x) = (J_N\otimes \hat{J})\widetilde{\alpha}^{\circ}(J_NxJ_N)(J_N\otimes \hat{J}),\] where \[\widetilde{\alpha}^{\circ}(x) = (\pi_{\rtimes}\otimes \id)\alpha^{\circ}(\pi_{\rtimes}^{-1}(x)),\qquad x\in N'.\] As $\alpha^{\circ}$ is the restriction of $\hat{\alpha}$ to the factor $ \alpha(N)'\cap N\rtimes M \cong N'\cong B(\overline{\Hsp})$ which splits of, it  follows immediately from the above  that the left invariant weight for $\widetilde{\alpha}^{\circ}$ must be \[x \mapsto \Tr((\overline{h})^{-1/2} x (\overline{h})^{-1/2}),\qquad x\in B(\overline{\Hsp}).\] From the above form for $\Ad_{\alpha}$, we then deduce that \[\varphi_{\Ad} = \Tr(h^{-1/2}\,\cdot\,h^{-1/2}).\] 
\end{proof}

Note that the inversion $h \rightarrow h^{-1}$ in the above theorem is not unexpected: if we rescale $\varphi \rightarrow \lambda\varphi$ for $\lambda>0$, then the dual weight gets rescaled in the inverse way, $\hat{\varphi} \rightarrow \lambda^{-1}\hat{\varphi}$. 

For our application, we will need the left-handed version of Theorem \ref{TheoAdjGal} and Theorem \ref{TheoWeight}. Recall from \eqref{EqVv} the definition of the Galois unitary $\Vv$ for \emph{left} Galois objects. 

\begin{Theorem}\label{TheoAdGalLeft} Let $(N,\gamma)$ be a left $I$-factorial Galois object with respect to $(M,\Delta)$. Then $(N,\Ad_{\gamma})$ is a right $I$-factorial Galois object with respect to $(\hat{M}',\hat{\Delta}')$, where \[\Ad_{\gamma}(x) = \Sigma \Vv(1\otimes x)\Vv^*\Sigma.\] Moreover, if $N= B(\Hsp)$ with the right invariant weight $\psi_N$ for $\gamma$ given by \[\psi_N(x) = \Tr(k^{1/2}xk^{1/2}),\quad \forall x\in N^+\] for some unbounded positive operator $k$ on $\Hsp$, then the left invariant weight $\varphi_{\Ad}$ for $(N,\Ad_{\gamma})$ is given by \[\varphi_{\Ad}(x) = \Tr(k^{-1/2}xk^{-1/2}),\qquad x\in N^+.\]  
\end{Theorem}

\section{Example: Heisenberg double}

Let $(M,\Delta)$ be a quantum group von Neumann algebra, and consider the tensor product quantum group von Neumann algebra $\tilde{M} = \hat{M}\vNtimes M$ with comultiplication \[\tilde{\Delta}(x) =  \varsigma_{23}(\hat{\Delta}\otimes \Delta)(x).\] Then one obtains a right coaction of $\tilde{M}$ on $\tilde{N} := B(L^2(M))$ by \[\tilde{\alpha}(x)  = \hat{V}_{12}V_{13} (x\otimes 1\otimes 1)V_{13}^*\hat{V}_{12}^*.\] Indeed, since $V\in \hat{M}'\vNtimes M$ and $\hat{V}\in M^{\prime}\vNtimes \hat{M}$, the map $\tilde{\alpha}$ restricts to the coaction $\Delta$ of $M$ on $M$ and the coaction $\hat{\Delta}$ of $\hat{M}$ on $\hat{M}$. Since $M$ and $\hat{M}$ generate $B(L^2(M))$ (see for example \cite[Proposition 2.5]{VV03}), this is sufficient to conclude that $\tilde{\alpha}$ is well-defined. This is a generalization of the \emph{Heisenberg algebra} for the Cartesian product of an abelian compact group with its Pontryagin dual. In general we call $(\tilde{N},\tilde{\alpha})$ the \emph{Heisenberg double} of $(M,\Delta)$. 

\begin{Prop} The coaction $(\tilde{N},\tilde{\alpha})$ is a $I$-factorial Galois object.
\end{Prop}

\begin{proof} The $\sigma$-weak closure of the first leg of $\hat{V}_{12}V_{13}$ contains the subspace $M' \hat{M}' = J\hat{J}M\hat{M}\hat{J}J$, and is hence $\sigma$-weakly dense in $B(L^2(M))$. As any coinvariant element in $B(L^2(M))$ commutes with the first leg of $\hat{V}_{12}V_{13}$, the coaction $\tilde{\alpha}$ is ergodic.

We deduce easily that $\tilde{\alpha}$ is integrable, as all elements of the form $y^*x^*xy$ for $x\in \mathscr{N}_{\varphi}$ and $y \in \mathscr{N}_{\hat{\varphi}}$ are integrable. 

Now the crossed product by $\tilde{\alpha}$ is generated by $\Delta(M)_{13}, 1\otimes 1\otimes \hat{M}', \hat{\Delta}(\hat{M})_{12}$ and $1\otimes M^{\prime}\otimes 1$. Applying $\Ad(\Sigma_{23}\Sigma_{13}W_{13})$, we obtain that it is isomorphic to the von Neumann algebra generated by \[(M\otimes 1\otimes 1)\cup (\hat{M}'\otimes 1\otimes 1)\cup (\hat{\Delta}\otimes \id)\hat{\Delta}(\hat{M}) \cup 1 \otimes 1\otimes M'.\] But $M\hat{M}' = \hat{J}M \hat{M}\hat{J}$ is $\sigma$-weakly dense in $B(L^2(M))$. In particular, we can throw in another copy of $\hat{M}\otimes 1\otimes 1$ in the first leg. Using coassociativity of $\hat{\Delta}$ and the fact that $(\hat{M}\otimes 1)\hat{\Delta}(\hat{M})$ is $\sigma$-weakly dense in $\hat{M}\vNtimes \hat{M}$, we obtain that the above von Neumann algebra is the same as the one generated by  \[B(L^2(M))\otimes 1\otimes 1 \cup 1\otimes \hat{\Delta}(\hat{M}) \cup 1 \otimes 1\otimes M'.\] Applying $\Ad(\Sigma_{23}\hat{W}_{23})$, we see that this becomes an isomorphic copy of the von Neumann algebra $B(L^2(M))\vNtimes B(L^2(M))$. In other words, the crossed product  $\tilde{N}\rtimes \tilde{M}$ is a type $I$-factor, which is sufficient to conclude that $(\tilde{N},\tilde{\alpha})$ is a Galois object since then \eqref{EqRepNcrossM} must necessarily be faithful. 
\end{proof}

Recall \cite[Proposition 2.1]{KV03} that $M$ can be endowed with a \emph{one-parameter scaling group} $\tau_t$, determined by the formula \[\tau_t(x) = \hat{\nabla}^{it} x\hat{\nabla}^{-it},\qquad x\in M,\] where $\hat{\nabla}$ is the modular operator of $\hat{M}$. 

\begin{Prop}\label{PropWeightHeis} The invariant weights on $\tilde{N}$ are given by \[\varphi_{\tilde{N}}(x) = \Tr(h^{1/2}xh^{1/2}),\qquad \psi_{\tilde{N}}(x) = \Tr(k^{1/2}xk^{1/2}),\] where $h,k$ are the positive invertible operators such that \[h^{it} = \nu^{it^2/2}\nabla^{it}J\delta^{it}J,\quad k^{it} = P^{-it},\] with $\nabla$ the modular operator of $\varphi$ and $P^{it}$ determined by  \[P^{it}\Lambda(x) = \nu^{t/2}\Lambda(\tau_t(x)),\qquad x\in \mathscr{N}_{\varphi}.\]
\end{Prop} 

\begin{proof} The formula for $\varphi_{\tilde{N}}$ follows from \cite[Proposition 2.8 and Proposition 2.9]{VV03}. By \cite[Theorem 4.19]{DeC11}, $\psi_{\tilde{N}} = \varphi_{\tilde{N}}(\delta_{\tilde{N}}^{1/2}\,\cdot\,\delta_{\tilde{N}}^{1/2})$ with the positive operator $\delta_{\tilde{N}}$ determined up to a positive scalar by the fact that \[\tilde{\alpha}(\delta_{\tilde{N}}^{it}) = \delta_{\tilde{N}}^{it}\otimes \hat{\delta}^{it}\otimes \delta^{it}.\] It follows that we can take \[\delta_{\tilde{N}}^{it} = \nu^{it^2/2} \hat{\delta}^{it}\delta^{it},\] and the form for $\psi_{\tilde{N}}$ now follows from the remarks above \cite[Proposition 2.10]{VV03}.
\end{proof} 

Let us now compute the Galois unitary associated to $(\tilde{N},\tilde{\alpha})$. In fact, following the discussion after \cite[Proposition 2.9]{VV03}, we may realize $\Lambda_{\tilde{N}}$ as the unique GNS-map having the linear span of $\mathscr{N}_{\hat{\varphi}}\mathscr{N}_{\varphi}$ as its core and on which \[\Lambda_{\tilde{N}}(xy) = \hat{\Lambda}(x)\otimes \Lambda(y).\] The corresponding GNS-representation of $B(L^2(M))$ is by \[x \rightarrow V(x\otimes 1)V^*,\] which identifies $B(L^2(M))$ with $M\ltimes_{\Delta}M$. Taking this presentation, we can represent the Galois unitary $\tilde{\Ww}$ of $(\tilde{N},\tilde{\alpha})$ as a unitary operator on $L^2(M)^{\otimes 4}$. 

\begin{Prop} The Galois unitary of $(\tilde{N},\tilde{\alpha})$ is given by 
\[\tilde{\Ww} = W_{14}W_{24} W_{31}^*.\]
\end{Prop} 
\begin{proof} Using KSGNS-maps of the form $\id\otimes \Lambda$ and identities of the form \[(\id\otimes \Lambda)(\Delta(y)) = W^*(\id\otimes \Lambda(y)),\qquad y\in \mathscr{N}_{\varphi},\] we compute for $x\in \mathscr{N}_{\hat{\varphi}}$, $y\in \mathscr{N}_{\varphi}$ and $\xi,\eta\in L^2(M)\otimes L^2(M)$ that \begin{eqnarray*} \hspace{-0.6cm} && \tilde{\Ww}^*((\xi\otimes \eta) \otimes \Lambda_{\tilde{N}}(xy))\\ && = ((\id\otimes \id\otimes \hat{\Lambda})(\hat{\Delta}(x)_{13})\otimes \id)(\id\otimes \id\otimes \Lambda_M)((\Delta\otimes \id)\Delta(y))(\xi\otimes \eta) \\ && = ((\id\otimes \id\otimes \hat{\Lambda})(\hat{\Delta}(x)_{13})\otimes \id) W_{23}^*W_{13}^* (\xi\otimes \eta \otimes \Lambda(y)) \\ &&= \hat{W}_{13}^*W_{24}^* W_{14}^* (\xi\otimes \eta\otimes \hat{\Lambda}(x)\otimes \Lambda(y)) \\ &&= \hat{W}_{13}^*W_{24}^* W_{14}^* ((\xi\otimes \eta)\otimes \Lambda_{\widetilde{N}}(xy)).
\end{eqnarray*} 
\end{proof}

To have an expression for $\tilde{\Ww} \in \tilde{N}\vNtimes B(L^2(\tilde{M}),L^2(\tilde{N}))$ with the first leg in its ordinary representation on $B(L^2(M))$, note that \begin{multline*} V_{12}^*W_{14}W_{24}W_{31}^*V_{12} = V_{12}^*(\Delta\otimes \id)(W)_{124}W_{31}^*V_{12}\\ = W_{14}V_{12}^*W_{31}^*V_{12} = W_{14}W_{31}^*,\end{multline*}  
so we obtain the natural expression \[\tilde{\Ww} = W_{13}\hat{W}_{12} \in \tilde{N}\vNtimes B(L^2(\tilde{M}),L^2(\tilde{N})).\]

Identifying the dual of $\widetilde{M}$ with $M\vNtimes \hat{M}$ in the natural way, we get the following corollary. 

\begin{Cor} The adjoint coaction $\Ad_{\tilde{\alpha}}$ of $M\vNtimes \hat{M}$ on $B(L^2(M))$ is given by \[x \mapsto \hat{W}_{12}^*W_{13}^*(x\otimes 1\otimes 1)W_{13}\hat{W}_{12}.\] 
\end{Cor} 

Using that \[(\id\otimes \varsigma)\tilde{\alpha}(x) = V_{12}\hat{V}_{13}(x\otimes 1\otimes 1)\hat{V}_{13}^*V_{12}^*,\] together with the fact that \[\hat{W}_{12}^*W_{13}^* = (\hat{J}J\otimes 1\otimes 1)V_{12}\hat{V}_{13}(J\hat{J}\otimes 1\otimes 1),\] we deduce that \[\Ad_{\tilde{\alpha}}(x) = (\Ad(\hat{J}J)\otimes \varsigma_{23})\tilde{\alpha}(\Ad(J\hat{J})(x)),\] so that identifying $M\vNtimes \hat{M}\cong \hat{M}\vNtimes M$ via the flip map, $\Ad_{\tilde{\alpha}}$ is just an isomorphic copy of $\tilde{\alpha}$ itself. Note that this is compatible with Theorem \ref{TheoWeight} as, with $h$ as in Proposition \ref{PropWeightHeis}, we have $\hat{J}Jh J\hat{J} = h^{-1}$ by an easy verification using the known commutation relations between the modular data, see \cite[Proposition 2.4]{VV03}.

\section{Quantized universal enveloping algebras and their quantum function algebra duals}

\subsection{Quantized univeral enveloping algebras}

Let $\mfg$ be a semisimple complex Lie algebra of rank $l$, and let \[\mfg = \mfk\oplus \mfa\oplus \mfn\] be an Iwasawa decomposition of $\mfg$ as a real Lie algebra. Let $\mft = i\mfa$ be a maximal Cartan subalgebra of $\mfk$, and write \[\mfb = \mft \oplus \mfa \oplus \mfn = \mfh \oplus \mfn\] for the associated (complex) Borel subalgebra, where $\mfh = \mfa \oplus \mft$ is a complex Cartan subalgebra for $\mfg$. We let $(\,\cdot\,,\,\cdot\,)$ denote a fixed positive multiple of the form on $\mfh^*$ induced from the Killing form on $\mfg$, chosen such that the shortest roots have length $\sqrt{2}$. We denote by $\Delta$ the root system, by $\Delta^+$ the positive roots, by $\{\alpha_r\mid r\in I\}$ the simple roots in $\Delta^+$, and by $\check{\alpha} = \frac{2}{(\alpha,\alpha)}\alpha$ the associated coroots. We denote by  \[Q\subseteq P \subseteq \mfh^*\] respectively the root and weight lattice, and by $P^+$ the positive cone of the weight lattice. We denote by $\varpi_r$ the fundamental weights, and we set \[\rho = \frac{1}{2}\sum_{\alpha\in \Delta^+}\alpha = \sum_{r\in I} \varpi_r.\] We further denote by $W$ the Weyl group of $\mfg$, and by $w_0$ the longest element in $W$ with respect to the length function determined by the reflections across the simple roots. We consider the natural action of $W$ on $\mfh$ and $\mfh^*$, where the simple reflections $s_{\alpha}$ act by \[s_{\alpha}v = v - (\check{\alpha},v)\alpha.\] We also use as a shorthand \[s_i = s_{\alpha_i}.\]

Let us recall the definition of the Hopf algebra $U_q(\mfg)$. Here we will fix our deformation parameter $q$ to be a real number between 0 and 1, as this is what will be needed for the remainder of this paper. 
 
We will use the standard notation for $q$-numbers,  \[[n]_q! = [n]_q[n-1]_q\ldots [2]_q[1]_q,\qquad [n]_q = \frac{q^{-n}-q^n}{q^{-1}-q}.\] We will also use the following \emph{quantum Serre polynomials} $f_{rs}$ in two non-commuting variables $Y,Z$ for $r\neq s$ elements of $I$, \begin{equation}\label{EqQSerre}f_{rs}(Y,Z) = \sum_{p=0}^{1-(\check{\alpha}_r,\alpha_s)}(-1)^p \frac{[1-(\check{\alpha}_r,\alpha_s)]_{q_r}!}{[p]_{q_r}![1-(\check{\alpha}_r,\alpha_s)-p]_{q_r}!}Y^{1-(\check{\alpha}_r,\alpha_s)-p}ZY^{p} = 0,\end{equation} where \[q_r = q^{\frac{(\alpha_r,\alpha_r)}{2}}.\]

For the following definition and the ensuing elementary properties we refer for example to \cite[Section 6]{KS97}.

\begin{Def}\label{DefQUE} We define $U_q(\mfg)$ to be the universal algebra over $\C$ generated by elements $K_{\omega}$ for $\omega \in P$ and elements $E_r,F_r$ for $r\in I$ such that $K_0$ is the unit and
\begin{equation}\label{KE} K_{\omega}K_{\chi} = K_{\omega+\chi},\quad K_{\omega} E_r  = q^{(\omega,\alpha_r)}E_rK_{\omega},\quad K_{\omega}F_r = q^{-(\omega,\alpha_r)}F_rK_{\omega},\end{equation}
\begin{equation}\label{SE} f_{rs}(E_r,E_s) = 0,\quad  f_{rs}(F_r,F_s) = 0,\qquad r\neq s, \end{equation}
\begin{equation}\label{EF} E_rF_s -F_sE_r = \delta_{rs} \frac{K_{\alpha_r}- K_{\alpha_r}^{-1}}{q_r-q_r^{-1}}.
\end{equation}
We endow $U_q(\mfg)$ with the coproduct determined by $\hat{\Delta}(K_{\omega}) = K_{\omega}\otimes K_{\omega}$ and 
\begin{equation}\label{CoBPlus}\hat{\Delta}(E_r) = E_r\otimes K_{\alpha_r}+ 1\otimes E_r,\end{equation}\begin{equation}\label{CoBMin} \hat{\Delta}(F_r) = F_r\otimes 1 + K_{\alpha_r}^{-1}\otimes F_r.\end{equation}
We denote by $\check{U}_q(\mfg)$ the Hopf subalgebra generated by the $E_r,F_r$ and $K_{\omega}$ with $\omega \in Q$. 
\end{Def}

We will denote the associated counit by $\hat{\varepsilon}$ and the associated antipode by $\hat{S}$, so that \[\hat{\varepsilon}(K_{\omega}) = 1,\quad \hat{\varepsilon}(E_r)=\hat{\varepsilon}(F_r) = 0,\]\[ \hat{S}(K_{\omega}) = K_{\omega}^{-1},\quad \hat{S}(E_r) = -E_rK_{\alpha_r}^{-1},\quad \hat{S}(F_r) =- K_{\alpha_r}F_r.\]

We can associate to $U_q(\mfg)$ various related Hopf algebras. 
\begin{Def} We define \[U_q(\mfh) = \textrm{ linear span of }\{K_{\omega}\mid \omega\in P\} \subseteq U_q(\mfg),\]
\[U_q(\mfn) = \textrm{unital algebra generated by }\{E_r\mid r\in I\} \subseteq U_q(\mfg),\]
\[U_q(\mfb) =   \textrm{algebra generated by }U_q(\mfn)\textrm{ and }U_q(\mfh),\]
\[U_q(\mfn^-) = \textrm{unital algebra generated by }\{F_r\mid \omega\in r\in I\} \subseteq U_q(\mfg),\]
\[U_q(\mfb^-) =   \textrm{algebra generated by }U_q(\mfn^-)\textrm{ and }U_q(\mfh).\]
\end{Def}

The $U_q(\mfh)$, $U_q(\mfb)$ and $U_q(\mfb^-)$ are  Hopf subalgebras.  The algebras $U_q(\mfh)$, $U_q(\mfn)$, $U_q(\mfb)$, $U_q(\mfn^-)$ and $U_q(\mfb^-)$ are generated universally by the relations of $U_q(\mfg)$ involving only the generators entering their definition. Furthermore, we have the \emph{triangular decomposition}, i.e. the bijectivity of the multiplication map \[U_q(\mfn^-) \otimes U_q(\mfh) \otimes U_q(\mfn) \rightarrow U_q(\mfg).\]

The algebra $U_q(\mfg)$ can be endowed with various $*$-algebra structures  corresponding to the real forms of $\mfg$. We require here that $U_q(\mfg)$ becomes a Hopf $*$-algebra, i.e.~ that the comultiplication is a $*$-homomorphism. This will be satisfied for example by the \emph{compact} real form, which is the only real form we will be interested in. 

\begin{Def} We denote by $U_q(\mfk)$ the Hopf algebra $U_q(\mfg)$ endowed with the Hopf $*$-algebra structure \[K_{\omega}^* = K_{\omega},\quad E_r^* = F_rK_{\alpha_r},\quad F_r^* = K_{\alpha_r}^{-1}E_r.\]
\end{Def} 

We then denote \[U_q(\mft) \subseteq U_q(\mfk)\] for the Hopf $*$-algebra generated by the $K_{\omega}$, so that $U_q(\mft)$ is a particular $*$-structure on $U_q(\mfh)$.

We will at some point need to refer to the \emph{unitary antipode} of $U_q(\mfk)$, which is a `correction' of the antipode as to be compatible with the $*$-structure. This nomenclature is in analogy with the terminology for quantum group von Neumann algebras.

\begin{Def}\label{DefUnAnt} We define the \emph{unitary antipode} $\hat{R}$ of $U_q(\mfk)$ to be the unique anti-homomorphism such that \[\hat{R}(K_{\omega}) = K_{\omega}^{-1},\quad \hat{R}(E_r) = -q_r^{-1}E_rK_{\alpha_r}^{-1},\quad \hat{R}(F_r) = -q_rK_{\alpha_r}F_r.\]  
\end{Def}

It is easy to check that $\hat{R}$ is involutive and $*$-preserving with \[\hat{\Delta} \circ \hat{R} = (\hat{R}\otimes \hat{R})\circ \hat{\Delta}^{\opp}.\]

 Note also that $\hat{R}$ leaves $U_q(\mfb)$ and $U_q(\mfb^-)$ invariant, and we will hence also refer to the unitary antipode  of $U_q(\mfb)$ and $U_q(\mfb^-)$.

\subsection{Representation theory of quantized universal enveloping algebras}

We recall some facts concerning the representation theory of $U_q(\mfk)$, see again for example \cite{KS97}.  

\begin{Def} A \emph{type $I$-representation} of $U_q(\mfk)$ is a unital $*$-representation of $U_q(\mfk)$ on a finite-dimensional Hilbert space such that the $K_{\omega}$ have positive eigenvalues. 
\end{Def}

One can concretely describe type $I$-representations as follows. For $\omega \in P$, we denote by $\C_{\omega}$ the one-dimensional representation of $U_q(\mfb)$ obtained from the character \[K_{\chi}\mapsto q^{(\chi,\omega)},\qquad E_r\mapsto 0.\] For $\varpi \in P^+$, we denote by \[M_{\varpi} = U_q(\mfg) \underset{U_q(\mfb)}{\otimes}\C_{\varpi}\] the associated Verma module for $U_q(\mfg)$, and by $V_{\varpi}$ its unique irreducible quotient. We denote by $\xi_{\varpi}$ the highest weight vector in $V_{\varpi}$ obtained as the image of $1\otimes 1$. We further endow $V_{\varpi}$ with the unique Hilbert space structure for which $\xi_{\varpi}$ is a unit vector and for which $V_{\varpi}$ becomes a $*$-representation of $U_q(\mfk)$. If $\xi \in V_{\varpi}$ is a joint eigenvector for the $K_{\omega}$, we write $\wt(\xi)\in P$ for the unique weight with \[K_{\omega}\xi = q^{(\omega,\wt(\xi))}\xi,\qquad \forall \omega\in P.\] 

\begin{Theorem} The $V_{\varpi}$ form a complete set, up to unitary equivalence, of irreducible type $I$-representations of $U_q(\mfk)$, and any type $I$-representation of $U_q(\mfk)$ decomposes as a direct sum of $V_{\varpi}$'s. 
\end{Theorem}

For $\mfk =\mfsu(2)$ with single positive root $\alpha$ and single  fundamental weight $\varpi$, we write the generators of $\check{U}_q(\mfsu(2))$ as $E,F,K = K_{\alpha}$. In this case we identify \[Q \cong 2\Z,\quad P = \Z,\quad P^+ = \N,\] with $P$ endowed with the inner product \[(r,s) = rs/2.\] Then the highest weight module $V_{N}$ for $N\in \N$ has an orthonormal basis \[\{\xi_{N},\xi_{N-2},\ldots, \xi_{-N+2},\xi_{-N}\}\] on which $K \xi_n = q^{n} \xi_n$ and \[(q^{-1}-q)E\xi_n = q^{-\frac{N-n-1}{2}} \left((1-q^{N-n})(1-q^{N+n+2)})\right)^{1/2} \xi_{n+2},\]\[(q^{-1}-q)F \xi_n = q^{-\frac{N+n+1}{2}}  \left((1-q^{N-n+2)})(1-q^{N+n})\right)^{1/2} \xi_{n-2}.\]

Returning to the general case, we denote for $r\in I$ \[T_r\xi = \underset{-a+b-c = (\wt(\xi),\check{\alpha}_r)}{\sum_{a,b,c\in \N}} \frac{(-1)^bq_r^{b-ac}}{[a]_{q_r}! [b]_{q_r}![c]_{q_r}!}E_r^{a}F_r^bE_r^c\xi \]  the Lusztig braid operators on the module $V_{\varpi}$, see for example \cite[Section 8]{Jan96}. For $\mathbf{r} = (r_1,\ldots,r_n)$ a sequence in $I$, we denote $T_{\mathbf{r}} = T_{r_1}\ldots T_{r_n}$. Writing again \[w_0 = s_{r_1}\ldots s_{r_M},\] let then \[T_{w_0} = T_{(r_1,\ldots,r_M)}.\] This operator is independent of the chosen decomposition of $w_0$.

For $\mfk = \mfsu(2)$, it can be computed \cite[Section 8.3]{Jan96} that in $V_N$ one has \[T\xi_N = (-1)^{N} q^{N/2} \xi_{-N}.\] Hence using that \[\Delta^+  = \{s_{r_M}\ldots s_{r_{n+1}}\alpha_{r_n}\mid 1\leq n \leq M\}\] is in one-to-one correspondence with $\Delta^+$, we obtain for general $\mfk$ that \[\eta_{w_0\varpi} := q^{-(\rho,\varpi)}T_{w_0}\xi_{\varpi},\] is a unit lowest weight vector of $V_{\varpi}$ at weight $w_0\varpi$.

\subsection{Dual Hopf $*$-algebra}

We assume still that $\mfg$ is a complex semisimple Lie algebra with compact real form $\mfk$, and we use notation as before. Let $G$ be the connected, simply connected complex Lie group with Lie algebra $\mfg$, and let $K,T,A,H,N,\overline{N}$ be the connected Lie groups inside $G$ integrating respectively $\mfk,\mft,\mfa,\mfh,\mfn,\mfn^-$. 

\begin{Def} We define $\Pol_q(G)$ to be the Hopf subalgebra of the restricted Hopf dual of $U_q(\mfg)$ consisting of the matrix coefficients $U_{\pi}(\xi,\eta)$ for $\xi,\eta \in V_{\pi}$, where $(V,\pi)$ are type $I$-representations of $U_q(\mfk)$ and where \[U_{\pi}(\xi,\eta)(x) = \langle \xi,\pi(x)\eta\rangle,\qquad x \in U_q(\mfg).\] 
\end{Def}

For $\varpi \in P^+$, we will denote the associated matrix coefficients by $U_{\varpi}(\xi,\eta)$. It is easily seen by semisimplicity that the $U_{\varpi,ij} = U_{\varpi}(e_i,e_j)$ form a basis of $\Pol_q(G)$, with $\varpi$ running over $P^+$ and $e_i,e_j$ over an orthonormal basis of $V_{\varpi}$. The coproduct on $\Pol_q(G)$ is given by \[\Delta(U_{\varpi}(\xi,\eta)) = \sum_i U_{\varpi}(\xi,e_i)\otimes U_{\varpi}(e_i,\eta).\]

It will be convenient to extend the $P$-valued grading on the $V_{\varpi}$ to $\Pol_q(G)$. 
 For $x  = U_{\varpi}(\xi,\eta)$, we denote \[\lwt(x) = \wt(\xi),\qquad \rwt(x) = \wt(\eta).\] In this way, $\Pol_q(G)$ becomes a $P$-bigraded algebra.

The Hopf algebra $\Pol_q(G)$ can be made into a Hopf $*$-algebra as follows. Let the \emph{contragredient} representation of $V_{\varpi}$ be defined as the conjugate linear Hilbert space $\overline{V_{\varpi}}$, equipped with the $*$-representation \[x\overline{\xi} = \overline{\hat{R}(x)^*\xi},\quad a\in U_q(\mfk),\] where $\hat{R}$ is the unitary antipode defined in Definition \ref{DefUnAnt}. There is then a natural identification of $*$-representations \[\overline{V_{\varpi}} \cong V_{-w_0\varpi},\quad x\overline{\xi_{\varpi}}\mapsto x\eta_{-\varpi},\quad x\in U_q(\mfk).\] We can now endow $\Pol_q(G)$ with the $*$-structure  \[U_{\varpi}(\xi,\eta)^* = q^{-(\rho,\wt(\xi)-\wt(\eta))} U_{-w_0\varpi}(\overline{\xi},\overline{\eta}).\] It is easily seen that this is well-defined, and that  we have the formula \[U_{\varpi}(\xi,\eta)^*(x) = \langle \eta,\hat{S}(x)\xi\rangle,\quad x\in \Pol_q(G),\xi,\eta\in V_{\varpi}.\] In particular, the last formula easily implies that we obtain a Hopf $*$-algebra.

\begin{Def} We define $\Pol_q(K)$ to be the Hopf algebra $\Pol_q(G)$ endowed with the above $*$-algebra structure.
\end{Def}

One then immediately gets also a formula for the antipode by \[S(U_{\varpi}(\xi,\eta)) = U_{\varpi}(\eta,\xi)^*.\] 

We will need certain special elements in $\Pol_q(K)$.

\begin{Def} We define for $\varpi\in P^+$ the following elements in $\Pol_q(K)$, \begin{equation}\label{DefEla} b_{\varpi} := U_{\varpi}(\xi_{\varpi},\eta_{w_0\varpi}),\qquad \varpi\in P_+.\end{equation}
\end{Def}

By our convention for the lowest weight vectors $\eta_{w_0\varpi}$, they satisfy \[b_{\varpi}b_{\varpi'} = b_{\varpi+\varpi'}.\] One can moreover show, using the specific properties of the universal $R$-matrix, that these elements are normal, \[b_{\varpi}b_{\varpi'}^* = b_{\varpi'}^*b_{\varpi},\] and that more generally one has for $x\in \Pol_q(K)$ the commutation relations \[b_{\varpi}x = q^{(\varpi,\lwt(x)-w_0\rwt(x))} x b_{\varpi}\]\[b_{\varpi}^*x = q^{(\varpi,\lwt(x)-w_0\rwt(x))} x b_{\varpi}^*,\] cf. \cite[Proposition 3.4.1]{LS91}.

We will need to localize our algebra at these elements $b_{\varpi}$, at the same time including also a polar decomposition for them. 

\begin{Def} We define $\Pol_q(K)_{\loc}$ to be the $*$-algebra generated by $\Pol_q(K)$ and elements $|b|_{\omega},u_{\omega}$ for $\omega\in P$ such that the following relations hold: the $u_{\omega}$ are central with $u_{\omega}^* = u_{-\omega}$ and \[u_0 = 1,\quad u_{\omega}u_{\chi} = u_{\omega+\chi},\qquad \omega,\chi\in P,\] while the $|b|_{\omega}$ satisfy $|b|_{\omega}^* = |b|_{\omega}$ and \[|b|_0 = 1,\quad |b|_{\omega}|b|_{\chi} = |b|_{\omega+\chi},\qquad \omega,\chi\in P\] and \[|b|_{\omega}x = q^{(\omega,\lwt(x)-w_0\rwt(x))} x |b|_{\omega},\qquad x\in \Pol_q(K),\omega\in P.\] Moreover, for all $\varpi\in P^+$ we impose \[b_{\varpi} = u_{\varpi}|b|_{\varpi}.\] 
\end{Def} 

We then write of course also $b_{\omega} = u_{\omega}|b|_{\omega}$ for a general $\omega\in P$. It is easy to see that we then have an embedding of unital $*$-algebras $\Pol_q(K) \subseteq \Pol_q(K)_{\loc}$, as $\Pol_q(K)$ is a domain \cite[Lemma 9.1.9]{Jos95}. Moreover, we can extend the $P$-bigradation on $\Pol_q(K)$ to a unique $P$-bigraded algebra structure on $\Pol_q(K)_{\loc}$ such that \[\lwt(|b|_{\omega}) = \rwt(|b|_{\omega}) = 0,\quad \lwt(u_{\omega}) = \omega,\quad   \rwt(u_{\omega}) = w_0\omega.\]

\section{Amplified quantized enveloping algebras and their quantum function algebra duals}

\subsection{Amplified quantized enveloping algebras}

It is well-known that quantized universal enveloping algebras are a quotient by killing one copy of the Cartan part in the Drinfeld double of $U_q(\mfb)$ with $U_q(\mfb^-)$ with respect to a natural skew pairing. For our purposes, we will need to work with the full Drinfeld double. Let us recall some general formalism, as it will be convenient for later purposes. 

\begin{Def}\label{DefTwoCocy} Let $H = (H,\Delta,\varepsilon,S)$ be a Hopf algebra. A \emph{(normalized) 2-cocycle functional} for $H$ is a convolution invertible  functional \[\omega: H\times H \rightarrow \C\] such that for all $g,h,k\in H$ \[\omega(g_{(1)}, h_{(1)})\omega(g_{(2)}h_{(2)}, k) = \omega(h_{(1)},k_{(1)}) \omega(g, h_{(2)}k_{(2)}),\]  and with $\omega(g,1) = \varepsilon(g) = \omega(1,g)$ for $g\in H$. 

For $\omega,\chi$ two $2$-cocycle functionals on $H$, we define ${}_{\omega}H_{\chi^{-1}}$ to be the vector space $H$ with the new product \[h\cdot g = \omega(h_{(1)}, g_{(1)})h_{(2)}g_{(2)} \chi^{-1}(h_{(3)}, g_{(3)}).\]
\end{Def} 

It is well-known and easy to see that ${}_{\omega}H_{\chi^{-1}}$ is again a unital associative algebra (with the same unit as $H$). For $\omega = \chi$, it can be proven that ${}_{\omega}H_{\omega^{-1}}$ is a Hopf algebra for the original coproduct on $H$, see e.g.~ \cite[Section 2.3]{Maj95}.

Examples of $2$-cocycles can be obtained from \emph{skew pairings} between Hopf algebras. For $H,K$ Hopf algebras, let us call a bilinear pairing \[(\,\cdot\,,\,\cdot\,): H\times K \rightarrow \C\] a \emph{skew pairing} if \[(1,x) = \varepsilon(x),\quad (x,1) = \varepsilon(x),\]\[ (xy,z) = (x,z_{(1)})(y,z_{(2)}),\quad (x,yz) = (x_{(2)},y)(x_{(1)},z).\] Then the tensor product Hopf algebra $L=H\otimes K$ can be endowed with the two-cocycle \[\omega: L\times L \rightarrow \C,\quad (x\otimes y,w\otimes z) \mapsto (w,y).\] The resulting Hopf algebra ${}_{\omega}L_{\omega^{-1}}$ is known as the \emph{Drinfeld double}. It is universally generated by copies of the Hopf algebras $H$ and of $K$ such that the following \emph{interchange relation} holds: \[yx = (x_{(1)},y_{(1)}) x_{(2)}y_{(2)} (S(x_{(3)}),y_{(3)}),\qquad x\in H,y\in K.\] 

Note now that $U_q(\mfb)$ and $U_q(\mfb^-)$, with the generator $K_{\omega}$ written as $L_{\omega}$, can be skew-paired by means of $(K_{\omega},F_r) = 0 = (E_r,L_{\omega})$ and \begin{equation}\label{EqPairing} (K_{\omega},L_{\chi}) = q^{-(\omega,\chi)},\quad (E_r,F_s) = \frac{\delta_{rs}}{q_r^{-1}-q_r},\end{equation} cf. \cite[Proposition 6.34]{KS97}. We then obtain the following concrete description of the Drinfeld double, see e.g.~ \cite[Proposition 8.13]{KS97}. We will at the same time endow this Hopf algebra with a particular $*$-structure.

\begin{Def} We define $U_q^{+}(\mfb_{\R})$ to be the Hopf $*$-algebra generated by a copy of the Hopf algebras $U_q(\mfb)$ and  $U_q(\mfb^-)$  such that $E_r^* = F_rL_{\alpha_r}$, $K_{\omega}^* = L_{\omega}$ and such that the following interchange relations holds: 
\[K_{\omega}L_{\chi} =  L_{\chi}K_{\omega},\]
\[K_{\omega} F_r = q^{-(\omega,\alpha_r)}F_rK_{\omega},\qquad L_{\omega}E_r = q^{(\omega,\alpha_r)}E_rL_{\omega}\]\[E_rF_s - F_sE_r = \delta_{rs} \frac{K_{\alpha_r} - L_{\alpha_r}^{-1}}{q_r-q_r^{-1}}.\] 
\end{Def} 

It is easy to see that the above $*$-structure is compatible with the coproduct and with the interchange relations between $U_q(\mfb)$ and $U_q(\mfb^-)$. 

Let us denote \[U_q(\mfa) := \textrm{ group }*\textrm{-algebra of }P,\] and let us write the generators of $U_q(\mfa)$ as $U_{\omega}$. Then we have natural Hopf $*$-algebra homomorphisms \[\pi_{\mfk}: U_q^{+}(\mfb_{\R}) \rightarrow U_q(\mfk),\quad E_r\mapsto E_r, \quad F_r\mapsto F_r,\quad K_{\omega}\mapsto K_{\omega},\quad L_{\omega}\rightarrow K_{\omega}\] \[\pi_{\mfa}:U_q^+(\mfb_{\R})\rightarrow U_q(\mfa),\quad E_r,F_r\mapsto 0,\quad K_{\omega}\mapsto U_{\omega},\quad L_{\omega}\rightarrow U_{-\omega}.\] 

This allows us to make an embedding of $*$-algebras \begin{equation}\label{EqEmbbR} \pi_{\mfk\oplus \mfa}: U_q^+(\mfb_{\R})\rightarrow U_q(\mfk)\otimes U_q(\mfa),\quad X\mapsto \pi_{\mfk}(X_{(1)})\otimes \pi_{\mfa}(X_{(2)}),\end{equation}  which is however not a Hopf algebra homomorphism.

\subsection{Amplified quantum function algebras}

We now treat the corresponding extension on the level of $\Pol_q(K)$, which will result in a \emph{multiplier} Hopf $*$-algebra \cite{VDae94}. Let us write $\Pol_q(K^n) = \Pol_q(K)^{\otimes n}$, and extend the left and right $P$-gradings on $\Pol_q(K)$ to left and right $P^n$-gradings on $\Pol_q(K^n)$. Let further $\Func(P^n,\C)$ be the $*$-algebra of complex-valued functions on $P^n$ with \[f^*(\omega) = \overline{f(\omega)},\] and let $\Func_{\compact}(P^n,\C)$ be the $*$-subalgebra of finitely supported functions (the symbol $\compact$ stands for `compact'). For $f\in  \Func(P^n,\C)$ and $\omega\in P^n$, we will write $f_{\omega}$ for the shift of $f$ by $\omega$, \[f_{\omega}(\chi) = f(\chi-\omega).\]

\begin{Def}  We define $\Fun_{q}^+((AK)^n)$ to be the unital $*$-algebra generated by a copy of $\Pol_q(K^n)$ and $\Func(P^n,\C)$ such that we have the exchange relations \[fx = xf_{\lwt(x)-\rwt(x)}.\] We define $\Fun_{q,\compact}^+((AK)^n)$ to be the (non-unital) $*$-subalgebra spanned linearly by elements of the form $xf$ for $f\in \Func_{q,\compact}(P^n,\C)$ and $x\in \Pol_q(K^n)$.  
\end{Def}

For $n= 1$, we simply write $\Fun_q^+(AK)$ and $\Fun_{q,\compact}^+(AK)$. Note that as a vector space, we can also identify \[\Fun_q^+((AK)^n) =   \Pol_q(K^n)\otimes \Func(P^n,\C),\]\[ \Fun_{q,\compact}^+((AK)^n) =   \Pol_q(K^n)\otimes \Func_{\compact}(P^n,\C)\] by the multiplication map, and that \[\Fun_{q,\compact}^+((AK)^n) = \Fun_{q,\compact}^+(AK)^{\otimes n}.\]

Note that $\Fun_{q,\compact}^+((AK)^n)$ is a faithful left and right $\Fun_q^+((AK)^n)$-module for the multiplication. Hence we can embed \[\Fun_q^+((AK)^n)\subseteq M(\Fun_{q,\compact}^+((AK)^n)),\] where the right hand side denotes the \emph{multiplier $*$-algebra} of $\Fun_{q,\compact}^+((AK)^n)$.  

We can now introduce on $\Fun_{q,\compact}^+(AK)$ the structure of a multiplier Hopf $*$-algebra  \cite{VDae94}. Using the formulation as in \cite[Proposition 2.9]{VDae98}, a multiplier Hopf $*$-algebra $(H,\Delta)$ is a not necessarily unital $*$-algebra $H$, together with a (non-degenerate) coassociative $*$-homomorphism \[\Delta: H\rightarrow M(H\otimes H)\] satisfying the condition that \[\Delta(H)(1\otimes H)\cup (H\otimes 1)\Delta(H) \subseteq H\otimes H,\] and admitting a counit $\varepsilon$ and antipode $S$. By applying the $*$-structure, also \[\Delta(H)(H\otimes 1) \cup (1\otimes H)\Delta(H) \subseteq H\otimes H,\] and $S$ is an invertible map $H \rightarrow H$. For multiplier Hopf $*$-algebras, we will also use the Sweedler notation \[\Delta(x) = x_{(1)}\otimes x_{(2)},\] but one has to be more careful with this notation as this element is now not necessarily a finite sum of simple tensors -- it is only such after multiplying one of the legs with an element from $H$. 

For example, using the identification \[\Func(P^2,\C) = M(\Func_{\compact}(P^2,\C)),\] we can interpret $\Func_{\compact}(P,\C)$ as a multiplier Hopf $*$-algebra by \[\Delta(f)(\omega,\chi) = f(\omega+ \chi),\quad \varepsilon(f) = f(0),\quad S(f)(\omega) = f(-\omega).\] Note that these formulas also makes sense for a general $f\in \Func(P,\C)$.

\begin{Lem} There exists a unique $*$-homomorphism
\[\Delta: \Fun_q^+(AK) \rightarrow \Fun_q^+((AK)^2)\subseteq M(\Fun_{q,\compact}^+(AK)\otimes \Fun_{q,\compact}^+(AK))\] such that $\Delta$ coincides with the respective comultiplications on $\Pol_q(K)$ and $\Func(P)$. Moreover, the restriction of $\Delta$ to $\Fun_{q,\compact}^+(AK)$ turns the latter into multiplier Hopf $*$-algebra with counit and antipode determined by \[\varepsilon(xf) = \varepsilon(x)\varepsilon(f),\quad S(xf) = S(f)S(x).\] 
\end{Lem} 

\begin{proof} The existence of $\Delta,\varepsilon,S$ as (anti-)$*$-homomorphisms is immediately checked from the definining relations for $\Fun_q^+(AK)$. Verifying the multiplier Hopf $*$-algebra conditions is then also straightforward.
\end{proof} 

\begin{Rem} The multiplier Hopf $*$-algebra $\Fun_{q,\compact}^+(AK)$ can also be seen abstractly as the Drinfel'd double of $\Pol_q(K)$ with $\Func_{\compact}(P)$ in the sense of \cite{DrD01} with respect to the degenerate pairing \[(x,f) = (\pi_{\mfa}(x),f),\qquad x\in \Pol_q(K),f\in \Func_{\compact}(P),\] where $(U_{\omega},f) = f(\omega)$. 
\end{Rem} 

For $X\in U_q(\mfk)^{\otimes n}$ and $\omega\in P^n$, let us denote by $X_{\omega}$ the functionals \[X_{\omega}: \Fun_q^+((AK)^n)\rightarrow \C,\quad xf \mapsto (X_{\omega},xf) = (X,x)f(\omega),\]  where the latter $(\,\cdot\,,\,\cdot\,)$ denotes the natural pairing between $\Pol_q(K^n)$ and $U_q(\mfk)^{\otimes n}$ given by \[(\otimes X_i, \otimes U_{\pi_i}(\xi_i,\eta_i)) = \prod \langle \xi_i,\pi_i(X_i)\eta_i\rangle.\] Let $\widetilde{U}_q(\mfk\oplus \mfa)^{\otimes n}$ be the vector space spanned by the $X_{\omega}$, which is indeed identifiable in the obvious way with the $n$-fold tensor product of $\widetilde{U}_q(\mfk\oplus \mfa) = \widetilde{U}_q(\mfk\oplus \mfa)^{\otimes 1}$ by \[\otimes_i X_{i,\omega_i} = (\otimes X_i)_{(\omega_1,\ldots,\omega_n)}.\]  Then $\widetilde{U}_q(\mfk\oplus \mfa)$ comes equipped with a unital $*$-algebra structure determined by \[(XY,x) = (X\otimes Y,\Delta(x)),\quad (X^*,x) = \overline{(X,S(x)^*)},\] where $X,Y\in \widetilde{U}_q(\mfk\oplus \mfa),x\in \Fun_q^+(AK)$.

\begin{Prop} The assignment  \[X_{\omega}\mapsto X\otimes U_{\omega}\] is an isomorphism of $*$-algebras between $\widetilde{U}_q(\mfk\oplus \mfa)$ and $U_q(\mfk)\otimes U_q(\mfa)$. 
\end{Prop}
\begin{proof}The fact that this is a vector space isomorphism is obvious. To see that we have a $*$-algebra isomorphism, it is enough to prove that $X_{\omega}Y_{\chi} = (XY)_{\omega+\chi}$. But \[(X_{\omega}Y_{\chi})(xf) = (X\otimes Y,\Delta(x))f(\omega+\chi) = (XY,x)f(\omega+\chi) = ((XY)_{\omega+\chi},xf).\] Similarly, to have that the above preserves the $*$-operation, we have to show that $(X_{\omega})^* = (X^*)_{-\omega}$, but \[((X_{\omega})^*,xf) = \overline{(X,S(x)^*)}f(-\omega) = (X^*,x)f(-\omega) = ((X^*)_{-\omega},xf).\]
\end{proof}

By means of \eqref{EqEmbbR}, we hence obtain an embedding \begin{equation}\label{EqPairUP} U_q^+(\mfb_{\R}) \rightarrow \widetilde{U}_q(\mfk\oplus \mfa),\end{equation}\[E_r \mapsto (E_r)_{\alpha_r},\quad F_r \mapsto (F_r)_0,\quad K_{\omega}\mapsto (K_{\omega})_{\omega},\quad L_{\omega}\mapsto (K_{\omega})_{-\omega},\] where the right hand side parenthesized elements lie in $U_q(\mfk)$. By this embedding, we obtain a pairing between $U_q^+(\mfb_{\R})$ and $\Fun_{q,\compact}^+(AK)$.

\begin{Prop} Under the above pairing, we have that \[(\hat{\Delta}(X),x\otimes y) = (X,xy),\qquad X\in U_q^+(\mfb_{\R}), x,y \in \Fun_q^+(AK).\]
\end{Prop} 

\begin{proof} As we already know that $U_q^{+}(\mfb_{\R})$ is embedded inside the dual of $\Fun_{q,\compact}^+(AK)$ as a $*$-algebra, it is enough to verify the above formula for the generators $E_r,K_{\omega}$. For $K_{\omega}$ we compute, using that $K_{\omega}$ (as an element of $U_q(\mfk)$) vanishes on bi-homogeneous elements $x\in \Pol_q(K)$ with $\lwt(x)\neq \rwt(x)$, \begin{multline*} (K_{\omega},xfyg) = ((K_{\omega})_{\omega},xyf_{\lwt(y)-\rwt(y)}g)\\ =  f(\omega-\lwt(y)+\rwt(y))g(\omega) (K_{\omega},xy) \\  =  f(\omega)g(\omega)(K_{\omega},x)(K_{\omega},y) = (K_{\omega},xf)(K_{\omega},yg).\end{multline*} 

For $E_r$, we use that $E_{r}$ (as an element of $U_q(\mfk)$) vanishes on bi-homogeneous elements $x\in \Pol_q(K)$ with $\lwt(x)\neq \rwt(x)+\alpha_r$ to find 
 \begin{eqnarray*} (E_r, xfyg) &=& ((E_r)_{\alpha_r},xyf_{\lwt(y)-\rwt(y)}g) \\
 &=& f(\alpha_r-\lwt(y)+\rwt(y))g(\alpha_r) (E_r,xy) \\ &=&   f(\alpha_r-\lwt(y)+\rwt(y))g(\alpha_r)  (E_r,x)(K_{\alpha_r},y) \\ &&\qquad +  f(\alpha_r-\lwt(y)+\rwt(y))g(\alpha_r) (1,x)(E_r,y) \\ &=& f(\alpha_r)g(\alpha_r) (E_r,x)(K_{\alpha_r},y) +  f(0)g(\alpha_r)  (1,x)(E_r,y) \\ &=& (E_r\otimes K_{\alpha_r} + 1\otimes E_r,xf\otimes yg).
\end{eqnarray*}
\end{proof} 

Note that we can also slice elements in $\Fun_q^+((AK)^2)$ from the left and right by elements in $U_q^+(\mfb_{\R})$: there exist for $X\in U_q^+(\mfb_{\R})$ and $x \in \Fun_q^+((AK)^2)$ unique elements $(\id\otimes X)\Delta(x)$ and $(X\otimes \id)\Delta(x)$ in $\Fun_q^+(AK)$ such that for all $Y\in U_q^+(\mfb_{\R})$ one has \[(Y, (X\otimes \id)\Delta(x)) = ((X\otimes Y),\Delta(x)),\quad (Y,(\id\otimes X)\Delta(x)) = (Y\otimes X,\Delta(x)).\] We can then make $\Fun_q^+(AK)$ into a left and right $U_q^+(\mfb_{\R})$-module $*$-algebra by \[X\rhd x = (\id\otimes X)\Delta(x),\qquad x\lhd X = (X\otimes \id)\Delta(x),\] by which is meant that \[X\rhd (xy) = (X_{(1)}\rhd x)(X_{(2)} \rhd y),\qquad (X\rhd x)^* = \hat{S}(X)^* \rhd x^*,\]\[xy \lhd X = (x\lhd X_{(1)})(y\lhd X_{(2)}),\qquad (x\lhd X)^* = x^*\lhd \hat{S}(X)^*.\]
This is easily seen to restrict to $U_q^+(\mfb_{\R})$-module $*$-algebra structures on $\Fun_{q,\compact}^+(AK)$.

\section{Galois objects for amplified quantized enveloping algebras and their duals}

\subsection{Hopf algebraic Galois objects}

We introduce first the notion of Galois object within the setting of Hopf algebras. This is the purely algebraic analogue of Definition \ref{DefGalObvN}.

\begin{Def}\cite[Definition 2.1.1]{Sch04} Let $(H,\Delta)$ be a Hopf algebra. We call a right $(H,\Delta)$-comodule algebra $(A,\alpha)$ a \emph{(right) Galois object} if the map \[\mathrm{\can}:A\otimes A \rightarrow A\otimes H,\quad x\otimes y \mapsto (x\otimes 1)\alpha(y)\] is a linear bijection. If $(H,\Delta)$ is a Hopf $*$-algebra, we also assume that $(A,\alpha)$ is a right comodule $*$-algebra.
\end{Def}

We then use the following Sweedler notation for the inverse map $\can^{-1}$ restricted to $H$: \[\can^{-1}(1\otimes h) = h^{[1]}\otimes h^{[2]}.\] The following lemma introduces an important module structure associated with a Galois object. It corresponds to the adjoint coaction of Definition \ref{DefAdj} in the analytic setting. 

\begin{Lem} Let $(A,\alpha)$ be a right $(H,\Delta)$-Galois object. Then $A$ is a right $H$-module algebra with respect to the \emph{adjoint action} (also known as \emph{Miyashita-Ulbrich action}) \[a \blhd h = h^{[1]}ah^{[2]}.\] If moreover $(H,\Delta)$ is a Hopf $*$-algebra, then $\blhd$ defines a module $*$-algebra. 
\end{Lem} 
\begin{proof} The fact that $A$ is a right $H$-module algebra follows from the relations in \cite[Lemma 2.1.7]{Sch04}. To see that it interacts correctly with the $*$-structure, it is sufficient to prove that, for $h\in H$, \[h^{[2]*} \otimes h^{[1]*} = (S(h)^*)^{[1]}\otimes (S(h)^*)^{[2]}.\] But applying $\can$ and the $*$-operation to both sides, this becomes \[\alpha(h^{[1]})(h^{[2]}\otimes 1) = 1\otimes S(h),\] which follows from (2.1.3) and (2.1.4) in  \cite[Lemma 2.1.7]{Sch04}.
\end{proof}

\begin{Rem} Since the adjoint module is inner, it descends to any $*$-algebra into which $A$ admits a $*$-homomorphism: if \[\pi: A\rightarrow B,\] then we have on $B$ the right $H$-module $*$-algebra structure \[x \blhd h = \pi(h^{[1]})x\pi(h^{[2]}).\]
\end{Rem}

Specific examples of Galois objects can be obtained from $2$-cocycle functionals. The following lemma is well-known. 

\begin{Lem} If $(H,\Delta)$ is a Hopf algebra, and $\omega,\chi$ are $2$-cocycle functionals on $H$, then ${}_{\chi}H_{\omega^{-1}}$ is a Galois object for ${}_{\omega}H_{\omega^{-1}}$, where the coaction $\alpha$ is given by  the original comultiplication on $H$. 
\end{Lem}
\begin{proof} It is checked directly that $\alpha$ indeed defines a comodule algebra structure. With $K= {}_{\omega}H_{\omega^{-1}}$, we then have ${}_{\chi}H_{\omega^{-1}} = {}_{\chi\omega^{-1}}K$. We may hence reduce to the case of $\omega = \varepsilon\otimes \varepsilon$. But in this case, one verifies easily that $\can$ is bijective, with $\can^{-1} = F_2\circ F_1$ where \[F_1(x\otimes y) = xS(y_{(1)})\otimes y_{(2)},\quad F_2(x\otimes y) = \chi^{-1}(x_{(1)},y_{(1)})x_{(2)}\otimes y_{(2)},\] where in the formula for $F_1$ the original multiplication in $H$ is used. Cf.~ also \cite[Theorem 2.2.7]{Sch04} and the discussion preceding it. 
\end{proof}

In particular, if $\omega$ is induced from a skew pairing between Hopf algebras $H$ and $K$ as in the discussion following Definition \ref{DefTwoCocy}, the associated cocycle twist ${}_{\omega}L$ of $L = H\otimes K$ is called the \emph{Heisenberg double} of $H$ and $K$. On the other hand, we also have that $L_{\omega^{-1}}$ is a right Galois object for the Drinfeld double ${}_{\omega}L_{\omega^{-1}}$, which we will call the \emph{opposite Heisenberg double}. We can hence use the skew pairing \eqref{EqPairing} to construct the opposite Heisenberg double for $U_q(\mfb)$ and $U_q(\mfb^-)$. Unfortunately, this algebra will not have the right spectral properties for us, and we will rather need to consider ${}_{\chi}(U_q(\mfb)\otimes U_q(\mfb^-))_{\omega^{-1}}$ where $\chi$ is the 2-cocycle obtained from the (degenerate) pairing \begin{equation}\label{EqPairingDeg} (K_{\omega},L_{\chi})_0 = q^{(\omega,\chi)},\quad (K_{\omega},F_r)_0 = (E_r,L_{\omega})_0  = (E_r,F_s)_0 = 0.\end{equation}

This results in the following algebra, which we can again endow with an appropriate $*$-structure.

\begin{Def}\label{DefHeisOmChi} We define $U_q^{0}(\mfb_{\R})$ to be the $*$-algebra generated by a copy of $U_q(\mfb)$  and a copy of $U_q(\mfb^-)$ such that $K_{\omega}^* = L_{\omega}$, $E_r^* = F_rL_{\alpha_r}$ and \[K_{\omega}L_{\chi} = q^{-2(\omega,\chi)} L_{\chi}K_{\omega},\]
\[L_{\omega} E_r  L_{\omega}^{-1}= q^{(\omega,\alpha_r)}E_r,\quad K_{\omega} F_r K_{\omega}^{-1}= q^{(\omega,\alpha_r)}F_r,\]\[E_r F_s - F_s E_r = \delta_{rs} \frac{- L_{\alpha_r}^{-1}}{q_r-q_r^{-1}}.\] We endow it with the comodule $*$-algebra structure $\alpha$ which restricts to $\hat{\Delta}$ on the components $U_q(\mfb)$ and $U_q(\mfb^-)$.
\end{Def} 

It is easy to see that for the Galois object $(U_q^0(\mfb_{\R}),\alpha)$, the adjoint action is simply given by \[X \blhd Y = S(Y_{(1)})XY_{(2)},\qquad X\in U_q^0(\mfb_{\R}), Y \in U_q(\mfb) \cup U_q(\mfb^-) \subseteq U_q^{+}(\mfb_{\R}).\] 

Related to $U_q^{0}(\mfb_{\R})$ we have the following versions involving only $\mfn$ or $\mfh$.

\begin{Def}\label{DefUqn}  We define $U_q^{0}(\mfn_{\R})$ to be the universal $*$-algebra generated by elements $X_{r},X_{r}^*$ such that the $X_r$ and $X_r^*$ satisfy the quantum Serre relations and such that \[X_rX_s^* - q^{-(\alpha_r,\alpha_s)}X_s^*X_r = \delta_{rs}\frac{-1}{q_r-q_r^{-1}}.\] We define $U_q^0(\mfh_{\R})$ to be the universal $*$-algebra generated by elements $K_{\omega}$ and $L_{\omega} = K_{\omega}^*$ such that $K_0 = 1$, $K_{\omega}K_{\chi} = K_{\omega+ \chi}$ and \[K_{\omega}L_{\chi} = q^{-2(\omega,\chi)}L_{\chi}K_{\omega}.\] 
\end{Def}

The following lemma follows immediately from the fact that \[U_q^0(\mfb_{\R}) \cong U_q(\mfh) \otimes U_q(\mfn) \otimes U_q(\mfn^-) \otimes U_q(\mfh^-)\] as vector spaces by the multiplication map, where we denote $U_q(\mfh^-)$ for the copy  of $U_q(\mfh)$ generated by the $L_{\omega}$.  

\begin{Lem} We have  embeddings of unital $*$-algebras \[U_q^{0}(\mfn_{\R})  \rightarrow U_q^{0}(\mfb_{\R}),\quad X_r \mapsto E_r,\]\[U_q^{0}(\mfh_{\R}) \rightarrow U_q^0(\mfb_{\R}),\quad K_{\omega}\mapsto K_{\omega}.\]
\end{Lem}

It will be convenient to have a slightly more elaborate version of this last lemma. For this, we first enlarge $U_q^0(\mfb_{\R})$. 

\begin{Def}\label{DefUtildb} We let $\widetilde{U}_q^0(\mfh_{\R})$ be the unital $*$-algebra generated by unitary elements $U_{\omega}$ for $\omega\in P$ with $U_0 = 1$, $U_{\omega}U_{\chi} = U_{\omega+\chi}$ and self-adjoint elements $Z_{\omega}$ for $\omega \in P$ with $Z_0 = 1$, $Z_{\omega}Z_{\chi} = Z_{\omega+\chi}$ such that \[U_{\omega}Z_{\chi}U_{\omega}^* = q^{(\omega,\chi)}Z_{\chi}.\]

We let $\widetilde{U}_q^0(\mfb_{\R})$ be the unital $*$-algebra generated by a copy of $U_q^0(\mfb_{\R})$ and $\widetilde{U}_q(\mfh_{\R})$ such that  
\[K_{\omega} = q^{-\frac{1}{2}(\omega,\omega)}Z_{\omega}U_{-\omega},\quad L_{\omega} = q^{-\frac{1}{2}(\omega,\omega)}U_{\omega}Z_{\omega},\] \[U_{\omega}E_r U_{\omega}^* = E_r,\quad U_{\omega}F_rU_{\omega}^* = q^{-(\omega,\alpha_r)}F_r,\]\[Z_{\omega} E_r = q^{(\omega,\alpha_r)}E_rZ_{\omega},\quad Z_{\omega}F_r = F_rZ_{\omega}.\]
\end{Def}

It is easily seen that this `adjoining of polar decompositions of the $K_{\omega}$' is compatible with the other relations, and that we have an embedding of $*$-algebras \[U_q^0(\mfb_{\R}) \rightarrow \widetilde{U}_q^0(\mfb_{\R}).\] 

\begin{Lem}\label{Lemtildb} The following map is a well-defined isomorphism of $*$-algebras: \begin{equation}\label{EqIsotildExt} \widetilde{U}_q(\mfh_{\R})\otimes U_q(\mfn_{\R}) \rightarrow \widetilde{U}_q(\mfb_{\R}),\quad U_{\omega}\mapsto U_{\omega},\quad Z_{\omega}\mapsto Z_{\omega},\quad X_r \mapsto U_{\alpha_r}E_r.\end{equation}
\end{Lem} 
\begin{proof} By direct computation using the defining relations. 
\end{proof}

Note that as we have an inclusion $U_q^0(\mfb_{\R}) \subseteq \widetilde{U}_q^0(\mfb_{\R})$, we can extend the adjoint $U_q^+(\mfb_{\R})$-action $\blhd$ to $\widetilde{U}_q^0(\mfb_{\R})$.

\subsection{Galois objects for multiplier Hopf algebras}

We will also need the notion of Galois object for multiplier Hopf $*$-algebras. Recall first that for $(H,\Delta)$ a multiplier Hopf $*$-algebra, a \emph{(left) coaction} of $(H,\Delta)$ on a (not necessarily unital) $*$-algebra $A$ is a $*$-homomorphism \[\gamma: A \rightarrow M(H\otimes A)\] such that \begin{equation}\label{EqGamNotDeg}(H\otimes 1)\gamma(A) = H\otimes A\end{equation}  and such that \[(\Delta\otimes \id)\gamma = (\id\otimes \gamma)\gamma.\] This last condition can be made sense of by the former condition. In fact, as our algebras are assumed to be idempotent with non-degenerate product, the maps $\Delta\otimes \id$ and $\id\otimes \gamma$ can be uniquely extended to $M(H\otimes A)$, giving another way to interpret the above coaction property.

\begin{Def} Let $(H,\Delta)$ be a multiplier Hopf $*$-algebra. A \emph{(left) Galois object} for $(H,\Delta)$ consists of a (non-unital) $*$-algebra $A$, together with a coaction such that $\gamma(x)(1\otimes y) \in H\otimes A$ for all $x,y\in A$ and such that the resulting map \[A\otimes A \rightarrow H\otimes A,\quad x\otimes y \mapsto \gamma(x)(1\otimes y)\] is a linear bijection.
\end{Def}

As an example, we construct a Galois object for $\Fun_{q,\compact}^+(AK)$.

\begin{Def} We define $\Fun_{q}^{0}(AK)$ to be the $*$-algebra generated by a copy of $\Pol_q(K)$ and $\Func(P,\C)$ with the interchange relation \[fy = yf_{\lwt(y)}.\] More generally, we define $\Fun_{q}^{0}((AK)^n)$ to be the $*$-algebra generated by a copy of $\Pol_q(K^n)$ and $\Func(P^n,\C)$ with the interchange relation \[fy = yf_{\lwt(y)-\eta\rwt(y)},\] where $\eta\omega : =(\omega_1,\ldots,\omega_{n-1},0)$.

We denote $\Fun_{q,\compact}^{0}((AK)^n)$ for the $*$-subalgebra generated by the elements of the form $xf$ with $x\in \Pol_q(K^n)$ and $f\in \Func_{\compact}(P^n,\C)$. 
\end{Def}

It is again easy to see that we have an embedding \[\Fun_{q}^{0}((AK)^n)\subseteq M(\Fun_{q,\compact}^{0}((AK)^n)),\] and that \[\Fun_{q,\compact}^{0}((AK)^n) =  \left(\otimes_{i=1}^{n-1} \Fun_{q,\compact}^+(AK)\right) \otimes \Fun_{q,\compact}^0(AK).\]

\begin{Lem} There exists a unital $*$-homomorphism \[\gamma: \Fun_q^{0}(AK) \rightarrow \Fun_q^{0}((AK)^2)\subseteq M(\Fun_{q,\compact}^+(AK)\otimes \Fun_{q,\compact}^0(AK)),\]\[ xf \mapsto \Delta(x)\Delta(f).\] Moreover, the restriction of $\gamma$ to $\Fun_{q,\compact}^{0}(AK)$ turns $\Fun_{q,\compact}^{0}(AK)$ into a left Galois object for $\Fun_{q,\compact}^+(AK)$. 
\end{Lem} 

\begin{proof} The fact that $\gamma$ is a well-defined $*$-homomorphism is immediate, as is the coaction property. As the multiplication maps give vector space isomorphisms \[\Pol_q(K)\otimes \Func_{\compact}(P) \cong \Fun_{q,\compact}^{\epsilon}(AK)\cong \Func_{\compact}(P)\otimes \Pol_q(K)\] for $\epsilon \in \{0,+\}$, and as the maps \[x\otimes y \mapsto (x\otimes 1)\Delta(y),\quad x\otimes y  \mapsto \Delta(x)(1\otimes y)\] are linear isomorphisms as maps $\Pol_q(K)^{\otimes 2} \rightarrow \Pol_q(K)^{\otimes 2}$ or $\Func_{\compact}(P)^{\otimes 2}\rightarrow \Func_{\compact}(P)^{\otimes 2}$, it follows that $\gamma$ satisfies \eqref{EqGamNotDeg} and that the associated Galois map is bijective. 

\end{proof}

By means of the coaction $\gamma$ and the pairing \eqref{EqPairUP}, the $*$-algebra $\Fun_{q}^{0}(AK)$ can be endowed with a right $U_q^+(\mfb_{\R})$-module $*$-algebra structure by \[xf \lhd F  = (F\otimes \id)\gamma(xf),\] which restricts to a module $*$-algebra structure on $\Fun_{q,\compact}^{0}(AK)$.

For further reference, we also build the associated amplifications of $\Pol_q(K)_{\loc}$ we will need. 

\begin{Def} We define $\Fun_q^{0}(AK)_{\loc}$ to be the $*$-algebra generated by a copy of $\Pol_q(K)_{\loc}$ and $\Func(P,\C)$ with the interchange relation \[fx = xf_{\lwt(x)}.\]
\end{Def}

It is clear that we then have an embedding of $*$-algebras 
\[\Fun_{q,\compact}^{0}(AK)\subseteq \Fun_q^{0}(AK) \subseteq \Fun_q^{0}(AK)_{\loc}.\]

\subsection{Correspondence between translation and adjoint actions}

In this section, we will show that the right $U_q^{+}(\mfb_{\R})$-modules $(\Fun_{q,\compact}^{0}(AK),\lhd)$ (equipped with the translation action) and $(U_q^{0}(\mfb_{\R}),\blhd)$ (equipped with the adjoint action) are closely related. 

Let us first introduce the following elements in $\Func(P,\C)\subseteq \Fun_q^{0}(AK)$.

\begin{Def} We denote by $z_{\omega} \in \Func(P,\C)$ the function \[z_{\omega}(\chi)  = q^{(\omega,\chi)}.\]
\end{Def}

Note that $z_0 = 1$, $z_{\omega}z_{\chi} = z_{\omega+\chi}$, $z_{\omega}^* = z_{\omega}$. Moreover, inside $\Fun_q^{0}(AK)_{\loc}$ we then have the commutation relations \[xz_{\chi} = q^{(\chi,\lwt(x))} z_{\chi}x,\quad x\in \Pol_q(K)_{\loc},\chi\in P,\] and hence \[|b|_{\omega}z_{\chi} = z_{\chi}|b|_{\omega},\qquad u_{\omega}z_{\chi} = q^{(\chi,\omega)} z_{\chi}u_{\omega}.\]

\begin{Lem}\label{LemInPol}
Put \[e_r  =(1-q_r^2)^{-1}  b_{\varpi_r}^{-1}(b_{\varpi_r}\lhd E_r) \in  \Pol_q(K)_{\loc}.\]  Then for all $x\in \Pol_q(K)$, we have \[x \lhd E_r = xe_r -q^{(\alpha_r,\lwt(x))}e_rx\] as an identity inside $\Pol_q(K)_{\loc}$.
\end{Lem}

\begin{proof} Using that $\lwt(x\lhd E_r) = \lwt(x) - \alpha_r$, we find by the $U_q(\mfk)$-module algebra structure that
\begin{eqnarray*} 0 &=& (b_{\varpi_r}x - q^{(\varpi_r,\lwt(x)-w_0\rwt(x))}xb_{\varpi_r})\lhd E_r \\ &=& (b_{\varpi_r}\lhd E_r) (x\lhd K_{\alpha_r})  + b_{\varpi_r}(x\lhd E_r) \\ &&\qquad 
- q^{(\varpi_r,\lwt(x)-w_0\rwt(x))}((x \lhd E_r)(b_{\varpi_r}\lhd K_{\alpha_r}) + x (b_{\varpi_r}\lhd E_r))\\ &=& q^{(\lwt(x),\alpha_r)} (b_{\varpi_r}\lhd E_r)x  + b_{\varpi_r}(x\lhd E_r) \\ &&\qquad 
- q^{(\varpi_r,\lwt(x)-w_0\rwt(x))}(q_r(x \lhd E_r)b_{\varpi_r} + x (b_{\varpi_r}\lhd E_r))\\
&=& (1-q_r^2)b_{\varpi_r}(x\lhd E_r)  + q^{(\lwt(x),\alpha_r)} (b_{\varpi_r}\lhd E_r)x 
\\ &&\qquad - q^{(\varpi_r,\lwt(x)-w_0\rwt(x))}x (b_{\varpi_r}\lhd E_r).
\end{eqnarray*} 
Upon simplifying, we find the desired formula. 
\end{proof}

Note that \[z_{\chi}e_r = q^{(\chi,\alpha_r)}e_rz_{\chi}.\]

\begin{Def} For $\omega \in P$ and $r\in I$, we define inside $\Fun_q^{0}(AK)_{\loc}$ the elements \[k_{\omega} = q^{-\frac{1}{2}(\omega,\omega)} z_{\omega}u_{-\omega},\qquad l_{\omega} =q^{-\frac{1}{2}(\omega,\omega)} u_{\omega}z_{\omega}\] and \[f_r = e_r^*l_{\alpha_r}^{-1}.\] 
\end{Def}

\begin{Prop}\label{PropImplAct} For all $x \in \Fun_q^0(AK)$, one has \[x \lhd K_{\omega} = k_{\omega}^{-1}xk_{\omega},\qquad x\lhd L_{\omega} = l_{\omega}^{-1}xl_{\omega},\]\[x \lhd E_r = -e_rk_{\alpha_r}^{-1} x k_{\alpha_r} + x e_r,\qquad x \lhd F_r = l_{\alpha_r} x f_r - l_{\alpha_r} f_rx.\]  \end{Prop} 
\begin{proof} The formulas for $\lhd K_{\omega}$ and $\lhd E_r$ follow straightforwardly from the commutation relations and Lemma \ref{LemInPol} above. The formulas for $\lhd L_{\omega}$ and $\lhd F_r$ then follow from the fact that $(x\lhd X)^* = x^*\lhd S(X)^*$. 
\end{proof}

Recall now the $*$-algebras $\widetilde{U}_q^0(\mfh_{\R})$ and $\widetilde{U}_q^0(\mfb_{\R})$ from Definition \ref{DefUtildb}.

\begin{Lem}\label{LemComHeis} There is a unital $*$-homomorphism \[\pi:\widetilde{U}_q^{0}(\mfb_{\R}) \rightarrow \Fun_q^{0}(AK)_{\loc},\]\[U_{\omega}\mapsto u_{\omega},\quad Z_{\omega}\mapsto z_{\omega},\quad E_r \mapsto e_r,\quad F_r\mapsto f_r.\]
\end{Lem} 

Note that under this morphism, we have $K_{\omega}\mapsto k_{\omega}, L_{\omega}\mapsto l_{\omega}$. 

\begin{proof} To see that we get an algebra homomorphism from $U_q(\mfb)$, we can apply the same argument as in \cite[Lemma 2.11]{DCN15}. It then follows immediately that we have also an algebra homomorphism from $U_q(\mfb^-)$ by applying the $*$-operation. The rest of the argument is now similar to the proof of \cite[Theorem 2.8]{DCN15}, but let us give a direct argument as we are in a simpler situation. 

It is easily seen that the $u_{\omega}, z_{\omega}$ generate a copy of $\widetilde{U}_q^0(\mfh_{\R})$. It is also easy to see that we have the proper interchange relations between respectively the $u_{\omega}$, $z_{\omega}$ and the $e_r$, $f_r$. We thus only have to check that \[e_rf_s - f_se_r = \delta_{rs} \frac{-l_{\alpha_r}^{-1}}{q_r-q_r^{-1}},\] which can be reduced to the relation \[e_re_s^* - q^{-(\alpha_s,\alpha_r)}e_s^*e_r = \delta_{rs} \frac{-1}{q_r-q_r^{-1}}.\] For this, we compute $(b_{\varpi_r}\lhd E_r) \lhd E_s^*$ in two ways. First, we have \[(b_{\varpi_r}\lhd E_r)\lhd E_s^* = U_{\varpi}(E_sE_r^*\xi_{\varpi_r},\eta_{w_0\varpi_r}) = \delta_{rs} q_r b_{\varpi_r},\] where we used that $E_rE_s^* - q^{-(\alpha_s,\alpha_r)}E_s^*E_r = \delta_{rs}\frac{K_{\alpha_r}^2  -1}{q_r-q_r^{-1}}$ in $U_q(\mfk)$. On the other hand, using the above commutation relations together with $b_{\omega} e_r = q^{-(\omega,\alpha_r)} e_r b_{\omega}$, we have \begin{eqnarray*} (b_{\varpi_r}\lhd E_r) \lhd E_s^* &=& -e_s^* k_{-\alpha_s}^*(b_{\varpi_r}\lhd E_r)k_{\alpha_s}^* + (b_{\varpi_r}\lhd E_r)e_s^* \\ &=&e_s^* k_{-\alpha_s}^*e_r k_{-\alpha_r}b_{\varpi_r}k_{\alpha_r}k_{\alpha_s}^* - e_s^*k_{-\alpha_s}^* b_{\varpi_r}e_r k_{\alpha_s}^*  \\&& \quad -e_rk_{-\alpha_r}b_{\varpi_r}k_{\alpha_r}e_s^* + b_{\varpi_r}e_re_s^* \\ &=& (q^{(\varpi_s+\varpi_r- \alpha_s,\alpha_r)} - q^{(\varpi_r,\alpha_s-\alpha_r)-(\alpha_s,\alpha_r)})e_s^*e_rb_{\varpi_r} \\ && \quad - (q^{(\varpi_r,\alpha_r+\alpha_s)} - q^{(\varpi_r,\alpha_s-\alpha_r)})  e_re_s^*b_{\varpi_r}. 
\end{eqnarray*} 
Comparing and simplifying, we find the required relation. 
\end{proof} 

As $\pi$ is a $*$-homomorphism, we can descend $\blhd$ to an inner action of $U_q^+(\mfb_{\R})$ on $\Fun_{q}^0(AK)_{\ext}$. We can then strengthen Proposition \ref{PropImplAct} as follows. 
\begin{Cor}\label{CorlhdBlhd} For all $x \in \Fun_{q,\compact}^{0}(AK)$ and $X \in U_q^{+}(\mfb_{\R})$, we have \[ x \lhd X =  x \blhd X,\] where a priori the right hand side lies in $\Fun_q^{0}(AK)_{\loc} \supseteq \Fun_{q,\compact}^{0}(AK)$.
\end{Cor}
\begin{proof} As both sides define module $*$-algebra structures, it is enough to verify that they agree for $X \in \{K_{\omega},E_r\}$. However, as the adjoint action in this case is given by $x\blhd X = \pi(S(X_{(1)}))x\pi(X_{(2)})$, the equality $x\blhd X = x\lhd X$ follows immediately from Proposition \ref{PropImplAct}.
\end{proof}

We now want to show that the image of $\pi$ is large enough. Let us first make the following formal definition.

\begin{Def} We define $\Pol_q^0(AK)_{\loc}$ as the $*$-subalgebra of $\Fun_{q}^0(AK)_{\loc}$ generated by $\Pol_q(K)_{\loc}$ and the $z_{\omega}$.
\end{Def}

It is clear that $\Pol_q^0(AK)_{\loc}$ can be realized abstractly as a crossed product of $\Pol_q(K)_{\loc}$ with the group algebra of $P$. It is also clear that the image of $\pi$ lands in $\Pol_q^0(AK)_{\loc}$. 

\begin{Prop}\label{PropGenAK} The $*$-algebra $\Pol_q^0(AK)_{\loc}$ is generated by $\pi(\widetilde{U}_q^0(\mfb_{\R}))$ and the $|b|_{\omega}$.   
\end{Prop}

\begin{proof} Let $B$ be the algebra generated by $\pi(\widetilde{U}_q^0(\mfb_{\R}))$ and the $|b|_{\omega}$. It is clear that $B$ is a unital $*$-algebra containing the $b_{\varpi} = U_{\varpi}(\xi_{\varpi},\eta_{w_0\varpi})$. Since $B$ will be closed under the adjoint action $\blhd$, we obtain by Corollary \ref{CorlhdBlhd} that $B$ will as well contain $b_{\varpi}\lhd U_q^+(\mfb_{\R})$ as a subspace, consisting of elements of the form $U_{\varpi}(\xi,\eta_{w_0\varpi})$ for $\xi \in V_{\varpi}$. However, from \cite[Proposition 3.2.1]{LS91} it follows that these elements and their adjoints generate $\Pol_q(K)$. The proposition now follows. 
\end{proof} 

Our following objective is to provide an implementation of the right coaction of $U_q^{+}(\mfb_{\R})$ on $U_q^0(\mfb_{\R})$ using the left coaction of $\Fun_{q,\compact}^+(AK)$ on $\Fun_{q,\compact}^0(AK)$. This implementation will be crucial to lift our algebraic constructions to the operator algebraic setting. 

Assume that $\mathscr{M}$ is a $\Fun_q^{0}(AK)_{\loc}$-module, and assume that $\mathscr{M}$ is non-degenerate as a left $\Fun_{q,\compact}^{0}(AK)$-module, i.e. \[\Fun_{q,\compact}^0(AK) \mathscr{M} = \mathscr{M}.\] Then we can define a linear map \[\Vv_{0}: \Fun_{q,\compact}^{0}(AK)\otimes \mathscr{M} \rightarrow \Fun_{q,\compact}^+(AK) \otimes\mathscr{M} ,\quad x\otimes v  \mapsto x_{(-1)}\otimes x_{(0)}v ,\] where we have used the Sweedler notation \[\gamma(x) = x_{(-1)}\otimes x_{(0)},\qquad x\in \Fun_{q,\compact}^0(AK).\]  Further consider $\Fun_{q,\compact}^+(AK)$ as a left $U_q^{+}(\mfb_{\R})$-module by the translation action, \[\pi_{\rhd}(F)x = F\rhd x =  (\id\otimes F)\Delta(x),\qquad x\in \Fun_{q,\compact}^+(AK),F \in U_q^+(\mfb_{\R}).\] Finally, recall that $\alpha$ is the right $U_q^+(\mfb_{\R})$-Galois object structure on $U_q^0(\mfb_{\R})$, and that $\pi$ is the map introduced in Lemma \ref{LemComHeis}.

\begin{Theorem}\label{TheoAdjTran} For all $X \in  U_q^{0}(\mfb_{\R})$, one has \[\Sigma \Vv_0 (1\otimes \pi(X)) = (\pi\otimes \pi_{\rhd})(\alpha(X))\Sigma\Vv_0 .\]
\end{Theorem}

\begin{proof} For $X\in U_q(\mfb)\cup U_q(\mfb^-)$, we will simply write $\alpha(X) = X_{(1)}\otimes X_{(2)}$. Then we compute for $X\in U_q(\mfb) \cup U_q(\mfb^-)$, $v\in \mathscr{M}$, $x\in \Fun_{q,\compact}^+(AK)$ and $g,h\in \Func_{\compact}(P,\C)$ that
\begin{eqnarray*}  && \hspace{-1cm} (1\otimes g)(\pi\otimes \pi_{\rhd})(\alpha(X))\Sigma\Vv_0  (x\otimes hv) \\ && =\pi(X_{(1)})x_{(0)}hv \otimes g(X_{(2)}\rhd x_{(-1)}) \\ && =  \pi(X_{(1)})x_{(0)}hv \otimes (X_{(2)},x_{(-1)})gx_{(-2)} \\ &&=  \pi(X_{(1)})(x_{(0)} \lhd X_{(2)}) hv \otimes gx_{(-1)} \\ &&= \pi(X_{(1)})\pi(S(X_{(2)}))x_{(0)}\pi( X_{(3)}) hv \otimes gx_{(-1)} \\ &&= x_{(0)}\pi(X) hv \otimes gx_{(-1)} \\ &&= (1\otimes g)\Sigma\Vv_0 (1\otimes\pi(X)) (x\otimes hv).
\end{eqnarray*}
Note that all expressions with multipliers are meaningful as they are covered by the elements $g,h$. As $g,h$ were arbitrary, the theorem follows. 
\end{proof}

\section{Representation theory}

\subsection{Representation theory of $\Fun_q^{0}(AK)$.}

Let us recall first some elements from the representation theory of the unital $*$-algebra $\Pol_q(K)$ from \cite{LS91}.  

Let us consider the case $K = SU(2)$. Then $\Pol_q(SU(2))$ is the Hopf $*$-algebra defined by 2 generators $a,b$ with the universal relation that \[ U = \begin{pmatrix} a & -qb^* \\ b & a^*\end{pmatrix}\] is a unitary corepresentation. It coincides with the construction as obtained from $U_q(\mfsu(2))$ by the pairing \[U(K) = \begin{pmatrix} q^{-1} &0 \\ 0 & q\end{pmatrix},\quad U(E) =\begin{pmatrix} 0 & 0 \\ q^{1/2}  & 0\end{pmatrix},\quad U(F) = \begin{pmatrix} 0 &  q^{-1/2}\\ 0 & 0 \end{pmatrix}.\] It has the following irreducible $*$-representation on $\C[\N]$, the latter being equipped with the usual pre-Hilbert space structure making the basis vectors $e_n$ an orthonormal basis:  \[\theta: \Pol_q(SU(2)) \rightarrow \End(\C[\N]),\quad a e_n = (1-q^{2n})^{1/2}e_{n-1},\quad be_n = q^{n}e_n.\] This representation is bounded, and can hence be extended to a $*$-representation on $l^2(\N)$. 

Returning to a general $\mfk$, note that the inclusion $\check{U}_{q_r}(\mfsu(2)) \subseteq U_q(\mfk)$ at the $r$th simple root induces a surjective $*$-homomorphism \[\pi_r: \Pol_q(K) \rightarrow \Pol_{q_r}(SU(2)).\] 
We write $\theta_r$ for the composition \[\theta_r = \theta\circ \pi_r: \Pol_q(K)\overset{\!\!\pi_r}{\rightarrow} \Pol_{q_r}(SU(2))\overset{\theta}{\rightarrow} \End(\C[\N]).\] For $\mathbf{r} = (r_1,\ldots,r_n)$ any sequence of simple roots, we denote  \[\theta_{\mathbf{r}} = \theta_{r_1}*\theta_{r_2}*\ldots * \theta_{r_n},\] which is a $*$-representation on $\C[\N]^{\otimes n}$. 

In particular, choosing the sequence belonging to a reduced decomposition \[w_0 = s_{r_1}\ldots s_{r_M},\] we obtain the `big cell' $*$-representation \[\theta_{w_0} = \theta_{(r_1,\ldots,r_M)}: \Pol_q(K) \rightarrow \End(\C[\N]^{\otimes M}),\] which is irreducible and independent of the choice of decomposition of $w_0$ up to unitary equivalence \cite[Theorem 3.6.7]{LS91}. In the following, we denote \[L^2_{\hol,q}(N)_0 = \C[\N]^{\otimes M},\quad L^2_{\hol,q}(N) = l^2(\N)^{\otimes M},\] where $\hol$ stands for `holomorphic'. The notation evokes the following: we view the above representation  $\theta_{w_0}$ as a quantum (and Hilbert space) analogue of the representation of the regular function algebra $\Pol(K)$ as multiplication by holomorphic (regular) functions on $N$ under the homomorphism \[\Pol(K)  = \Pol_{\hol}(G) \overset{\Res}{\rightarrow}\Pol_{\hol}(N),\]  where the first equality is an equality of algebras equating matrix coefficients of unitary representations of $K$ with matrix coefficients of holomorphic representations of $G$.  Note that the representation $\theta_{w_0}$ will again be bounded, and can hence be extended to a representation on $L^2_{\hol,q}(N)$. 

\begin{Lem}\label{LemCycKq} The vector $e_0^{\otimes M}$ is cyclic for $\theta_{w_0}$ as a representation on $L^2_{\hol,q}(N)_0$. 
\end{Lem} 
\begin{proof} This is immediate from the irreducibility of $\theta_{w_0}$ as a representation on $L^2_{\hol,q}(N)$. 
\end{proof} 

To obtain a faithful representation of $\Pol_q(K)$, one now only needs to amplify $\theta_{w_0}$ with the characters of $\Pol(T)$, where we see $\Pol(T)$ as the group $*$-algebra of the discrete group $(P,+)$ with generators $u_{\omega}$, paired with $U_q(\mft)$ by means of \[(u_\omega,K_{\chi}) = q^{(\omega,\chi)}.\] This pairing induces a  quotient Hopf $*$-algebra homomorphism \[\pi_T:\Pol_q(K) \rightarrow \Pol(T),\quad x \mapsto \varepsilon(x) u_{\lwt(x)} = \varepsilon(x)u_{\rwt(x)}.\] It will follow from the upcoming Lemma \ref{LemRepAK} that there is no ambiguity in using the symbol $u_{\omega}$ both for elements in $\Pol_q(K)_{\loc}$ and $\Pol(T)$.

Let us further also denote  \[L^2_{\hol,q}(H)_0 = \C[P],\quad L^2_{\hol,q}(H) = l^2(P).\]
Then we can view $\Pol(T)\subseteq \End(L^2_{\hol,q}(H)_0)$ by \[u_{\omega}e_{\chi} = e_{\chi+\omega}.\]Denote further
\[ L^2_{\hol,q}(B)_0 = L^2_{\hol,q}(H)_0 \otimes L^2_{\hol,q}(N)_0,\quad L^2_{\hol,q}(B) =  L^2_{\hol,q}(H)\otimes  L^2_{\hol,q}(N),\] where we will sometimes use the shorthand notation \[e_{\omega}\otimes \xi = \xi_{\omega}.\] 
Then we obtain the faithful, bounded $*$-representation 
\[\pi_L = \pi_T*\theta_{w_0}:\Pol_q(K) \rightarrow  \End(L^2_{\hol,q}(B)_0).\] The faithfulness follows essentially from the faithfulness in the rank 1, combined with the fact that $U_q(\mfk)$ equals the product $\check{U}_{q_1}(\mfsu(2))\ldots \check{U}_{q_M}(\mfsu(2))$ up to factors in the Cartan part.  

We can extend the above $*$-representation of $\Pol_q(K)$ to a $*$-representation of $\Fun_q^{0}(AK)$ in the following way. 

\begin{Lem}\label{LemRepAK} The $*$-representation of $\Pol_q(K)$ on $L^2_{\hol,q}(B)_{0} = \C[P]\otimes \C[\N]^{\otimes M}$ extends uniquely to $\Fun_q^{0}(AK)_{\loc}$ such that for all $\xi \in  \C[\N]^{\otimes M}$ and $f\in \Func(P,\C)$ \[f \xi_ {\chi} = f(-\chi) \xi_{\chi},\quad u_{\omega}\xi_{\chi} =\xi_{\chi+\omega}.\] Moreover, the restriction of this representation to $\Fun_{q,\compact}^{0}(AK)$ is a bounded, non-degenerate, faithful and irreducible representation on $L^2_{\hol,q}(B)_{0}$. 
\end{Lem} 

\begin{proof} First of all, note that for $x$ left homogeneous in $\Pol_q(K)$ we have \begin{equation}\label{EqActxHom} x\xi_{\chi} =  (\theta_{w_0}(x)\xi)_{\chi + \lwt(x)}.\end{equation}

Let now again $\pi_r: \Pol_q(K) \rightarrow \Pol_{q_r}(SU(2))$ be the canonical Hopf $*$-algebra surjection corresponding to the inclusion $\check{U}_{q_r}(\mfsu(2)) \subseteq U_q(\mfk)$ at position $r$. Then it is well-known that 
 \[(\pi_{r_1}*\ldots *\pi_{r_M})(b_{\varpi}) = \otimes_{i=1}^M b_{r_i}^{(\varpi,\check{\beta}_{i})},\] where $b_r$ denotes the element $b$ within $\Pol_{q_r}(SU(2))$, and where \[\beta_i = s_{r_1}\ldots s_{r_{i-1}} \alpha_{r_i}.\] 
It follows in particular that \begin{equation}\label{EqActa}  b_{\varpi}(e_{n_1,\ldots,n_M})_{\chi} = q^{\sum_{i=1}^M n_i(\varpi,\beta_i)}  (e_{n_1,\ldots, n_M})_{\chi+\varpi},\end{equation} with $\{e_{n_1,\ldots,n_M}\}$ the ordinary orthonormal basis of $L^2_{\hol,q}(N)_{0} = \C[\N]^{\otimes M}$.

Hence it is clear that upon defining $u_{\omega}$ as in the statement of the lemma, we get a $*$-representation of $\Pol_q(K)_{\loc}$. By \eqref{EqActxHom}, it follows immediately that the proposed action of $\Func(P,\C)$ makes this into a $\Fun_{q}^{0}(AK)_{\loc}$-representation. Also the non-degeneracy, faithfulness, boundedness and irreducibility of $L^2_q(B)_0$ as a left $\Fun_{q,\compact}^{0}(AK)$-representation are immediately clear. 
\end{proof}

\subsection{Representation theory of $U_q^{0}(\mfb_{\R})$}

By means of Lemma \ref{LemRepAK} and the $*$-homomorphism $\pi$ in Lemma \ref{LemComHeis}, we obtain a $*$-representation of $U_q^0(\mfb_{\R})$ on $L^2_q(B)_0$. The goal of this section is to give another interpretation of this representation. The main idea is to study the representation theory of the $*$-algebra $U_q^0(\mfn_{\R})$ introduced in Definition \ref{DefUqn}. 

\begin{Def} A \emph{highest weight representation} of $U_q^{0}(\mfn_{\R})$ is a pre-Hilbert space $V$ with a $*$-representation by $U_q^{0}(\mfn_{\R})$ admitting a vector $\xi_0$, called \emph{highest weight vector}, such that $X_r\xi_0 = 0$ for all $r$ and $V = U_q^{0}(\mfn_{\R})\xi_0$.
\end{Def}

\begin{Theorem}\label{TheoUnique} Up to isomorphism, there is a unique highest weight representation $V$ for $U_q^{0}(\mfn_{\R})$, which is then necessarily faithful and irreducible. Moreover, with $\xi_0$ a highest weight vector, the map \[U_q(\mfn)^* \rightarrow V,\quad X \mapsto X\xi_0\] is an isomorphism of vector spaces. 
\end{Theorem}
 
\begin{proof} Recall the pairings introduced in \eqref{EqPairing} and \eqref{EqPairingDeg}. If $V$ is a highest weight representation for $U_q^0(\mfn_{\R})$ with highest weight vector $\xi_0$, then for $X,Y\in U_q(\mfn)$ we have, by the discussion before Definition \ref{DefHeisOmChi} describing $U_q^0(\mfn_{\R})$ as a cocycle deformation, \begin{eqnarray*} \langle X^*\xi_0,Y^*\xi_0\rangle &=& \langle \xi_0, XY^* \xi_0\rangle \\ &=& (S(X_{(1)}),Y_{(1)}^*)_0 (X_{(3)},Y_{(3)}^*) \langle \xi_0,Y_{(2)}^*X_{(2)}\xi_0\rangle \\ &=&  (S(X_{(1)}),Y_{(1)}^*)_0 (X_{(2)},Y_{(2)}^*) \\ &=&  (X,Y^*),
\end{eqnarray*}
where in the last step we used that $U_q(\mfn)^*$ is a right coideal in $U_q(\mfb^-)$, pairing as the counit with $U_q(\mfb)$ under the pairing $(\,\cdot\,,\,\cdot\,)_0$. Since the pairing $(\,\cdot\,,\,\cdot\,)$ is well-known to be non-degenerate, it follows that the map \[U_q(\mfn)^* \rightarrow V,\quad X\mapsto X\xi_0\] is bijective. 

To show that there exists a highest weight representation, let $U_q^{0}(\mfk)$ be the unital $*$-algebra obtained by adjoining to $U_q^{0}(\mfn_{\R})$ self-adjoint elements $K_{\omega}$ for $\omega\in \frac{1}{2}P$ such that $K_0 = 1$, $K_{\omega}K_{\chi} = K_{\omega+\chi}$ and \[K_{\omega}X_r  = q^{(\omega,\alpha_r)}X_rK_{\omega}.\] Then writing $X_r^+ = K_{\alpha_r/2}X_r$ and $X_r^- = X_r^* K_{\alpha_r/2}$, it is easy to see that we get the same $*$-algebra as the one which is introduced as $U_q(\mfg;0,+)$ in \cite[Definition 1.1]{DeC13}. It now follows from \cite[Proposition 2.6]{DeC13} that $U_q(\mfn)^*$ can be turned into a pre-Hilbert space with $1$ as a highest weight vector for the unique $U_q^0(\mfn_{\R})$-representation such that $U_q(\mfn)^*$ acts by left multiplication. 

The irreducibility of a highest weight module $V$ is clear. To see that it must be faithful, it is enough to prove the following: if $\{X_i\}$ is a basis of $U_q(\mfn)$, and $c_{ij}$ are such that for all $Z,W \in U_q(\mfn)$ \[\langle Z^*\xi_0,\left(\sum_{i,j} c_{ij}X_iX_j^*\right) W^*\xi_0\rangle = 0,\] then $c_{ij}=0$ for all $i,j$. But \begin{eqnarray*} \langle Z^*\xi_0, \sum_{i,j} c_{ij}X_iX_j^* W^*\xi_0\rangle  &=& \sum_{i,j} c_{ij}\langle X_i^*Z^*\xi_0, X_j^* W^*\xi_0\rangle \\ &=&\sum_{i,j} c_{ij} (ZX_i,X_j^* W^*) \\ &=& \sum_{i,j} c_{ij} (Z,X_{j(1)}^*W_{(1)}^*) (X_i,X_{j(2)}^*W_{(2)}^*)\end{eqnarray*} Since $(\,\cdot\,,\,\cdot\,)$ is non-degenerate, it follows that for all $W\in U_q(\mfn)$ \[\sum_{i,j} c_{ij} (X_i,X_{j(2)}^*W_{(2)}^*) X_{j(1)}^*W_{(1)}^* = 0.\] It is clear that this still holds for all $W\in U_q(\mfb)$. But since $\Delta(H)(H\otimes 1) = H\otimes H$ for a Hopf algebra, it follows that also \[\sum_{i,j} c_{ij} (X_i,X_{j(2)}^*W^*) X_{j(1)}^* = 0\] for all $W\in U_q(\mfn)$. Then \[\sum_{i,j} c_{ij} (X_{i(2)},X_{j(2)}^*)(X_{i(1)},W^*) X_{j(1)}^* = 0,\] hence \[\sum_{i,j} c_{ij} (X_{i(2)},X_{j(2)}^*)X_{i(1)}\otimes X_{j(1)}^* = 0.\]
From this, we get \[\sum_{i,j} c_{ij} (X_{i(3)},X_{j(3)}^*)(S(X_{i(2)}),X_{j(2)})X_{i(1)}\otimes X_{j(1)}^* = \sum_{i,j} c_{ij} X_i\otimes X_i^* =  0,\] hence all $c_{ij}=0$. 
\end{proof}
 
\begin{Cor}\label{CorGradHW} Let $V$ be a highest weight representation of $U_q^0(\mfn_{\R})$. Then there exists a unique $-Q^+$-grading on $V$ such that $\wt(\xi_0) = 0$ and $\wt(X_r^*\xi) = \wt(\xi) - \alpha_r$ for all homogeneous $\xi$ and all $r$.  
\end{Cor} 
 
Consider now in $\Fun_q^0(AK)$ the elements \[x_r = u_{\alpha_r}e_r = e_ru_{\alpha_r}.\] Then it follows immediately that we obtain a $*$-homomorphism \[\pi:U_q^0(\mfn_{\R}) \rightarrow \Fun_q^0(AK),\quad X_r \mapsto x_r.\] On the other hand, under the representation of $\Fun_q^0(AK)$ on $L^2_{\hol,q}(B)_0$, the operators $x_r$ and $x_r^*$ act only on $L^2_{\hol,q}(N)_0$. 

\begin{Cor}\label{CorFaithRep} The representation of $U_q^{0}(\mfn_{\R})$ on $L^2_{\hol,q}(N)_0$ is faithful, and the map \[U_q(\mfn)^* \rightarrow L^2_{\hol,q}(N)_0, \quad X \mapsto \pi(X)e_0^{\otimes M}\] is a linear isomorphism.  
\end{Cor} 

\begin{proof} We claim that $e_0^{\otimes M}$ is a highest weight vector.
By the formula for the $e_r$ in Lemma \ref{LemInPol}, it is easy to see that the vector $e_0^{\otimes M}$ vanishes under the action of the $x_r$. To see that $e_0^{\otimes M}$ is cyclic, it is sufficient to show that $e_0^{\otimes l}\otimes e_0^{\otimes M}$ is cyclic for the tensor product representation of $\widetilde{U}_q^0(\mfh_{\R})\otimes \widetilde{U}_q^0(\mfn_{\R}) \cong \widetilde{U}_q^0(\mfb_{\R})$ on $L^2_{\hol,q}(H)_0 \otimes L^2_{\hol,q}(N)_0 = L^2_{\hol,q}(B)_0$, where we have used the isomorphism \eqref{EqIsotildExt}. However, by Proposition \ref{PropGenAK} this will follow from the cyclicity of $e_0^{\otimes l}\otimes e_0^{\otimes M}$ for the representation of $\Fun_{q}^{0}(AK)_{\loc}$. But this holds since $e_0^{\otimes l}$ is clearly cyclic for $\widetilde{U}_q^0(\mfh_{\R})$, while $e_0^{\otimes M}$ is cyclic for $\Pol_q(K)$ under $\theta_{w_0}$ by Lemma \ref{LemCycKq}. 

It thus follows that $L^2_{\hol,q}(N)_0$ is a highest weight representation for $U_q^0(\mfn_{\R})$ under $\pi$. The corollary now follows from the last part of Theorem \ref{TheoUnique}.
\end{proof} 

Note that the notation is consistent with $U_q(\mfn)$ acting as locally nilpotent differential operators on holomorphic (regular) functions on $N$.

\begin{Cor}\label{CorFaithPi} The $*$-homomorphism \[\pi: U_q^0(\mfb_{\R}) \rightarrow \Fun_{q}^0(AK)_{\loc}\] is faithful, and in particular the representation of $U_q^0(\mfb_{\R})$ on $L^2_{\hol,q}(B)_0$ is faithful. 
\end{Cor}
\begin{proof} Recall the $*$-algebra $\widetilde{U}_q(\mfb_{\R})$ introduced in Definition \ref{DefUtildb}. Then the representation of $U_q^0(\mfb_{\R})$ on $L^2_{\hol,q}(B)_0$ extends in a unique way to $\widetilde{U}_q^0(\mfb_{\R})$ in such a way that the $U_{\omega}$ are sent to the $u_{\omega}$, and under the isomorphism of Lemma \ref{Lemtildb} the resulting representation is simply the tensor product of the representations of $\widetilde{U}_q^0(\mfh_{\R})$ and
$U_q^0(\mfn_{\R})$  on respectively $L^2_q(H)_0$ and $L^2_q(N)_0$. The second representation is faithful by Theorem \ref{TheoUnique}. The first representation is easily seen to be faithful as we have \[U_{\omega}e_{\chi} = e_{\chi+\omega},\quad Z_{\omega}e_{\chi} = q^{-(\omega,\chi)}e_{\chi},\] and the maps \[\evv_{\chi}: Z_{\omega}\mapsto q^{(\omega,\chi)},\qquad \chi\in P\] separate the elements in the linear span of the $Z_{\omega}$ by Zariski density.
\end{proof}

\section{Invariant integrals and von Neumann algebraic completions}

\subsection{Invariant integral on $\Fun_{q,\compact}^{+}(AK)$}

We begin by recalling the following definition.

\begin{Def}\cite{VDae98} A multiplier Hopf $*$-algebra $(H,\Delta)$ is said to admit \emph{positive left and right invariant integrals} if there exist non-zero functionals \[\varphi,\psi: H \rightarrow \C\] which are positive, that is, \[\varphi(x^*x) \geq 0,\qquad \psi(x^*x) \geq 0,\qquad x\in H,\] and such that, in the sense of multipliers, \[(\id\otimes \varphi)\Delta(h) = \varphi(h)1,\qquad (\psi\otimes \id)\Delta(h) = \psi(h)1.\] We then say that $(H,\Delta)$ determines a \emph{$*$-algebraic quantum group}. If $(H,\Delta)$ is an ordinary Hopf $*$-algebra, we say that $(H,\Delta)$ is a $*$-algebraic quantum group \emph{of compact type}. 
\end{Def} 

It can be shown that positive invariant integrals are automatically unique up to scalar multiplication and faithful, \[\varphi(x^*x) = 0 \Rightarrow x=0,\qquad \psi(x^*x)=0 \Rightarrow x=0.\] In the following, we fix $\psi$. It follows that one can form an associated pre-Hilbert space $L^2(H)_0$, which is $H$ endowed with the inner product \[\langle x,y \rangle = \psi(x^*y).\] We will denote by $L^2(H)$ the completion of $L^2(H)_0$, and by \[\Gamma: H \rightarrow L^2(H),\quad h\mapsto \Gamma(h)\] the canonical embedding. It is then also a general fact that the left $H$-module structure on $H$ by left multiplication extends to a $*$-representation of $H$ by bounded operators on $L^2(H)$, see \cite{KV97}. Moreover, as follows from the general theory in \cite{VDae98}, there exists a unique invertible grouplike \emph{modular element} $\delta \in M(H)$ such that, for a unique choice of $\varphi$, \[\psi(x) = \varphi(x\delta),\qquad x\in H.\] One also has \emph{modular automorphisms} $\sigma,\acsigma$ on the algebra $H$, satisfying  \begin{equation}\label{EqModAut}\varphi(x\sigma(y)) = \varphi(yx),\qquad \psi(x\acsigma(y)) = \psi(yx),\qquad x,y\in H.\end{equation} When $(H,\Delta)$ is of compact type we have $\varphi = \psi$ and hence $\delta = 1$, but $\sigma$ can be non-trivial. 

An example of a $*$-algebraic quantum group of compact type is given by $\Pol_q(K)$. 

\begin{Theorem} There exists a positive invariant state $\psi: \Pol_q(K) \rightarrow \C$, given by \[\psi(U_{\varpi}(\xi,\eta)) = \delta_{\varpi,\hat{\varepsilon}} \langle \xi,\eta\rangle,\] where we see $\hat{\varepsilon}$ as the trivial representation of $U_q(\mfk)$ on $\C$. 
\end{Theorem} 

\begin{proof} This can be proven by general Tannakian methods, cf. \cite[Section 2]{NT13}, or by observing that by \cite{LS91} the $*$-algebra $\Pol_q(K)$ embeds in its universal C$^*$-algebraic envelope, so that the results of \cite{Wor98} provide a positive invariant state, which then necessarily has to agree on $\Pol_q(K)$ with $\psi$. 
\end{proof} 

We will in this case write $L^2_q(K)_0$ for the associated pre-Hilbert space, and by $L^2_q(K)$ its completion. 

We recall the following orthogonality relations, see e.g.~ \cite{NT13} (with slightly different conventions):  \[\psi(U_{\varpi}(\xi,\eta)^*U_{\varpi'}(\xi',\eta')) = \delta_{\varpi,\varpi'} q^{-2(\rho,\wt(\xi))}\langle \xi',\xi\rangle \langle \eta,\eta'\rangle  ,\]\[ \psi(U_{\varpi}(\xi,\eta) U_{\varpi'}(\xi',\eta')^*)=  \delta_{\varpi,\varpi'} q^{2(\rho,\wt(\eta))}\langle \xi,\xi'\rangle \langle \eta',\eta\rangle.\] It follows that we have \begin{equation}\label{EqDefAcSig} \acsigma(x) = q^{2(\rho,\lwt(x)+\rwt(x))}x.\end{equation}

The following proposition shows that also the amplification $(\Fun_{q,\compact}^+(AK),\Delta)$ is a $*$-algebraic quantum group. 

\begin{Prop}\label{PropMultIsAQG} The multiplier Hopf $*$-algebra $(\Fun_{q,\compact}^+(AK),\Delta)$ admits a left and right invariant positive integral $\psi^+$, determined by \[\psi^+(xf) = \psi(x)\sum_{\omega} f(\omega).\] In particular, $\Fun_{q,\compact}^+(AK)$ is a $*$-algebraic quantum group. 
\end{Prop} 

\begin{proof} The invariance of $\psi^{+}$ is immediate. To see that $\psi^{+}$ defines a positive functional on $\Fun_{q,\compact}^{+}(AK)$, recall first that $\lwt(x^*) = -\lwt(x)$. We note then that \begin{eqnarray*}&& \hspace{-0.5cm} \psi((\sum_i x_if_i)^*(\sum_i x_if_i))  \\ &&= \sum_{i,j,\omega} \psi(x_i^*x_j) (f_{i,(\lwt(x_j)-\lwt(x_i))-(\rwt(x_j)-\rwt(x_i))}^*f_j)(\omega) \\ &&= \sum_{i,j,\omega}  \psi(x_i^*x_j)  \overline{f_{i}(\omega - (\lwt(x_j)-\lwt(x_i))+(\rwt(x_j)-\rwt(x_i)))} f_j(\omega)
\\ && =
 \sum_{i,j,\omega}  \psi(x_i^*x_j)  \overline{f_{i}(\omega + \lwt(x_i)-\rwt(x_i))}  f_j(\omega + \lwt(x_j) - \rwt(x_j))\\
 && = \sum_{\omega}\psi ((\sum_i \lambda_{\omega,i} x_i)^* (\sum_i \lambda_{\omega,i}x_i))
\\ &&\geq 0,
\end{eqnarray*}
where $\lambda_{\omega,i} = f_{i}(\omega +  \lwt(x_i)-\rwt(x_i))$.
\end{proof} 

We can hence form the pre-Hilbert space  $L^{2}_{q,+}(AK)_0$, which is $\Fun_{q,\compact}^{+}(AK)$ endowed with the inner product \[\langle x,y\rangle = \psi^{+}(x^*y).\] We will denote the Hilbert space completion as $L^{2}_{q,+}(AK)$, and we denote by \[\Gamma^{+}: \Fun_{q,\compact}^{+}(AK) \rightarrow L^{2}_{q,+}(AK)\] the inclusion. We also easily find from \eqref{EqDefAcSig} that the modular automorphism $\acsigma^+$ is given by \begin{equation}\label{EqDefAcSigAmp} \acsigma^+(xf) = \acsigma(x)f.\end{equation}

Let us now consider again a general $*$-algebraic quantum group $(H,\Delta)$. As the representation of $H$ on $L^2(H)_0$ extends to a bounded $*$-representation on $L^2(H)$, we can consider the von Neumann algebraic closure \[M = H'' \subseteq B( L^2(H)).\] Let \[ V_0: L^2(H)_0\otimes L^2(H)_0\rightarrow L^2(H)_0\otimes L^2(H)_0,\] \begin{equation}\label{DefVAqg} \Gamma(x)\otimes \Gamma(y) \mapsto (\Gamma\otimes \Gamma)(\Delta(x)(1\otimes y)).\end{equation} Then one easily sees that $V_0$ is a unitary linear isomorphism, implementing $\Delta$ in the sense that \[\Delta(x) = V_0(x\otimes 1)V_0^*,\qquad x\in H.\] Let $V$ be the closure of $V_0$. Then it is easy to deduce from the above that we obtain a coassociative normal $*$-homomorphism \[\Delta: M\rightarrow M\vNtimes M,\quad x\mapsto V(x\otimes 1)V^*.\] Let us also remark that the multiplier $*$-algebra $M(H)$ can be $*$-represented on the pre-Hilbert space $L^2(H)_0$, but in general this will not be a $*$-representation by bounded operators. 

\begin{Theorem}\label{TheoCompHopf} With the above notations, $(M,\Delta)$ is a quantum group von Neumann algebra, and one can extend the $\varphi,\psi$ on $H$ to resp. left and right invariant nsf weights on $M$. Moreover, the modular element for $(M,\Delta)$ is the closure of the modular element $\delta\in M(H)$ acting on $L^2(H)_0$. 
\end{Theorem}
 
\begin{proof} Existence of a left invariant nsf weight $\varphi$ is proven in \cite[Proposition 6.2 and Corollary 6.14]{KV97} -- these results are proven in the C$^*$-context, but the same arguments can be used to deduce the result on the level of von Neumann algebras (alternatively, one could invoke the results of \cite{KV03} to pass from the C$^*$-level to the von Neumann algebra level). By considering the opposite coproduct, the same result holds for $\psi$. 

The claim concerning $\delta$ follows from \cite[Proposition 8.14 and Theorem 9.18]{KV97}. 
\end{proof} 

For the case of $H = \Fun_{q,\compact}^+(AK)$, we introduce the following notation. 

\begin{Def} We define 
\[L^{\infty}_{q,+}(AK) = \Fun_{q,\compact}^+(AK)''\subseteq B(L^2_{q,+}(AK)).\] 
\end{Def}

We can thus equip $L^{\infty}_{q,+}(AK)$ with a comultiplication extending the one on $\Fun_{q,\compact}^+(AK)$, giving us a quantum group von Neumann algebra. Let us give a direct proof for the existence of a (left and right) invariant nsf weight for this specific case

\begin{Prop}\label{PropPsiConc} The couple $(L^{\infty}_{q,+}(AK),\Delta)$ is a quantum group von Neumann algebra. 
\end{Prop} 
\begin{proof} 
Let $\delta_{\chi}$ be the Dirac function in $\chi\in P$, and let $\psi^{+}$ be the restriction to $L^{\infty}_{q,+}(AK)$ of the weight \begin{equation}\label{EqWeightPsiTilde} B(L^{2}_{q,+}(AK))^+ \rightarrow [0,+\infty),\quad x \mapsto \sum_{\chi\in P} \langle \Gamma^{+}(\delta_{\chi}),x\Gamma^{+}(\delta_{\chi})\rangle.\end{equation} Then clearly $\psi^{+}$ is an nsf weight. We claim that it is invariant for $\Delta$. We will prove right invariance, as left invariance follows similarly.

 Take $\xi \in L^{2}_{q,+}(AK)$, and let \[\xi = \oplus_{\chi} \xi_{\chi}\] be its decomposition along the orthogonal subspaces $\delta_{\chi}L^{2}_{q,+}(AK)$. Then with $x\in L^{\infty}_{q,+}(AK)$, we compute
   \begin{eqnarray*}&&\hspace{-0.6cm}\psi^{+}((\id\otimes \omega_{\xi,\xi})\Delta(x))\\ &&= \sum_{\chi,\chi',\nu} \langle V^* ( \Gamma^{+}(\delta_{\nu})\otimes \xi_{\chi}),(x\otimes 1)V^* (\Gamma^{+}(\delta_{\nu})\otimes \xi_{\chi'})\rangle \\ &&= \sum_{\chi,\chi',\nu} \langle  \Gamma^{+}(\delta_{\nu+\chi})\otimes \xi_{\chi},x\Gamma^{+}(\delta_{\nu+\chi'})\otimes  \xi_{\chi'}\rangle\\ &&= \sum_{\chi,\nu} \langle \Gamma^{+}(\delta_{\nu})\otimes \xi_{\chi},x\Gamma^{+}(\delta_{\nu})\otimes \xi_{\chi}\rangle \\ &&  = \langle \xi,\xi\rangle \psi^{+}(x).
\end{eqnarray*}
\end{proof}

\subsection{Duals for $*$-algebraic quantum groups and von Neumann algebraic completion of $U_q^+(\mfb_{\R})$.}

There is the following duality theory for $*$-algebraic quantum groups \cite{VDae98}. Let $(H,\Delta)$ be a $*$-algebraic quantum group, and let \[\hat{H} := \{\omega: H \rightarrow \C\mid \exists x\in H, \forall y\in H, \omega(y) = \psi(xy).\}.\] Then $\hat{H}$ becomes a $*$-algebra for the operations \[(\omega* \chi)(x) = (\omega\otimes \chi)(\Delta(x)),\qquad \omega^*(x) = \overline{\omega(S(x)^*)},\] where the right hand side of the first expression is meaningful by the concrete form of $\omega,\chi$. The $*$-algebra $\hat{H}$ then becomes a $*$-algebraic quantum group for the unique comultiplication map $\hat{\Delta}: \hat{H}\rightarrow M(\hat{H}\otimes \hat{H})$ such that \[(\hat{\Delta}(\omega)(1\otimes \chi))(x\otimes y) = (\omega\otimes \chi)((x\otimes 1)\Delta(y)),\qquad x,y\in H,\omega,\chi\in \hat{H}.\] 

We can represent $\hat{H}$ faithfully as a $*$-algebra on the pre-Hilbert space $L^2(H)_0$ by means of translation operators, \[\pi_{\rhd}(\omega)\Gamma(x) = \Gamma(\omega \rhd x) = \Gamma((\id\otimes \omega)(x)),\qquad x\in H,\omega\in \hat{H}.\] In fact, since $H = HH$, it follows from \eqref{EqModAut} that \[\hat{H} = \{\omega \in H\mid \exists x,y,\forall z\in H,\omega(z) = \psi(x^*zy)\}.\] Hence with $V_0$ as in \eqref{DefVAqg}, we see that, for $\omega(z) = \psi(x^*zy)$, \[\pi_{\rhd}(\omega) = (\omega_{\Gamma(x),\Gamma(y)}\otimes \id)(V_0),\] and the $*$-representation of $\hat{H}$ is bounded. It follows straightforwardly from these considerations that, with $(M,\Delta)$ the von Neumann algebraic completion of $(H,\Delta)$, we obtain a $*$-homomorphism \[\hat{H}\rightarrow \hat{M}',\quad (\omega_{\Gamma(x),\Gamma(y)}\otimes \id)(V_0) \rightarrow (\omega_{\Gamma(x),\Gamma(y)}\otimes \id)(V).\] It is clear that the image will be $\sigma$-weakly dense in $\hat{M}'$. Moreover, this embedding preserves the comultiplication since an easy computation reveals that \[(\pi_{\rhd}\otimes \pi_{\rhd})(\hat{\Delta}(\omega)) = V_0(1\otimes \pi_{\rhd}(\omega))V_0^*,\] in correspondence with the $V$ being the left regular corepresentation for $(\hat{M}',\hat{\Delta}')$ by \eqref{EqFormDualOp}. By an explicit identification of $L^2(H)_0 \cong L^2(\hat{H})_0$, intertwining $\pi_{\rhd}$ with the presentation by left multiplication, it is also not hard to see that $\hat{M}'$ will then coincide with the completion of $\hat{H}$ as constructed in the previous section. 

Let us further note that, by the remark above \cite[Proposition 4.3]{VDae98}, we can identify \[M(\hat{H}) = \{\omega: H \rightarrow \C\mid \forall x\in H: (\omega\otimes \id)\Delta(x) \in H \textrm{ and }(\id\otimes \omega)(\Delta(x)) \in H\}.\] It follows that we can extend $\pi_{\rhd}$ to a $*$-representation of $M(\hat{H})$ on $L^2(H)_0$. 

Let us now consider the case of $H = \Fun_{q,\compact}^+(AK)$, in which case we will write the dual and completed dual as \[\Fun_{q,\compact}^+(B_{\R})  = \Fun_{q,\compact}(AK)^{\wedge},\qquad  (\msR_{q}^+(B_{\R}),\hat{\Delta}) = (L^{\infty}_{q,+}(AK)^{\wedge \prime},\hat{\Delta}').\] Note also that by the above characterization of $M(\Fun_{q,\compact}^+(B_{\R}))$, it follows that the embedding of $U_q^+(\mfb_{\R})$ into the dual vector space of $\Fun_{q,\compact}^+(AK)$, provided by \eqref{EqPairUP}, gives us a comultiplication preserving $*$-homomorphism \[U_q^+(\mfb_{\R}) \rightarrow M(\Fun_{q,\compact}^+(B_{\R})).\] One can concretely write down the resulting action of $U_q^+(\mfb_{\R})$ on $L^2_{q,+}(AK)_0$ as follows: For $X\in U_q^+(\mfb_{\R})$ such that $\pi_{\mfk}(X_{(1)}) \otimes \pi_{\mfa}(X_{(2)}) = \pi_{\mfk}(X) \otimes u_{\omega}$ for some $\omega \in P$, we have
  \[\pi_{\rhd}(X) \Gamma^+(U_{\varpi}(\xi,\eta)f) =  \Gamma^+(U_{\varpi}(\xi,\pi_{\mfk}(X)\eta)f_{-\omega}).\] Note that any element of $U_q^{+}(\mfb_{\R})$ can be written as a linear combination of such elements. 
 
We see that each $\pi_{\rhd}(X)$ is in fact a direct sum of bounded operators on Hilbert spaces $L^2_{q,+}(AK)_{\varpi}$, whose natural dense pre-Hilbert spaces we write as $L^2_{q,+}(AK)_{\varpi,0}$. It follows that we can consider for each $X\in U_q^+(\mfb_{\R})$ the closed unbounded operator \[\mbX = \textrm{ closure of }\pi_{\rhd}(X),\] which will be affiliated with $\msR_{q}^+(B_{\R})$. The implementation of $\hat{\Delta}$ by means of $V_0$ also implies that we have compatibility with the coproduct, with \[\hat{\Delta}(\mbX) = \textrm{ closure of }(\pi_{\rhd}\otimes \pi_{\rhd})(\hat{\Delta}(X)).\] 

Let us note the following formula for the modular operator of $\msR_{q}^+(B_{\R})$.

\begin{Def} The modular operator $\hat{\delta}\eta \msR_{q}^+(B_{\R})$ is given by \[\hat{\delta} = \mbK_{-2\rho}^*\mbK_{-2\rho}.\] 
\end{Def} 
\begin{proof} By \cite{Kus97}, we have that, as a functional, the modular element $\hat{\delta} \in \Fun_{q,\compact}^+(B_{\R})$ is given by \[\hat{\delta}(x) = \varepsilon(\acsigma^{+,-1}(x)),\qquad x\in \Fun_{q,\compact}^+(AK).\] By \eqref{EqDefAcSigAmp}, it follows that \[\hat{\delta}(xf) = q^{-2(\rho,\lwt(x)+\rwt(x))} f(0) \varepsilon(x) = (L_{-2\rho}K_{-2\rho}, xf).\] By the final part of Theorem \ref{TheoCompHopf}, it now follows that the modular operator of $\msR_{q}^+(B_{\R})$ is the closure of $\pi_{\rhd}(L_{-2\rho}K_{-2\rho}) = \pi_{\rhd}(K_{-2\rho})^*\pi_{\rhd}(K_{-2\rho})$, which is  $\mbK_{-2\rho}^*\mbK_{-2\rho}$ as $\mbK_{-2\rho}$ is normal. 
\end{proof}

\subsection{von Neumann algebraic completion of $\Fun_{q,\compact}^0(AK)$ as a $I$-factorial Galois object.}

We now turn towards the von Neumann algebraic completion of the Galois object $(\Fun_{q,\compact}^0(AK),\gamma)$. 

\begin{Prop} Let \[\psi^0: \Fun_{q,\compact}^0(AK) \rightarrow \C,\quad xf \mapsto \psi(x)\sum_{\omega} f(\omega).\] Then $\psi^0$ is a positive, faithful functional, and \begin{equation}\label{EqInvPsi+} (\psi^+\otimes \id)\gamma(x) = \psi^0(x)1,\qquad x\in \Fun_{q,\compact}^0(AK).\end{equation}
\end{Prop} 
\begin{proof} The formula \eqref{EqInvPsi+} is immediate by the definition of $\gamma$. Since $\Fun_{q,\compact}^0(AK)$ is realizable as operators on a pre-Hilbert space, it follows by \cite[Corollary 1.8 and Proposition 1.22]{DeC09} that $\psi^0$ is faithful and positive. Alternatively, and can give a direct proof as in Proposition \ref{PropMultIsAQG}.
\end{proof} 

We can hence form the pre-Hilbert space $L^2_{q,0}(AK)_0$ as the vector space $\mathscr{F}_{q,\compact}^0(AK)$ with inner product \[\langle x,y\rangle = \psi^0(x^*y).\] We denote its completion by $L^2_{q,0}(AK)$ and by \[\Gamma^0: \Fun_{q,\compact}^0(AK) \rightarrow L^2_{q,0}(AK)\] for the natural inclusion map. Note that the $*$-representation of $\Fun_{q,\compact}^0(AK)$ as left multiplication operators on $L^2_{q,0}(AK)_0$ must be bounded, and hence extendable to $L^2_{q,0}(AK)$, since \emph{any} $*$-representation of $\Pol_q(K)$ or $\Func_{\compact}(P)$ on a pre-Hilbert space is bounded. 

\begin{Def} We define \[L^{\infty}_{q,0}(AK) = \Fun_{q,\compact}^0(AK)'' \subseteq B(L^2_{q,0}(AK)).\]
\end{Def} 

Since $(\Fun_{q,\compact}^0(AK),\gamma)$ is a left Galois object for $\Fun_{q,\compact}^+(AK)$, we have that the map \[\Vv_0: L^2_{q,0}(AK)_0 \otimes L^2_{q,0}(AK)_0 \rightarrow L^2_{q,+}(AK)_0\otimes L^2_{q,0}(AK)_0,\]\[\Gamma^0(x) \otimes \Gamma^0(y) \mapsto (\Gamma^+\otimes \Gamma^0)(\gamma(x)(1\otimes y))\] must be a linear isomorphism. By the invariance of $\psi^0$, it is also immediate that it is unitary. Furthermore, \[\Vv_0(x\otimes 1)\Vv_0^* = \gamma(x)\] as operators on $L^2_{q,+}(AK) \otimes L^2_{q,0}(AK)$. Write $\Vv$ for the closure of $\Vv_0$. Then it is clear by the above that $\gamma$ extends to a map \[\gamma: L^{\infty}_{q,0}(AK) \rightarrow L^{\infty}_{q,+}(AK) \vNtimes L^{\infty}_{q,0}(AK),\] and that this defines a coaction of $L^{\infty}_{q,+}(AK)$ on $L^{\infty}_{q,0}(AK)$. 

\begin{Theorem} The couple $(L^{\infty}_{q,0}(AK),\gamma)$ is a left Galois object for $L^{\infty}_{q,+}(AK)$.
\end{Theorem}
 
\begin{proof} As in Proposition \ref{PropPsiConc}, one shows that \[\psi^0: B(L^2_{q,0}(AK))^+ \rightarrow [0,+\infty],\quad x\mapsto \sum_{\chi\in P} \langle \Gamma^{0}(\delta_{\chi}),x\Gamma^{0}(\delta_{\chi})\rangle,\] restricted to $L^{\infty}_{q,0}(AK)$, defines an nsf weight satisfying \[(\psi^+\otimes \id)\gamma(x) = \psi^0(x)1,\quad x\in L^{\infty}_{q,0}(AK)^+.\] It follows that $\gamma$ is an ergodic and integrable coaction. Since the GNS-space of $L^{\infty}_{q,0}(AK)$ is easily seen to be identifiable with $L^2_{q,0}(AK)$, we see that $\Fun_{q,\compact}^0(AK) \subseteq \mathscr{N}_{\psi^0}$ is dense in the GNS-space of $L^{\infty}_{q,0}(AK)$. Hence the Galois map of $L^{\infty}_{q,0}(AK)$ coincides with the unitary $\Vv$, and $(L^{\infty}_{q,0}(AK),\gamma)$ is a Galois object. 
\end{proof}

We now will show that $L^{\infty}_{q,0}(AK)$ is in fact a $I$-factorial Galois object. For this, we first recall the following theorem from \cite{RY01}, establishing $\psi$ in terms of a trace formula with respect to the representation on $L^2_{\hol,q}(B)_0$. Let $T$ be the real spectrum of $\Pol(T)$, and let \[\evv_t: \Pol(T) \rightarrow \C\] be the evaluation map in $t\in T$. Then $T$ is a real torus, and we endow it with the Lebesgue measure.

\begin{Theorem}\cite{RY01}\label{TheoRY} An invariant positive integral $\psi$ on $\Pol_q(K)$ is given by \begin{equation}\label{EqInvFunct}\psi(x) = \int_{T} \Tr((\pi_t*\theta_{w_0})(x|b|_{\rho}^2))\rd t.\end{equation}\end{Theorem} 

It can be shown that \[\psi(1) =  \left(\prod_{\alpha\in \Delta^+} (1-q^{2(\rho,\alpha)})\right),\] with which $\psi$ can be rescaled to a state. This rescaling will however not be relevant for our purposes. Note also that we have adapted the statement of this theorem with respect to the conventions used in this paper. 

\begin{Theorem}\label{TheoExtIso} The $*$-representation \[\Fun_{q,\compact}^0(AK) \rightarrow \End(L^2_{\hol,q}(B)_0)\] extends to a normal $*$-isomorphism \[L^{\infty}_{q,0}(AK) \cong B(L^2_{\hol,q}(B)).\]
\end{Theorem} 

\begin{proof} As noted before, any $*$-representation of $\Fun_{q,\compact}^0(AK)$ must be bounded, hence the representation on $\Fun_{q,\compact}^0(AK)$ can be extended to a representation on $L^2_{\hol,q}(B)$. By Theorem \ref{TheoRY}, it then follows that, with \[\widetilde{\psi}^0: B(L^2_{\hol,q}(B))^+ \rightarrow [0,+\infty],\quad x \mapsto  \Tr(x|b|_{\rho}^2),\] we have $\Fun_{q,\compact}^0(AK) \subseteq \mathscr{M}_{\widetilde{\psi}^0}$ with 
\[\psi^0(x) = \widetilde{\psi}^0(x)\qquad x\in \Fun_{q,\compact}^0(AK).\] 

Now it is easily seen that $\Func_{\compact}(P)$ and the $u_{\omega}$ generate a copy of $B(l^2(P))$. Hence if $x\in B(L^2_{\hol,q}(B))$ commutes with $\Fun_{q,\compact}^0(AK)$, it must be of the form $x = y\otimes 1$ for $y\in B(L^2_{\hol,q}(N))$. However, since $y$ then commutes with all $|b|_{\omega}$, which have finite-dimensional joint eigenspaces, it follows that $y$ must leave the subspace $L^2_{\hol,q}(N)_0$ invariant. But as $y$ then commutes with all $x_r$, and the latter have $e_0^{\otimes M}$ as a highest weight vector, it follows easily that $y$ must be a scalar. We deduce that $\Fun_{q,\compact}^0(AK)$ must be $\sigma$-weakly dense in $B(L^2_{\hol,q}(B))$. Moreover, it is also immediate that $\Fun_{q,\compact}^0(AK)$ is invariant under the modular automorphism group of $\widetilde{\psi}^0$. 

From the above, we deduce that we can make a unique unitary \[L^2_{q,0}(AK) \rightarrow L^2(B(L^2_{\hol,q}(B))),\quad \Gamma^0(x) \mapsto \widetilde{\Gamma}^0(x),\qquad x\in \Fun_{q,\compact}^0(AK).\] As this unitary intertwines the $\Fun_{q,\compact}^0(AK)$-representations, we deduce that it induces a normal $*$-isomorphism \[L^{\infty}_{q,0}(AK) \cong B(L^2_{\hol,q}(B)).\] 
\end{proof} 

To end, let us concretely identify the dual right coaction $\Ad_{\gamma}$ of $\msR_{q}^+(B_{\R})$ on $L^{\infty}_{q,0}(AK)$ in terms of the Galois object $(U_q^0(\mfb_{\R}),\alpha)$. Let us first introduce the following notation.

\begin{Def} We define $(\msR_{q}^0(B_{\R}),\alpha) = (L^{\infty}_{q,0}(AK),\Ad_{\gamma})$. 
\end{Def}

Again, for $X\in U_q^+(\mfb_{\R})$ interpreted as an endomorphism on $L^2_{\hol,q}(B)_0$, we denote by $\mbX$ its closure as an unbounded operator on $L^2_{\hol,q}(B)$. Note that this closure exists as $X$ has an adjoint as an operator on the pre-Hilbert space $L^2_{\hol,q}(B)_0$. It is clear that $\mbX$ is then affiliated with $L^{\infty}_{q,0}(AK) = B(L^2_{\hol,q}(B))$, since this is true of \emph{any} unbounded operator on $L^2_{\hol,q}(B)$. 

\begin{Theorem}\label{TheoAdisAd} For $X \in U_q^+(\mfb_{\R})$, we have that \[\alpha(\mbX) = \textrm{ closure of }\alpha(X)\textrm{ on }L^2_{\hol,q}(B)_0 \otimes L^2_{q,+}(AK)_0.\]
\end{Theorem} 
\begin{proof} By definition, we have that \[\Ad_{\gamma}(x) = \Sigma \Vv(1\otimes x)\Vv^*\Sigma,\qquad x\in L^{\infty}_{q,0}(AK).\] However, as $\Vv\in B(L^2_{q,0}(AK),L^2_{q,+}(AK))\vNtimes L^{\infty}_{q,0}(AK)$, we can interpret $\Vv$ as an operator \[\Vv: L^2_{q,0}(AK)\otimes L^2_{\hol,q}(B) \rightarrow L^2_{q,+}(AK)\otimes L^2_{\hol,q}(B),\]\[\Gamma^0(x) \otimes \xi \mapsto (\Gamma^0\otimes \id)(\gamma(x)) \xi.\] Note now that this restricts to a unitary operator \[\Vv_0:L^2_{q,0}(AK)_0 \otimes L^2_{\hol,q}(B)_0 \rightarrow L^2_{q,+}(AK)_0\otimes  L^2_{\hol,q}(B)_0.\] As $L^2_{\hol,q}(B)_0$ satisfies the conditions for $\mathscr{M}$ in Theorem \ref{TheoAdjTran}, it follows that \[\Ad_{\gamma}(\mbX)_{\mid L^2_{\hol,q}(B)_0\otimes L^2_q(AK)_0} = (\id\otimes \pi_{\rhd})(\alpha(X)).\] But then $\Ad_{\gamma}(\mbX) = \Sigma \Vv(1\otimes \mbX)\Vv^*\Sigma$ must be the closure of the unbounded operator $(\id\otimes \pi_{\rhd})(\alpha(X)) = \Sigma \Vv_0 (1\otimes X)\Vv_0^*\Sigma$. 
\end{proof}

\begin{Cor} The left and right $\msR_q^+(B_{\R})$-invariant weights on $L^{\infty}_{q,0}(AK) \cong B(L^2_{\hol,q}(B))$ are given by \[\varphi_{\Ad}(x) = \Tr(x|b|_{\rho}^{-2}),\qquad \psi_{\Ad}(x) = \Tr(x|b|_{\rho}^{-2}\mbK_{-2\rho}^*\mbK_{-2\rho}).\]
\end{Cor} 
Note that $|b|_{\rho}^{-2}$ and $\mbK_{-2\rho}^*\mbK_{-2\rho}$ strongly commute as they act on different tensorands of $L^2_{\hol,q}(B) = L^2_{\hol,q}(H)\otimes L^2_{\hol,q}(N)$. There is thus no problem in interpreting the formula for $\psi_{\Ad}$. 

\begin{proof} The formula for $\varphi_{\Ad}$ follows immediately from Theorem \ref{TheoWeight} and the proof of Theorem \ref{TheoExtIso}. By \eqref{EqDelGrouplike}, the formula for $\psi_{\Ad}$ will follow if we can show that $\delta_0 = \mbK_{-2\rho}^*\mbK_{-2\rho}$ satisfies \[\Ad_{\gamma}(\delta_0) = \delta_0 \otimes \mbK_{-2\rho}^*\mbK_{-2\rho},\] where the right tensorand is the modular element in $\msR_q^+(B_{\R})$. However, this is immediate from Theorem \ref{TheoAdisAd}. 
\end{proof}

To end, let us strengthen Theorem \ref{TheoAdisAd}. We first prove the following lemma. 

\begin{Lem}\label{LemClosAd} For $X \in \{E_r,E_r^*,F_r,F_r^*,K_{\omega},L_{\omega}\}\subseteq U_q^0(\mfb_{\R})$, we have that $\mbX^*$ is the closure of $X^*$ on $L^2_{\hol,q}(B)_0$. 
\end{Lem} 

\begin{proof} In the notation of the proof of Corollary \ref{CorFaithPi}, we have on $L^2_{\hol,q}(B)_0 = L^2_{\hol,q}(H)_0\otimes L^2_{\hol,q}(N)_0$ that \[K_{\omega} = q^{-\frac{1}{2}(\omega,\omega)}Z_{\omega}U_{-\omega}\otimes 1,\quad L_{\omega} = q^{-\frac{1}{2}(\omega,\omega)}U_{\omega}Z_{\omega}\otimes 1,\]\[E_r = U_{-\alpha_r}\otimes X_r,\quad F_r = q_r^{-1} Z_{-\alpha_r} \otimes X_r^*.\] 

The statement in the lemma is clear for $K_{\omega}$ and $L_{\omega}$, and for $E_r,F_r,E_r^*,F_r^*$ it suffices to prove the corresponding statement for $X_r,X_r^*$ on $L^2_{\hol,q}(N)_0$. However, if $U_{q_r}^0(\C_{\R})$ is the $*$-algebra generated by $X_r, X_r^*$, it follows inductively from the fact that $L^2_{\hol,q}(N)_0 \cong U_q(\mfn)^*$ is graded by $-Q^+$ by Corollary \ref{CorGradHW}, with finite-dimensional weight spaces, that $L^2_{\hol,q}(N)_0$ decomposes as a direct sum of copies of the (unique) highest weight representation of $U_{q_r}^0(\C_{\R})$. 

It is thus sufficient to prove the statement for the case $\mfk = \mfsu(2)$. But it is easily seen that one can then identify $L^2_{\hol,q}(N)_0$ with $\C[\N]$ in such a way that the generator $X$ of $U_q^0(\mfn)$ acts by \[q^{1/2}(q^{-2}-1)Xe_n =  (q^{-2n}-1)^{1/2}e_n.\] For this $X$ it is immediately clear that the closure of $X^*$ is the adjoint of the closure of $X$, and similarly for $X^*$.  
\end{proof} 

For $X,Y$ closed unbounded operators, let us write $X\dotplus Y$ for the closure of the sum $X+Y$, and $X\vNtimes Y$ for the closure of the algebraic tensor product $X\otimes Y$. 

\begin{Prop}\label{PropAdConc} The following identities hold: \[\alpha(\mbK_{\omega}) = \mbK_{\omega}\vNtimes \mbK_{\omega},\quad \alpha(\mbK_{\omega}^*)  = \mbK_{\omega}^*\vNtimes \mbK_{\omega}^*,\]\[\alpha(\mbE_r) = \mbE_r\vNtimes \mbK_{\alpha_r} \dotplus 1\vNtimes \mbE_r,\quad \alpha(\mbE_r^*) = \mbE_r^*\vNtimes \mbK_{\alpha_r}^* \dotplus 1\vNtimes \mbE_r.\] 
\end{Prop}

\begin{proof} The statement is clear for $\mbK_{\omega}$ and $\mbK_{\omega}^*$. 

Let us prove the statement for $\alpha(\mbE_r)$. Since $L^2_{q,+}(AK)$ is a direct sum $\oplus_{\varpi\in P^+} L^2_{q,+}(AK)_{\varpi}$ of bounded representations $\pi_{\varpi}$ of $U_q^+(\mfb_{\R})$, it is sufficient to prove that \[(\id\otimes \pi_{\varpi})\alpha(\mbE_r) = \mbE_r\vNtimes \pi_{\varpi}(K_{\alpha_r}) + 1\otimes \pi_{\varpi}(E_r).\]  But the domain of the right hand side is the same as for $\mbE_r\vNtimes \pi_{\varpi}(K_{\alpha_r})$, which has $L^2_{\hol,q}(B)_0\otimes L^2_q(AK)_{\varpi,0}$ as a core. It follows that $\mbE_r\vNtimes \mbK_{\alpha_r} \dotplus 1\vNtimes \mbE_r$ is equal to the closure of $\alpha(E_r)$, from which the statement follows by \eqref{TheoAdisAd}.

The formula for $\alpha(\mbE_r^*)$ follows similarly, using Lemma \eqref{LemClosAd} and the fact that the adjoint of $\mbE_r\vNtimes \mbK_{\alpha_r} \dotplus 1\vNtimes \mbE_r$ is given by $\mbE_r^*\vNtimes \mbK_{\alpha_r}^* \dotplus 1\vNtimes \mbE_r^*$ by the  same reasoning as in the previous paragraph. 
\end{proof}

\begin{Prop}\label{PropAdGen} The operators $\mbK_{\omega},\mbL_{\omega},\mbE_r,\mbF_r$ generate $\msR_{q}^0(B_{\R})$. 
\end{Prop} 

\begin{proof} It is clear that the $\mbK_{\omega},\mbL_{\omega}$ generate $B(L^2_{\hol,q}(H))$. It is then sufficient to prove that the $\mbX_r$ generate $B(L^2_{\hol,q}(N))$. However, by considering the decomposition in terms of the rank $1$ case as in the proof of Lemma \ref{LemClosAd}, it follows that the von Neumann algebra generated by the $\mbX_r$ contains operators $C_{r}$ such that \[C_{r} \xi = q^{-(\check{\alpha}_r,\wt(\xi))}\xi,\qquad \xi \in L^2_q(N)_0,\] where $\wt$ is the weight on $L^2_{\hol,q}(N)_0$ inherited from $U_q(\mfn)^* \cong L^2_{\hol,q}(N)_0$ by Corollary \ref{CorGradHW} and Corollary \ref{CorFaithRep}. It follows in particular that the von Neumann algebra generated by the $\mbX_r$ will contain the projection onto the highest weight vector. But then any operator in the commutant must preserve the one-dimensional subspace spanned by the highest weight vector, and is hence a scalar as the highest weight vector is cyclic. 
\end{proof} 

Propositions \ref{PropAdConc} and \ref{PropAdGen} give a complete description of $(\msR_q^0(B_{\R}),\alpha)$ as generated by unbounded operators. 

\section{Conclusion and outlook}

In this article, we have constructed by means of Theorem \ref{TheoDualMain} a duality theory for von Neumann algebraic Galois objects whose underlying von Neumann algebra is a type $I$-factor. We then have shown in Theorem \ref{TheoAdisAd} how, for $\mfg$ a semisimple complex Lie algebra with Borel subalgebra $\mfb$ and compact form $\mfk$, a particular Hopf algebraic Galois object $U_q^0(\mfb_{\R})$ for a certain amplification $U_q^{+}(\mfb_{\R})$ of $U_q(\mfk)$ can be integrated into an operator algebraic $I$-factorial Galois object $\msR_q^0(B_{\R})$ for the operator algebraic integration $\msR_q^+(B_{\R})$ for $U_q^+(\mfb_{\R})$.

Let us sketch in the following how this theory can be further developed. In \cite{Sch96}, it is shown how a Hopf algebraic (right) Galois object $A$ for a Hopf algebra $H$ can be completed into a \emph{bi-Galois object}, meaning that there exists a (unique) new Hopf algebra $L$ with a (left) Galois structure on $A$, commuting with the coaction of $H$. For example, if $H$ is a Hopf algebra and $\omega,\chi$ two $2$-cocycles on $H$, then ${}_{\chi}H_{\chi^{-1}}$ is the new Hopf algebra constructed out of the right ${}_{\omega}H_{\omega^{-1}}$-Galois object ${}_{\chi}H_{\omega^{-1}}$. For the right Galois object $U_q^0(\mfb_{\R})$, this will lead to the Hopf $*$-algebra $U_q(\mfb_{\R})$ which is generated by a copy of $U_q(\mfb)$ and a copy of $U_q(\mfb^-)$ such that $E_r^* = F_rL_{\alpha_r}$, $K_{\omega}^* = L_{\omega}$ and such that the following interchange relations holds: 
\[K_{\omega}L_{\chi} =  L_{\chi}K_{\omega},\]
\[K_{\omega} F_r = q^{(\omega,\alpha_r)}F_rK_{\omega},\qquad L_{\omega}E_r = q^{-(\omega,\alpha_r)}E_rL_{\omega}\]\[\lbrack E_r,F_s\rbrack = 0.\]

On the other hand, in \cite{DeC11} it is shown how from an operator algebraic Galois object can also be constructed a new quantum group von Neumann algebra, providing an analytic version of Schauenburgs result. Given our Theorem \ref{TheoAdisAd}, we see that the new quantum group von Neumann algebra $\msR_q(B_{\R})$, constructed from the Galois object $(\msR_q^0(B_{\R}),\alpha)$, can be seen as the operator algebraic analogue of $U_q(\mfb_{\R})$. However, the precise relation between $U_q(\mfb_{\R})$ and $\msR_q(B_{\R})$ requires some further delicate analysis, which we leave for a future occasion.

\end{document}